\documentclass[11pt]{article}

\usepackage[margin=1.2in]{geometry}

\usepackage[T1]{fontenc}
\usepackage[utf8]{inputenc}
\usepackage[USenglish, UKenglish, english]{babel}
\usepackage{amssymb}
\usepackage{amsthm}
\usepackage{amsmath}
\usepackage[mathscr]{euscript}
\usepackage{enumitem}
\usepackage{graphicx}
\usepackage{MnSymbol}
 \let\mathscr\relax
\usepackage[scr]{rsfso}
\usepackage{tikz-cd}
\usepackage{tikz}
\usetikzlibrary{matrix,arrows,decorations.pathmorphing}
\usepackage{hyperref}
\usepackage{marginnote}
\usetikzlibrary{arrows.meta}

\usepackage{libertine}

\theoremstyle{definition}
\newtheorem{defin}{Definition}[section]
\theoremstyle{definition}

\theoremstyle{plain}
\newtheorem{theo}[defin]{Theorem}
\theoremstyle{plain}
\newtheorem{prop}[defin]{Proposition}
\theoremstyle{plain}
\newtheorem{lem}[defin]{Lemma}
\theoremstyle{plain}
\newtheorem{cor}[defin]{Corollary}
\theoremstyle{definition}
\newtheorem{rmk}[defin]{Remark}
\theoremstyle{definition}

\theoremstyle{definition}

\theoremstyle{plain}

\theoremstyle{plain}
\newtheorem{conj}[defin]{Conjecture}
\theoremstyle{definition}

\theoremstyle{definition}
\newtheorem{hyp}[defin]{Hypothesis}

\theoremstyle{definition}

\theoremstyle{definition}

\theoremstyle{definition}
\newtheorem*{defin*}{Definition}
\theoremstyle{definition}
\newtheorem*{ex*}{Example}
\theoremstyle{plain}
\newtheorem*{theo*}{Theorem}
\theoremstyle{plain}
\newtheorem*{prop*}{Proposition}
\theoremstyle{plain}
\newtheorem*{lem*}{Lemma}
\theoremstyle{plain}
\newtheorem*{cor*}{Corollary}
\theoremstyle{definition}
\newtheorem*{rmk*}{Remark}
\theoremstyle{definition}
\newtheorem*{exe*}{Exercise}

\theoremstyle{plain}
\newtheorem{theoA}{Theorem}

\theoremstyle{plain}

\theoremstyle{plain}

\theoremstyle{plain}

\numberwithin{equation}{section}

\usepackage{parskip}

\makeatletter
\def\thm@space@setup{%
  \thm@preskip=\parskip \thm@postskip=0pt
}
\makeatother

\usepackage{xparse,etoolbox}

\newcommand{\blocktheorem}[1]{%
  \csletcs{old#1}{#1}
  \csletcs{endold#1}{end#1}
  \RenewDocumentEnvironment{#1}{o}
    {\par\addvspace{1.5ex}
     \noindent\begin{minipage}{\textwidth}
     \IfNoValueTF{##1}
       {\csuse{old#1}}
       {\csuse{old#1}[##1]}}
    {\csuse{endold#1}
     \end{minipage}
     \par\addvspace{1.5ex}}
}

\blocktheorem{theo}
\blocktheorem{conj}
\blocktheorem{theoA}
\blocktheorem{conjA}

\setlist[enumerate]{label=(\roman*)}

\def\Bl{{\rm Bl}}

\def\bl{{\rm bl}}

\def\irr{{\rm Irr}}

\def\ker{{\rm Ker}}

\def\aut{{\rm Aut}}

\def\n{{\mathbf{N}}} 

\def\c{{\mathbf{C}}} 
 
\def\z{{\mathbf{Z}}} 

\def\O{{\mathbf{O}}} 

\def\F{{\mathbf{F}}} 

\def\C{{\mathcal{C}}} 

\def\D{{\mathcal{D}}} 

\def\Pr{{\mathcal{P}}} 

\def\Qr{{\mathcal{Q}}} 

\makeatletter
\newcommand{\uset}[3][0ex]{%
  \mathrel{\mathop{#3}\limits_{
    \vbox to#1{\kern-7\ex@
    \hbox{$\scriptstyle#2$}\vss}}}}
\makeatother

\newcommand{\wt}[1]{\widetilde{#1}} 

\newcommand{\wh}[1]{\widehat{#1}}

\newcommand{\ws}[1][1.5]{
  \mathrel{\scalebox{#1}[1]{$\sim$}}
}

\newcommand{\iso}[1]{\ws_{#1}}

\usepackage{sectsty}
\allsectionsfont{\centering}

\makeatletter
\def\blfootnote{\gdef\@thefnmark{}\@footnotetext}
\makeatother

\title{
{\huge\bf A reduction theorem for the \\Character Triple Conjecture}\\
\author{Damiano Rossi}
\date{}
\blfootnote{\emph{$2010$ Mathematical Subject Classification:} $20$C$20$.
\\
\emph{Key words and phrases:} Character Triple Conjecture, Dade's Conjecture, inductive conditions, local-global conjectures.
\\
This work initiated during the author's doctoral studies, conducted as a member of the \textit{GRK2240:Algebro-geometric Methods in Algebra, Arithmetic, and Topology} funded by the DFG, was then further developed at City University of London funded by the EPSRC grant EP/T$004592/1$, and completed at Loughborough University funded by the EPSRC grant EP/W$028794/1$. The author is grateful to Radha Kessar for useful conversations on the structure of nilpotent block algebras, to Gunter Malle for providing helpful corrections, to Josep Miquel Mart\'inez for proofreading parts of an earlier draft, and finally to Geoffrey Robinson for clarifying some of the arguments presented in his wonderful paper \cite{Rob02} and which inspired part of our work. 
}
}

\begin{document}

\renewcommand{\thetheoA}{\Alph{theoA}}

\selectlanguage{english}

\maketitle

\begin{abstract}
In this paper, we show that the Character Triple Conjecture holds for all finite groups once assumed for all quasi-simple groups. This answers the question on the existence of a self-reducing form of Dade's conjecture, a problem that was long investigated by Dade in the 1990s. Our result shows that this role is played by the Character Triple Conjecture, recently introduced by Sp\"ath, that we present here in a general form free of all previously imposed restrictions.
\end{abstract}

\tableofcontents

\section{Introduction}

In the paper \cite{Dad92}, Dade introduced a conjecture that predicts an exact formula expressing the number of irreducible characters with a fixed $p$-defect in a $p$-block of a finite group in terms of information encoded in the $p$-local structure of the group, where $p$ is a given prime number. This statement is the pinnacle of an important part of group representation theory: the study of the so-called local-global counting conjectures. These include the McKay Conjecture \cite{McK72}, the Alperin--McKay Conjecture \cite{Alp76}, and the Alperin Weight Conjecture \cite{Alp87} among others, and suggest that several numerical invariants of the $p$-modular representation theory of a finite group are determined by various $p$-local subgroups of the group. Dade's Conjecture provides a unifying statement in this context as it can be shown to imply all the above mentioned conjectures. We also mention that Dade's Conjecture is known to imply the celebrated Brauer's Height Zero Conjecture, the proof of which has recently been completed in \cite{MNSFT}. 

The form of Dade's Conjecture stated in the orginal paper \cite{Dad92} is, however, not amenable to inductive arguments and Clifford theoretic reductions. Morally, a purely enumerative statement does not carry enough structural information and cannot be used to infer the conjecture when assumed for given subgroups or quotients. To circumvent this obstacle, in \cite{Dad94} and \cite{Dad97} Dade introduced a series of increasingly refined statements with the aim of finding a self-reducing form of the conjecture, that is, a form of the conjecture that would hold for all finite groups once proven for all (quasi-)simple groups. A candidate for this role was proposed by Dade in \cite[5.8]{Dad97} and is known as Dade's Inductive Conjecture. Together with this statment, in \cite{Dad97} Dade also announced a final paper of his series containing a reduction theorem which unfortunately never appeared. Several years later, following the ground-breaking works \cite{Isa-Mal-Nav07} and \cite{Nav-Spa14I}, Sp\"ath introduced in \cite{Spa17} a new refinement of Dade's Conjecture known as the Character Triple Conjecture. The latter statement shares several features with Dade's Inductive Conjecture: both retain information about projective representations, twisted group algebras, Clifford extensions, and invariance under group automorphisms (we refer the reader to \cite[Proposition 6.4]{Spa17} for a comparison result). In addition, the main result of \cite{Spa17} shows that Dade's Conjecture holds for all finite groups if the Character Triple Conjecture holds for all quasi-simple groups. In this paper, we complete Dade's programme and identify the Character Triple Conjecture as the right candidate for a self-reducing form of Dade's Conjecture. More precisely, we prove the following theorem.

\begin{theoA}
\label{thm:Main reduction}
Let $G$ be a finite group and $p$ a prime number. If the Character Triple Conjecture holds at the prime $p$ for every covering group of every non-abelian simple group involved in $G$, then it holds for the group $G$ at the prime $p$.
\end{theoA}

Stating the Character Triple Conjecture requires a substantial amount of notation and is beyond the scope of this introduction. We refer instead the reader to Section \ref{sec:Conjectures}. What we care to remark here is that, in analogy with the work done by Dade in \cite[Section 17]{Dad94} for his conjecture, we now have to work with a generalisation of Sp\"ath's statement. Once again, this is necessary in order to perform the reduction. More precisely, we introduce a formulation of the Character Triple Conjecture (see Conjecture \ref{conj:CTC non-central Z}) that is free from any restriction on the choice of the $p$-chains and $p$-blocks under consideration. As we will see later, our version turns out to be equivalent to Sp\"ath's statement and is related to Robinson's Ordinary Weight Conjecture \cite{Rob96}, \cite{Rob04}. We also want to point out that, as it was the case for the introduction of Dade's Conjecture in \cite{Dad92} and the proof of Sp\"ath's reduction theorem in \cite{Spa17}, our result owes a great deal to the work of Robinson. In particular, as a step towards the proof of Theorem \ref{thm:Main reduction}, we prove a cancellation theorem for the Character Triple Conjecture (see Theorem \ref{thm:Minimal counterexample}) inspired by Robinson's paper \cite{Rob02}.

Our work can be related to recent results of a similar nature. As we have mentioned at the beginning of this introduction, Dade's Conjecture provides a unifying framework for the local-global counting conjectures and implies the McKay Conjecture, the Alperin--McKay Conjecture, and the Alperin Weight Conjecture. Following Dade's example, refinements for these simpler conjectures have been introduced throughout the years in order to obtain statements more amenable to reductions. For any of these conjectures, say Conjecture X, this refinement is often referred to as the inductive condition for Conjecture X (in analogy with Dade's Inductive Conjecture). These inductive conditions were then shown to imply the corresponding conjecture for all finite groups once assumed for all quasi-simple groups. This was done for the McKay Conjecture in \cite{Isa-Mal-Nav07}, for the Alperin--McKay Conjecture in \cite{Spa13II}, and for the Alperin Weight Conjecture in \cite{Nav-Tie11} and \cite{Spa13I}. All these results have since then been refined to show that, in fact, these inductive conditions are the right self-reducing forms of the conjectures. This was first achieved in \cite{Nav-Spa14I} for the inductive condition for the Alperin--McKay Conjecture and more recently in \cite{Ros-iMcK} and \cite{MRR} for the inductive conditions for the McKay Conjecture and the Alperin Weight Conjecture respectively. Beside confirming a phenomenon expected by Dade, the results of \cite{Nav-Spa14I}, \cite{Ros-iMcK}, and \cite{MRR} are fundamental for the final verification of these and other conjectures. For example, the reduction obtained in \cite{Nav-Spa14I} was used by Ruhstorfer in \cite{Ruh22AM} to prove that Brauer's Height Zero Conjecture holds for the prime $2$, a fundamental step towards the final proof in \cite{MNSFT}. Furthermore, the reduction given in \cite{Ros-iMcK} has recently been used in the verification of the inductive condition for the McKay Conjecture for simple groups of Lie type (in type D), hence contributing to the final proof of the McKay Conjecture \cite{Cab-Spa23}. Finally, the results of \cite{MRR} can be used to show that a recent conjecture of Navarro \cite{Nav17}, which unifies the Alperin Weight Conjecture and the Glauberman correspondence, can be reduced to quasi-simple groups. Similarly, we expect that Theorem \ref{thm:Main reduction} will play an active role in the final proof of Dade's Conjecture and that several other conclusions could be drawn from it. This will be investigated in a future work.

For what concerns the validity of the Character Triple Conjecture, we mention that it has been verified in \cite{Ros-CTC-max} for all finite groups with respect to maximal defect characters and the prime $2$ (using Ruhstorfer's results from \cite{Ruh22AM}) while the study of the conjecture for groups of Lie type in non-defining characteristic has been initiated by the author following ideas of Brou\'e, Fong, and Srinivasan (see the series of papers \cite{Ros24}, \cite{Ros-Clifford_automorphisms_HC}, \cite{Ros-Unip}, and \cite{Ros-Homotopy}). Furthermore, the conjecture has been verified for all sporadic simple groups with the exception of the Baby Monster and the Monster at the prime $2$ (see \cite[Theorem 9.2]{Spa17}).

The paper is organised as follows. In Section \ref{sec:Statement} we introduce the basic notation necessary to state the Character Triple Conjecture and its generalisation (see Conjecture \ref{conj:CTC non-central Z}). In Section \ref{sec:Nilpotent blocks}, we apply results of Puig--Zhou \cite{Pui-Zho12} and Marcus--Minuta \cite{Mar-Min21} on graded basic Morita equivalences, to prove the conjecture for extensions of nilpotent blocks (see Proposition \ref{prop:Auxiliary theorem}). In Section \ref{sec:Equivalence of conjectures} we give a first reduction of our formulation of the Character Triple Conjecture and show that this is logically equivalent to the statement introduced by Sp\"ath in \cite{Spa17}. Section \ref{sec:Cancellation} is devoted to a cancellation theorem inspired by the results of \cite{Rob02}. These results are finally applied in the final Section \ref{sec:Reduction} to prove Theorem \ref{thm:Main reduction}.

\section{The conjecture}
\label{sec:Statement}

In this section we introduce the main notation used throughout the paper and state a new form of the Character Triple Conjecture. More precisely, in our formulation we remove the restrictions considered previously: we consider blocks with arbitrary defect groups and $p$-chains starting with a $p$-subgroup that is not necessarily central. This is done by using relatively projective characters (see \cite{Rob87} and \cite{Kul-Rob87}).

\subsection{Notation}
\label{sec:Notation}

Let $p$ be a prime, which we assume fixed throughout the rest of the paper. All the modular representation theoretic notions are considered with respect to $p$. We freely use basic results and definitions from \cite{Nag-Tsu89}, \cite{Nav98}, \cite{Lin18I}, and \cite{Lin18II}. In particular, we consider a complete discrete valuation ring $\mathcal{O}$, with fraction field $\mathcal{K}$ of characteristic $0$ and residue field $k:=\mathcal{O}/J(\mathcal{O})$ of characteristic $p$. We assume that $k$ is algebraically closed and that $\mathcal{K}$ contains roots of unity of order $|H|$ for any finite group $H$ considered below. If $G$ is a finite group, we denote by $\Bl(G)$ the set of $p$-blocks $B$ of $G$. By abuse of notation we will indicate by $B$ both the primitive central idempotent in $\z(\mathcal{O}G)$ or the corresponding indecomposable direct factor of the $\mathcal{O}$-algebra $\mathcal{O}G$. However, we often specify the latter by writing $\mathcal{O}GB$. We denote by $\irr(B)$ the set of irreducible characters (over $\mathcal{K}$) that belongs to the block $B$.

Central to our study, is the notion of character triple. This is a triple $(G,N,\vartheta)$ where $N\unlhd G$ are finite groups and $\vartheta$ is an irreducible character of $N$ invariant under the action of $G$. Associated to each such triple is a projective representation $\Pr$ of $G$ and a factor set $\alpha$ satisfying certain properties. This datum determines a central extension $\wh{G}$ of $G$ in which the projective representation $\Pr$ affords an irreducible character that extends $\vartheta$ (seen as a character of a subgroup $\wh{N}$ of $\wh{G}$). We then say that $\wh{G}$ is the central extension associated to $\Pr$. We refer the reader to \cite[Chapter 11]{Isa76}, \cite[Chapter 3.5]{Nag-Tsu89}, \cite[Chapter 8]{Nav98}, \cite[Chapter 5 and 10]{Nav18} for a full account on projective representations. Character triples can be related by several notions of isomorphisms (see, for instance, \cite{Spa18}). Here, we consider the notion of $G$-block isomorphism of character triples, denoted by $\iso{G}$, as introduced in \cite[Definition 3.6]{Spa17}.

\subsection{A new version of the Character Triple Conjecture}
\label{sec:Conjectures}

We now come to the statement of the Character Triple Conjecture that was first introduced in \cite{Spa17}. In this section, we consider a statement that encompasses the original formulation in several ways. Firstly, while the original formulation only allows to work with $p$-chains starting with a central $p$-subgroup, our version is compatible with $p$-chains starting at an arbitrary normal (but not necessarily central) $p$-subgroup. Secondly, while the original formulation is only valid for $p$-blocks with non-central defect groups, the statement introduced below allows to consider all $p$-blocks and to remove this unnecessary restriction. Finally, in our formulation we replace the $G$-isomorphism of character triples considered in \cite{Spa17} with an a priori weaker condition that is more easily verifiable in the case of groups of Lie type in non-defining characteristic (see, for instance, (6.5-6.6) in the proof of \cite[Theorem 6.13]{Ros24} as well as \cite[Remark 5.8 and Theorem 5.10]{Ros-Unip}, in conjunction with \cite[Theorem F]{Ros-Homotopy}). This last feature can be seen as an adaptation of the introduction of the intermediate group considered in the formulation of the \textit{inductive McKay condition} (see, for instance, the group $T$ considered in \cite[p.70 (1)]{Isa-Mal-Nav07} or the group $M$ considered in \cite[Conjecture A]{Ros-iMcK}). Beside the gained generality, the flexibility inherent in our formulation will play a fundamental role in the proof of our main result and, as explained above, in its future verification. We point out however that, as we will see later (see Corollary \ref{cor:Equivalent conjectures}), our formulation turns out to be equivalent to that introduced in \cite{Spa17}.

For a normal $p$-subgroup $U$ of $G$, we denote by $\mathfrak{P}(G,U)$ the set of \emph{$p$-chains} of $G$ starting with $U$. These are the chains $\sigma=\{D_0<D_1<\dots<D_n\}$ where each $D_i$ is $p$-subgroup of $G$ and $D_0=U$. We write $D(\sigma):=D_n$ to indicate the last term of the chain $\sigma$ and $|\sigma|:=n$ to indicate the \emph{length} of $\sigma$. Using this notion we obtain a natural partition of the set $\mathfrak{P}(G,U)$ into two subsets: the set $\mathfrak{P}(G,U)_+$ of $p$-chains of even length and the set $\mathfrak{P}(G,U)_-$ of $p$-chains of odd length. The conjugation action of $G$ induces an action on $\mathfrak{P}(G,U)$ and we denote by $G_\sigma=\bigcap_i\n_G(D_i)$ the stabiliser in $G$ of the chain $\sigma$ and by $\mathfrak{P}(G,U)/G$ the corresponding set of $G$-orbits. This action preserves the length of $p$-chains and hence restricts to the sets $\mathfrak{P}(G,U)_+$ and $\mathfrak{P}(G,U)_-$. In this paper, we will mainly work with the subset of $\mathfrak{P}(G,U)$ consisting of \emph{normal $p$-chains}, that is, the subset $\mathfrak{N}(G,U)$ of chains $\sigma$ satisfying $D(\sigma)\leq G_\sigma$.

Next, recall that according to \cite[Lemma 3.2]{Kno-Rob89}, for every $p$-chain $\sigma$ and every block $b$ of $G_\sigma$, it is well defined the block $b^G$ of $G$ obtained by Brauer induction. In this case, we can find defect groups $Q$ and $D$, of $b$ and $b^G$ respectively, such that $Q\leq D$ (see \cite[Lemma 4.13]{Nav98}). If $\sigma\in\mathfrak{P}(G,U)$, then $U\leq \O_p(G_\sigma)$ and therefore $U\leq Q\leq D$ according to \cite[Theorem 4.8]{Nav98}. Similarly, if $\sigma\in\mathfrak{N}(G,U)$, then we have $D(\sigma)\leq \O_p(G_\sigma)\leq Q\leq D$. This basic observation will be used without further reference. Now, for every block $B$ of $G$, we denote by $\irr(B_\sigma)$ the set of irreducible characters $\vartheta$ of $G_\sigma$ satisfying $\bl(\vartheta)^G=B$, where $\bl(\vartheta)$ denotes the unique block of $G_\sigma$ to which $\vartheta$ belongs. Equivalently, $\irr(B_\sigma)$ is the set of irreducible characters belonging to the union of blocks $B_\sigma$ of $G_\sigma$ in Brauer correspondence with $B$.

We now recall the notion of relative $p$-defect with respect to a normal subgroup. Let $N\unlhd G$ and consider an irreducible character $\chi$ of $G$ lying above an irreducible character $\vartheta$ of $N$. By elementary Clifford theory it follows that $\vartheta(1)$ divides $\chi(1)$ and that $\chi(1)/\vartheta(1)$ divides $|G:N|$. Then, we define the \emph{$N$-relative $p$-defect} of $\chi$ to be the non-negative integer $d_N(\chi)$ such that
\[p^{d_N(\chi)}=\dfrac{|G:N|_p\vartheta(1)_p}{\chi(1)_p}.\]
By Clifford's theorem this definition only depends on $\chi$ and $N$ but not on the choice of the constituent $\vartheta$. When $N=1$ we recover the usual definition of \emph{$p$-defect} of $\chi$ which we denote simply by $d(\chi)$. Given a non-negative integer $d$ and a block $B$, we denote by $\irr^d(G)$ the set of characters $\chi$ of $G$ such that $d(\chi)=d$ and by $\irr^d(B)$ the subset of those characters belonging to $B$.

Characters of $N$-relative defect zero have been considered in \cite[Definition 17.1]{Dad94} and are often referred to as \emph{$N$-relative weights}. In this paper, we denote by by $\omega(G,N)$ the set of $N$-relative weights of $G$ while we use the symbol $\omega(G,N)^{\rm c}$ to indicate its complement in the set $\irr(G)$. More precisely, $\chi\in\irr(G)$ belongs to the set $\omega(G,N)^{\rm c}$ if and only if $d_N(\chi)>0$. As shown in \cite[Lemma 4.4]{Rob96} an irreducible character $\chi$ has $N$-relative defect zero if and only if it is afforded by an $N$-projective module.

\begin{lem}
\label{lem:N-projective and N-defect zero}
Let $N\unlhd G$ be finite groups and consider an irreducible character $\chi\in\irr(G)$. Then $d_N(\chi)=0$ if and only if $\chi$ is afforded by an $N$-projective module of $G$.
\end{lem}

\begin{proof}
This is \cite[Lemma 4.4]{Rob96}.
\end{proof}

We can now introduced the set of characters that will be considered in our formulation of the Character Triple Conjecture.

\begin{defin}
\label{def:C^d(B) non-central}
Let $G$ be a finite group, $U$ a normal $p$-subgroup of $G$, $B$ a $p$-block of $G$ and $d$ a non-negative integer. For every $p$-chain $\sigma\in\mathfrak{P}(G,U)$ we define the set of characters
\[\omega^d(B_\sigma,U)^{\rm c}:=\omega(G_\sigma,U)^{\rm c}\cap \irr^d(B_\sigma),\]
that is, the set of those characters $\vartheta\in\irr(G_\sigma)$ such that $d(\vartheta)=d$, $d_U(\vartheta)>0$ and $\bl(\vartheta)^G=B$. Then, we define the set of pairs
\[\C^d(B,U):=\left\lbrace (\sigma,\vartheta)\enspace\middle|\enspace \sigma\in\mathfrak{P}(G,U), \vartheta\in\irr^d(B_\sigma)\right\rbrace\]
and its subset
\[\underline{\C}^d(B,U):=\left\lbrace (\sigma,\vartheta)\enspace\middle|\enspace \sigma\in\mathfrak{P}(G,U), \vartheta\in\omega^d(B_\sigma,U)^{\rm c}\right\rbrace.\]
Furthermore, for any fixed irreducible character $\varphi$ of $U$ we denote by $\C^d(B,U,\varphi)$ the subset of $\C^d(B,U)$ consisting of those pairs $(\sigma,\vartheta)$ such that $\vartheta\in\irr^d(B_\sigma,\varphi):=\irr^d(B_\sigma)\cap \irr(G_\sigma\mid \varphi)$, that is, $\vartheta$ lies above $\varphi$. Similarly, we define $\underline{\C}^d(B,U,\varphi)$ to be the subset of $\underline{\C}^d(B,U)$ consisting of pairs $(\sigma,\vartheta)$ such that $\vartheta\in\omega^d(B_\sigma,U,\varphi)^{\rm c}:=\omega^d(B_\sigma,U)^{\rm c}\cap \irr(G_\sigma\mid \varphi)$.  
\end{defin}

Consider now the partition of $\C^d(B,U)$ induced by the length of $p$-chains and given by the subsets $\C^d(B,U)_+$ and $\C^d(B,U)_-$ consisting of those pairs $(\sigma,\vartheta)$ such that $\sigma$ lies in $\mathfrak{P}(G,U)_+$ and $\mathfrak{P}(G,U)_-$ respectively. Furthermore, the conjugation action of $G$ naturally induces an action on $\C^d(B,U)$ preserving the length of chains and, for $\epsilon\in\{+,-\}$, we denote by $\C^d(B,U)_\epsilon/G$ the corresponding set of $G$-orbits and by $\overline{(\sigma,\vartheta)}$ the $G$-orbit of any $(\sigma,\vartheta)\in\C^d(B,U)_\epsilon$. If in addition we consider $\varphi\in\irr(U)$ and $\C^d(B,U,\varphi)$, then the character $\varphi$ might not be $G$-invariant so that $G$ might not act on $\C^d(B,U,\varphi)$. In this case, we consider the action of $G_\varphi$ and the above notation are considered with respect to the group $G_\varphi$. Observe that when $U$ is contained in the centre of $G$, then $G_\varphi=G$ for any $\varphi\in\irr(U)$. Analogous definitions and remarks apply to the subsets $\underline{\C}^d(B,U)$ and $\underline{\C}^d(B,U, \varphi)$. Our formulation of the Character Triple Conjecture can now be stated as follows. Recall once again that we denote by $\iso{N}$ the $N$-block isomorphism of character triples from \cite[Definition 3.6]{Spa17}.

\begin{conj}[Character Triple Conjecture]
\label{conj:CTC non-central Z}
Let $G\unlhd A$ be finite groups, $U\unlhd G$ a $p$-subgroup and consider a $p$-block $B$ of $G$. Then, for every $d\geq 0$ and $\varphi\in\irr(U)$, there exists an $\n_A(U)_{B,\varphi}$-equivariant bijection
\[\Omega:\underline{\C}^d\left(B,U,\varphi\right)_+/G_\varphi\to\underline{\C}^d(B,U,\varphi)_-/G_\varphi\]
such that
\[\left(A_{\sigma,\vartheta},G_\sigma,\vartheta\right)\iso{M}\left(A_{\rho,\chi},G_\rho,\chi\right)\]
for every $(\sigma,\vartheta)\in\underline{\C}^d(B,U,\varphi)_+$, some $(\rho,\chi)\in\Omega(\overline{(\sigma,\vartheta)})$, and some $\langle G_\sigma, G_\rho\rangle\leq M\leq G$.
\end{conj}

As mentioned at the beginning of this section, Conjecture \ref{conj:CTC non-central Z} presents several advantages over the original formulation given in \cite[Conjecture 6.3]{Spa17} (see also \cite[Conjecture 2.2]{Ros22}). On one hand, the possibility of working with $p$-chains starting with non-central $p$-subgroups and the lack of restrictions on defect groups offer the necessary flexibility to prove our reduction to quasi-simple groups (see Section \ref{sec:Minimal counterexample proof}). Notice that, in the setting of Dade's Conjecture, an analogous more general formulation has been considered in \cite[Conjecture 17.10]{Dad94}, \cite[Section 4]{Rob96}, and \cite{Rob04} in relation to Robinson's Ordinary Weight Conjecture (see also \cite{Eat04}). On the other hand, the introduction of the intermediate group $\langle G_\sigma,G_\rho\rangle\leq M\leq G$ in Conjecture \ref{conj:CTC non-central Z} simplifies the verification for quasi-simple groups of Lie type in non-defining characteristic. The need for this latter adjustment can be found in recent work of the author on the verification of the conjecture for this class of groups (see the series of papers \cite{Ros24}, \cite{Ros-Clifford_automorphisms_HC}, \cite{Ros-Unip} and \cite{Ros-Homotopy}).

Before proceeding further, we show that the sets $\C^d(B,U)$ and $\underline{\C}^d(B,U)$ (and similarly the sets $\C^d(B,U,\varphi)$ and $\underline{\C}^d(B,U,\varphi)$ with $\varphi\in\irr(U)$) considered in Definition \ref{def:C^d(B) non-central} coincide whenever the normal $p$-subgroup $U$ is central in $G$ and strictly contained in the defect groups of $B$. This is an immediate consequence of the the following lemma, whose proof is essentially due to Robinson.

\begin{lem}
\label{lem:CTC non-central implies CTC}
Let $U$ be a normal $p$-subgroup of $G$ and $B$ a block of $G$ with defect group $D$. If $\omega(B_\sigma,U)$ is non-empty for some $\sigma\in\mathfrak{N}(G,U)$, then $\c_D(U)\leq U$. In particular, if $U$ is strictly contained in $D$ and is central in $G$, then $\C^d(B,U)=\underline{\C}^d(B,U)$ for every $d\geq 0$.
\end{lem}

\begin{proof}
Suppose that $\sigma$ and $\vartheta\in\irr(B_\sigma)$ are given such that $d_U(\vartheta)=0$. By Lemma \ref{lem:N-projective and N-defect zero} we know that $\vartheta$ is $U$-projective and hence $\c_{Q}(U)\leq U$, for every defect group $Q$ of $\bl(\vartheta)$, according to \cite[Proposition 3.2]{Rob96}. Then, arguing as in \cite[p.209]{Rob02} we can find a defect group $D$ of $B$ such that $D\cap G_\sigma=Q$ and show that the inclusion $\c_{Q}(U)\leq U$ actually implies $\c_D(U)\leq U$. Since by assumption $U\leq \z(G)$, we deduce that $D=U$. This contradiction proves the result.
\end{proof}

Using the above lemma, we can show that Conjecture \ref{conj:CTC non-central Z} implies the statement of \cite[Conjecture 6.3]{Spa17}. We include Sp\"ath's formulation below for completeness and for future reference.

\begin{conj}[Character Triple Conjecture as in {{\cite[Conjecture 6.3]{Spa17}}}]
\label{conj:CTC}
Let $G\unlhd A$ be finite groups, $Z\leq \z(G)$ a $p$-subgroup and consider a $p$-block $B$ of $G$ whose defect groups strictly contain $Z$. Then, for every $d\geq 0$, there exists an $\n_A(Z)_B$-equivariant bijection
\[\Omega:\C^d(B,Z)_+/G \to \C^d(B,Z)_-/G\]
such that 
\[\left(A_{\sigma,\vartheta},G_\sigma,\vartheta\right)\iso{G}\left(A_{\rho,\chi},G_{\rho},\chi\right),\]
for every $(\sigma,\vartheta)\in\C^d(B,Z)_+$ and some $(\rho,\chi)\in\Omega(\overline{(\sigma,\vartheta)})$.
\end{conj}

We now show how to recover Conjecture \ref{conj:CTC} from Conjecture \ref{conj:CTC non-central Z}.

\begin{prop}
\label{prop:CTC non-central implies CTC}
Let $G\unlhd A$ be finite groups and consider a $p$-subgroup $Z\leq \z(G)$, a $p$-block $B$ of $G$ whose defect groups strictly contain $Z$, and a non-negative integer $d$. Suppose that Conjecture \ref{conj:CTC non-central Z} holds at the prime $p$ for the choices $G\unlhd A$, $U:=Z$, $B$, $d$, and every character $\varphi\in\irr(Z)$. Then, Conjecture \ref{conj:CTC} holds at the prime $p$ for $G\unlhd A$, $Z$, $B$, and $d$.
\end{prop}

\begin{proof}
By assumption, for every $\varphi\in\irr(Z)$, we have $G_\varphi=G$ and we obtain a bijection $\Omega_\varphi:\underline{\C}^d(B,Z,\varphi)_+/G\to \underline{\C}^d(B,Z,\varphi)_-/G$. Combining the bijections $\Omega_\varphi$, for $\varphi$ running over a representative set for the action of $\n_A(Z)_B$ on $\irr(Z)$, we can construct an $\n_A(Z)_B$-equivariant bijection $\Omega:\underline{\C}^d(B,Z)_+/G\to \underline{\C}^d(B,Z)_-/G$. Furthermore, by applying Lemma \ref{lem:CTC non-central implies CTC} we deduce that $\underline{\C}^d(B,Z)_\pm$ coincides with $\C^d(B,Z)_\pm$. Suppose now that the $G$-orbits of the pairs $(\sigma,\vartheta)$ and $(\rho,\chi)$ correspond under the map $\Omega$. By Conjecture \ref{conj:CTC non-central Z}, and replacing if necessary $(\rho,\chi)$ with a $G$-conjugate, we can find a subgroup $\langle G_\sigma,G_\rho\rangle\leq M\leq G$ such that
\begin{equation}
\label{eq:CTC non-central implies CTC, 1}
\left(A_{\sigma,\vartheta},G_\sigma,\vartheta\right)\iso{M}\left(A_{\rho,\chi},G_\rho,\chi\right)
\end{equation}
and it remains to show that \eqref{eq:CTC non-central implies CTC, 1} implies the stronger relation
\begin{equation}
\label{eq:CTC non-central implies CTC, 2}
\left(A_{\sigma,\vartheta},G_\sigma,\vartheta\right)\iso{G}\left(A_{\rho,\chi},G_\rho,\chi\right).
\end{equation}
By \cite[Proposition 6.10]{Spa17}, and its proof, it is no loss of generality to assume that $\sigma$ and $\rho$ are normal $p$-chains. Consider the final term $D(\sigma)$ of the $p$-chain $\sigma$ and notice that, as $\sigma$ is a normal $p$-chain, we have $D(\sigma)\leq \O_p(G_\sigma)$. If $D$ is a defect group of the unique block of $G_\sigma$ containing $\vartheta$, then \cite[Theorem 4.8]{Nav98} shows that $D(\sigma)\leq D$ and therefore $\c_{A_{\sigma,\vartheta}G}(D)$ centralises $D(\sigma)$ and hence normalises all the terms of $\sigma$. Therefore we get
\[\c_{A_{\sigma,\vartheta}G}(D)\leq A_{\sigma,\vartheta}G_\sigma\leq A_{\sigma,\vartheta}\leq A_{\sigma,\vartheta}M.\]
Applying the same argument to the normal $p$-chain $\rho$ and a defect group $Q$ of the unique block of $G_\rho$ containing $\chi$, we obtain $\c_{A_{\rho,\chi}G}(Q)\leq A_{\rho,\chi}M$. We can now apply \cite[Lemma 2.11]{Ros22} to the $M$-block isomorphism \eqref{eq:CTC non-central implies CTC, 1} in order to get \eqref{eq:CTC non-central implies CTC, 2}. This completes the proof.
\end{proof}

Although Conjecture \ref{conj:CTC non-central Z} is, at a first glance, stronger than Conjecture \ref{conj:CTC}, we will prove in Theorem \ref{thm:CTC equivalent to CTC non-central} that the two conjectures are in fact equivalent. More precisely, we show that when proving Conjecture \ref{conj:CTC non-central Z} it is enough to consider the case in which $U$ is contained in the centre of $G$. On the other hand, in the case of quasi-simple groups, the $p$-subgroup $U$ is automatically central, in which case the (finer) bijection required in Conjecture \ref{conj:CTC non-central Z} can be recovered from the one given in Conjecture \ref{conj:CTC} thanks to well-known properties of isomorphisms of character triples. Furthermore, in this case, it is no loss of generality to assume that $U$ is strictly contained in the defect groups of the block under consideration. As a consequence, the verification of Conjecture \ref{conj:CTC non-central Z} for quasi-simple groups is no harder than that of Conjecture \ref{conj:CTC} and, thanks to the more relaxed condition required on character triples, actually simpler. We make this precise in the following remark.

\begin{rmk}
\label{rmk:CTC non-central vs CTC for quasi-simple}
Let $G$ be a quasi-simple group and consider a normal $p$-subgroup $U$ of $G$ and a block $B$ of $G$ with defect group $D$. Since $G/\z(G)$ is simple, it follows that $U\leq \z(G)$. Suppose first that $U=D$. Let $(\sigma,\vartheta)\in\underline{\C}^d(B,U)_\pm$ for some non-negative integer $d$. Since $D=U\leq \O_p(G_\sigma)$ is contained in every defect group of $\bl(\vartheta)$, it follows from \cite[Lemma 4.13]{Nav98} that $D$ is a defect group of $\bl(\vartheta)$. Then, noticing that $D$ is central in $G_\sigma$, \cite[Theorem 9.12]{Nav98} implies that $\vartheta(1)_p=|G_\sigma:D|$ and hence $d_U(\vartheta)=0$, a contradiction (see Definition \ref{def:C^d(B) non-central}). This shows that when $U=D$ the set $\underline{\C}^d(B,U)_\pm$ is empty. We can therefore assume that $U$ is strictly contained in $D$. Now, assume to be given an $\n_A(U)_B$-equivariant bijection $\Omega:\C^d(B,U)_+/G\to \C^d(B,U)_-/G$ (as in Conjecture \ref{conj:CTC}). If $\overline{(\rho,\chi)}=\Omega(\overline{(\sigma,\vartheta)})$ for some $(\sigma,\vartheta)\in\C^d(B,U)_+$ and $(\rho,\chi)\in\C^d(B,U)_-$, then the restrictions of $\vartheta$ and $\chi$ to the central subgroup $U$ have the same irreducible constituent thanks to \cite[Lemma 3.4]{Spa17}. In particular, for every $\varphi\in\irr(U)$, we deduce that $(\sigma,\vartheta)\in\C^d(B,U,\varphi)_+$ if and only if $(\rho,\chi)\in\C^d(B,U,\varphi)_-$ and hence $\Omega$ restricts to an $\n_A(U)_{B,\varphi}$-equivariant bijection $\Omega_\varphi:\C^d(B,U,\varphi)_+/G\to\C^d(B,U,\varphi)_-/G$. Notice here that $G=G_\varphi$ because $U\leq \z(G)$. Finally, recalling that $U<D$, we conclude that $\C^d(B,U,\varphi)_\pm=\underline{\C}^d(B,U,\varphi)_\pm$ by Lemma \ref{lem:CTC non-central implies CTC} .
\end{rmk}

We conclude this section with a discussion related to the statement of the \textit{inductive condition for Dade's Conjecture} introduced in \cite[Definition 6.7]{Spa17}. In fact, it is possible to strengthen the isomorphism of character triples required in the above statements so that it is inherited by central quotients. This is sometimes useful if we want to deal with all covering groups of a given simple group at once. We include this condition below noticing that, for the purpose of dealing with quasi-simple groups, it is no loss of generality to consider the setting of Conjecture \ref{conj:CTC} thanks to Remark \ref{rmk:CTC non-central vs CTC for quasi-simple}. We recall that the isomorphisms of character triples required below is stronger than the one given by Conjecture \ref{conj:CTC} thanks to \cite[Corollary 4.4]{Spa17}.

\begin{conj}
\label{conj:CTC+}
In the situation described in Conjecture \ref{conj:CTC}, there exists a bijection $\Omega$ such that $\ker(\vartheta_{\z(G)})=\ker(\chi_{\z(G)})=:W$ and 
\[\left(A_{\sigma,\vartheta}/W,G_\sigma/W,\overline{\vartheta}\right)\iso{G/W}\left(A_{\rho,\chi}/W,G_{\rho}/W,\overline{\chi}\right),\]
for every $(\sigma,\vartheta)\in\C^d(B,Z)_+$ and $(\rho,\chi)\in\Omega(\overline{(\sigma,\vartheta)})$ and where $\overline{\vartheta}$ and $\overline{\chi}$ correspond to $\vartheta$ and $\chi$ via inflation of characters.
\end{conj}

Recall that for every perfect group $G$, we say that $X$ is a \emph{covering group} of $G$ if $X$ is a central extension of $G$ and $X$ is perfect. The next lemma describes the relation between Conjecture \ref{conj:CTC} and Conjecture \ref{conj:CTC+} for covering groups of perfect groups. This follows immediately from the argument used to prove \cite[Proposition 6.8]{Spa17}.

\begin{lem}
\label{lem:CTC vs CTC+}
Let $K$ be a perfect group. Then the following are equivalent:
\begin{enumerate}
\item Conjecture \ref{conj:CTC} holds at the prime $p$ for every covering group of $K$;
\item Conjecture \ref{conj:CTC+} holds at the prime $p$ for the universal covering group of $K$.
\end{enumerate}
\end{lem}

\section{Extensions of nilpotent blocks}
\label{sec:Nilpotent blocks}

In this section, we show that the Character Triple Conjecture holds for every block covering a nilpotent block whose defect groups strictly contain the starting term of the $p$-chains considered. This fact will be useful in our study of a minimal counterexample to the conjecture and, in particular, to prove our cancellation theorem for chains (see Theorem \ref{thm:Minimal counterexample}). To prove this result we will make use of certain character bijections induced by basic Morita equivalences between (extensions of) a nilpotent block and its Brauer correspondent.

\subsection{Character bijections between nilpotent Brauer correspondent blocks}

Basic Morita equivalences, first introduced by Puig in \cite[Section 7]{Pui99}, are certain Morita equivalences compatible with the $p$-local structure of blocks. For instance, it can be shown that basic Morita equivalent blocks have isomorphic defect groups and (saturated) fusion systems. The prototypical example of a basic Morita equivalence is given by considering a nilpotent block, or more generally an inertial block, and its Brauer correspondent. For this, let $\mathcal{O}$, $\mathcal{K}$, and $k=\mathcal{O}/J(\mathcal{O})$ be as in Section \ref{sec:Notation} and consider a nilpotent block idempotent $b$ of $\mathcal{O}G$ with defect group $Q$ and Brauer correspondent $b_0$ of $\mathcal{O}\n_G(Q)$. In this case, it follows from the results of \cite{Pui88} that there is a basic Morita equivalence between $\mathcal{O}Gb$ and $\mathcal{O}\n_G(Q)b_0$.

Suppose now that $G$ is embedded as a normal subgroup in a finite group $A$ and assume that $b$ is $A$-invariant. The results of \cite{Pui88} were later extended to the block extension $\mathcal{O}Ab$ by K\"ulshammer--Puig \cite{Kul-Pui90} and by Puig--Zhou \cite{Pui-Zho12}. Notice that in this case, the algebra $\mathcal{O}Ab$ admits a grading with respect to the quotient group $\overline{A}:=A/G$. Moreover, if as before $b_0$ denotes the Brauer correspondent of $b$ in $\mathcal{O}\n_G(Q)$, then the block $b_0$ is $\n_A(Q)$-invariant and a Frattini argument shows that $A=\n_A(Q)G$. In this case, there is a canonical isomorphism $\overline{A}\simeq \n_A(Q)/\n_G(Q)$ through which $\mathcal{O}\n_A(Q)$ can be seen as an $\overline{A}$-graded algebra. The notion of basic Morita equivalence was generalised to block extensions by Coconet--Marcus in \cite{Coc-Mar17} where group graded basic Morita equivalences were introduced. The results of Puig and Zhou \cite[Theorem 3.14 and Corollary 3.15]{Pui-Zho12} (see also \cite{Coc-Mar-Tod20}) then show that $\mathcal{O}Ab$ and $\mathcal{O}\n_A(Q)b_0$ are $\overline{A}$-graded basic Morita equivalent. Analogous partial results were obtained by Zhou in \cite{Zho15}, \cite{Zho16}, and \cite{Zho21} in the more general case of inertial blocks.

\begin{theo}
\label{thm:Graded basic Morita equivalences for nilpotent blocks}
Let $b$ be a block of $\mathcal{O}G$ with defect group $Q$ and consider its Brauer correspondent $b_0$ in $\mathcal{O}\n_G(Q)$. Assume that $G\unlhd A$, set $\overline{A}:=A/G$, and suppose that $b$ is nilpotent and $A$-invariant. Then there exists an $\overline{A}$-graded basic Morita equivalence between $\mathcal{O}Ab$ and $\mathcal{O}\n_A(Q)b_0$.
\end{theo}

\begin{proof}
This follows from \cite[Theorem 3.14 and Corollary 3.15]{Pui-Zho12}.
\end{proof}

In the papers \cite{Mar-Min21-central} and \cite{Mar-Min21} group graded equivalences were shown to imply character bijections such that the order relations $\geq_c$ and $\geq_b$ respectively (see \cite[Definition 10.14]{Nav18} and \cite[Definition 2.7 and Definition 4.2]{Spa18}) hold between corresponding character triples. Beside giving a structural explanation for the validity of these Clifford-theoretic and cohomological conditions, these results allow us to exploit known graded equivalences to construct well behaved character bijections.

\begin{cor}
\label{cor:Graded basic Morita equivalences for nilpotent blocks}
Let $b$ be a block of $G$ with defect group $Q$ and consider its Brauer correspondent $b_0$ in $\n_G(Q)$. If $b$ is nilpotent and $G\unlhd A$, then there exists a defect preserving $\n_A(Q)_b$-equivariant bijection
\[\Psi:\irr(b)\to \irr(b_0)\]
such that
\[\left(A_\vartheta,G,\vartheta\right)\iso{G}\left(\n_A(Q)_\vartheta,\n_G(Q),\Psi(\vartheta)\right)\]
for every $\vartheta\in\irr(b)$.
\end{cor}

\begin{proof}
Since $A_\vartheta\leq A_b$ for every $\vartheta\in\irr(b)$ it is no loss of generality to assume that $b$ is $A$-invariant. Set $\overline{A}:=A/G$, $\mathcal{A}:=\mathcal{O}Ab$ and $\mathcal{A}':=\mathcal{O}\n_A(Q)b_0$. By Theorem \ref{thm:Graded basic Morita equivalences for nilpotent blocks} there exists an $\overline{A}$-graded $(\mathcal{A},\mathcal{A}')$-bimodule $M$ inducing an $\overline{A}$-graded basic Morita equivalence between $\mathcal{A}$ and $\mathcal{A}'$. Consider the identity components $\mathcal{B}:=\mathcal{A}_1=\mathcal{O}Gb$ and $\mathcal{B}':=\mathcal{A}'_1=\mathcal{O}\n_G(Q)b_0$ and notice that $M_1$ induces a basic Morita equivalence between $\mathcal{B}$ and $\mathcal{B}'$. Let $\Psi:\irr(b)\to\irr(b_0)$ be the character bijection induced by $\mathcal{K}\otimes_{\mathcal{O}}M_1$ and observe that $\Psi$ preserves the defect of characters. Furthermore, since the bimodule $M_1$ is $\overline{A}$-invariant and recalling that $\overline{A}\simeq\n_A(Q)/\n_G(Q)$, it follows that $\Psi$ is $\n_A(Q)$-equivariant. Next, fix $\vartheta\in\irr(b)$. We have to show that
\begin{equation}
\label{eq:Graded basic Morita equivalences for nilpotent blocks, 1}
\left(A_\vartheta,G,\vartheta\right)\iso{G}\left(\n_A(Q)_\vartheta,\n_G(Q),\Psi(\vartheta)\right).
\end{equation}
Set $\overline{A}_\vartheta:=A_\vartheta/G$ and notice that $A_\vartheta=G\n_A(Q)_{\Psi(\vartheta)}$ so that $\overline{A}_\vartheta\simeq \n_A(Q)_{\Psi(\vartheta)}/\n_G(Q)$. Since $M$ determines an $\overline{A}$-graded basic Morita equivalence between $\mathcal{A}$ and $\mathcal{A}'$, we deduce that $M_{\overline{A}_\vartheta}$ determines an $\overline{A}_{\vartheta}$-graded basic Morita equivalence between $\mathcal{O}A_\vartheta b$ and $\mathcal{O}\n_A(Q)_{\Psi(\vartheta)}b_0$ (see \cite[Corollary 4.3]{Coc-Mar17}). Then, it follows from \cite[Proposition 5.6 and Remark 5.7]{Mar-Min21} that the relation
\begin{equation}
\label{eq:Graded basic Morita equivalences for nilpotent blocks, 2}
\left(A_\vartheta,G,\vartheta\right)\geq_b\left(\n_A(Q)_\vartheta,\n_G(Q),\Psi(\vartheta)\right)
\end{equation}
(with the notation of \cite[Definition 4.2]{Spa18}) holds in the particular case described in \cite[Remark 4.3 (c)]{Spa18}. Since $\c_{A_\vartheta}(Q)$ is contained in $\n_A(Q)_{\Psi(\vartheta)}$, it follows from the definition of $G$-block isomorphism of character triples \cite[Definition 3.6]{Spa17} that \eqref{eq:Graded basic Morita equivalences for nilpotent blocks, 1} holds if and only if \eqref{eq:Graded basic Morita equivalences for nilpotent blocks, 2} holds, as required.
\end{proof}

Next, we use the properties of $G$-block isomorphisms of character triples to lift the character correspondence obtained in Corollary \ref{cor:Graded basic Morita equivalences for nilpotent blocks} to all blocks covering a nilpotent block in a normal subgroup. More precisely, we have the following result.

\begin{cor}
\label{cor:Bijections over nilpotent blocks}
Let $N\leq G\leq A$ be finite groups with $N,G\unlhd A$ and consider a block $B$ of $G$ covering a nilpotent block $b$ of $N$. Let $Q$ be a defect group of $b$ and consider its Brauer correspondent $b_0\in\Bl(\n_N(Q))$. If $B_0\in\Bl(\n_G(Q))$ is the Harris--Kn\"orr correspondent of $B$ covering $b_0$, then there exists a defect preserving $\n_A(Q)_B$-equivariant bijection
\[\Phi_B:\irr(B)\to\irr(B_0)\]
such that
\[\left(A_\vartheta,G,\vartheta\right)\iso{G}\left(\n_A(Q)_\vartheta,\n_G(Q),\Phi_B(\vartheta)\right)\]
for every $\vartheta\in\irr(B)$.
\end{cor}

\begin{proof}
We may assume without loss of generality that $B$ is $A$-invariant, in which case $A=G\n_A(Q)_b$. In fact, if $x\in A=A_B$ then $b$ and $b^x$ are covered $B=B^x$ while $Q^x$ is a defect group of $b^x$. By \cite[Corollary 9.3]{Nav98} we can find an element $g\in G$ such that $b=b^{xg}$. Noticing that $Q$ and $Q^{xg}$ are defect groups of $b$ it follows that $Q=Q^{xgn}$ for some $n\in N$. This shows that $A=G\n_A(Q)_b$. We can then apply \cite[Proposition 2.10]{Ros22} with $A_0=\n_A(Q)_b$, $\Psi:\irr(b)\to\irr(b_0)$ given by Corollary \ref{cor:Graded basic Morita equivalences for nilpotent blocks} and $J=G$. This yields a defect preserving $\n_A(Q)_b$-equivariant bijection
\[\Phi:\irr(G\mid \irr(b))\to\irr(\n_G(Q)\mid \irr(b_0))\]
inducing $G$-block isomorphisms of character triples as required in the statement. By \cite[Theorem 9.28]{Nav98} the bijection $\Phi$ restricts to a bijection $\Phi_B:\irr(B)\to\irr(B_0)$. Since the action of $\n_G(Q)$ is trivial on both $\irr(B)$ and $\irr(B_0)$, the equality $\n_A(Q)=\n_G(Q)\n_A(Q)_b$ implies that $\Phi_B$ is $\n_A(Q)$-equivariant and we are done.
\end{proof}

\subsection{The Character Triple Conjecture for extensions of nilpotent blocks}

Using the character bijection constructed in Corollary \ref{cor:Bijections over nilpotent blocks}, we now prove that the Character Triple Conjecture holds for all blocks covering a nilpotent block with defect groups strictly larger than $Z$, the starting term of $p$-chains.

We start with a preliminary lemma that will be used multiple times throughout this paper. If $\mathcal{B}$ is a collection of blocks of $G$, $d$ a non-negative integer, and $Z\unlhd G$ a $p$-subgroup, then we define $\C^d(\mathcal{B},Z)$ as the (disjoint) union of the sets $\C^d(B,Z)$ for $B$ running over the blocks contained in $\mathcal{B}$. Moreover, if $G\unlhd A$, then we denote by $A_{\mathcal{B}}$ the stabiliser of $\mathcal{B}$ in $A$, that is, the set of elements $x\in A$ such that $B^x\in\mathcal{B}$ for every $B\in\mathcal{B}$.

\begin{lem}
\label{lem:Bijections for union of blocks}
Let $Z\unlhd G\unlhd A$ with $Z$ a $p$-group and consider a collection $\mathcal{B}$ of blocks of $G$ such that $A_B\leq A_{\mathcal{B}}$ for every $B\in\mathcal{B}$. Let $d$ be a non-negative integer and assume that for every $B\in\mathcal{B}$ there exists an $\n_A(Z)_B$-equivariant bijection
\[\Omega_B:\C^d(B,Z)_+/G\to\C^d(B,Z)_-/G,\]
then there exists an $\n_A(Z)_\mathcal{B}$-equivariant bijection
\[\Omega_\mathcal{B}:\C^d(\mathcal{B},Z)_+/G\to\C^d(\mathcal{B},Z)_-/G.\]
\end{lem}

\begin{proof}
Let $\mathcal{T}_0$ be an $\n_A(Z)_\mathcal{B}$-transversal in $\mathcal{B}$ and fix an $\n_A(Z)_B$-transversal $\mathcal{T}_B^+$ in $\C^d(B,Z)_+/G$ for every $B\in\mathcal{T}_0$. By assumption $\mathcal{T}_B^-:=\Omega_B(\mathcal{T}_B^+)$ is an $\n_A(Z)_B$-transversal in $\C^d(B,Z)_-/G$. Now $\coprod_{B\in\mathcal{T}_0}\mathcal{T}_B^+$ and $\coprod_{B\in\mathcal{T}_0}\mathcal{T}_B^-$ are $\n_A(Z)_\mathcal{B}$-transversals in $\C^d(\mathcal{B},Z)_+/G$ and $\C^d(\mathcal{B},Z)_-/G$ respectively. We conclude by defining $\Omega_\mathcal{B}(\overline{(\sigma,\vartheta)}^x):=\Omega_B(\overline{(\sigma,\vartheta)})^x$ for every $(\sigma,\vartheta)\in\mathcal{T}^+_B$, $B\in\mathcal{T}_0$ and $x\in\n_A(Z)_\mathcal{B}$.
\end{proof}

We are now ready to prove the main result of this section. Observe that this can be seen as an adaptation of Robinson's \textit{Auxiliary Theorem} from \cite{Rob02} to the framework of the Character Triple Conjecture.

\begin{prop}
\label{prop:Auxiliary theorem}
Let $Z\leq N\leq G\leq A$ be finite groups with $Z\unlhd G$ a $p$-group and $N, G\unlhd A$. Consider a block $B$ of $G$ covering a block $b$ of $N$ and assume that $b$ is nilpotent with defect groups strictly containing $Z$. Then, for every non-negative integer $d$, there exists an $\n_A(Z)_B$-equivariant bijection
\[\Omega:\C^d(B,Z)_+/G \to \C^d(B,Z)_-/G\]
such that 
\[\left(A_{\sigma,\vartheta},G_\sigma,\vartheta\right)\iso{G}\left(A_{\rho,\chi},G_{\rho},\chi\right),\]
for every $(\sigma,\vartheta)\in\C^d(B,Z)_+$ and $(\rho,\chi)\in\Omega(\overline{(\sigma,\vartheta)})$.
\end{prop}

\begin{proof}
By \cite[Theorem 9.26]{Nav98} there is a defect group $D$ of $B$ such that $D\cap N$ is a defect group of $b$. For every $\overline{(\sigma,\vartheta)}\in\C^d(B,Z)/G$ we may assume that all terms of $\sigma$ are contained in $D$. In fact, if $C$ is a block of $G_\sigma$ such that $C^G=B$, then \cite[Lemma 4.13]{Nav98} implies that there exists a defect group $P$ of $C$ and $g\in G$ such that $P\leq D^g$. Then $D_i\leq\O_p(G_\sigma)\leq P\leq D^g$. Replacing $(\sigma,\vartheta)$ with $(\sigma,\vartheta)^{g^{-1}}$ we obtain our claim. Now, the chain $\sigma=\{Z=D_0<D_1<\dots<D_n\}$ satisfies exactly one of the following conditions:
\begin{enumerate}
\item $\sigma=\{Z\}$;
\item $\sigma=\{Z<D\cap N\}$;
\item $Z\leq D_1\cap N<D\cap N$;
\item $Z<D_1\cap N=D\cap N<D_1$; or
\item $Z<D_1\cap N=D\cap N=D_1$ and $n=|\sigma|\geq 2$.
\end{enumerate}
We first consider the contribution given by the $p$-chains of type (iv) and (v). If $(\sigma,\vartheta)\in\C^d(B,Z)_+$ and $\sigma$ is of type (iv), then we define $\rho$ to be the chain obtained by adding the term $D_1\cap N=D\cap N$ to $\sigma$. Notice that $\rho$ is a chain of type (v) and that $G_\rho=G_\sigma$. In this case, we define $\Omega(\overline{(\sigma,\vartheta)}):=\overline{(\rho,\vartheta)}$. On the other hand, if $\sigma$ is of type (v), we define $\rho$ to be the chain obtained by removing the term $D_1$ from $\sigma$. Since $D_1\cap N\leq D_2\cap N\leq D\cap N=D_1\cap N$, we deduce that $\rho$ is of type (iv) and that $G_\rho=G_\sigma$. In this case, we define $\Omega(\overline{(\sigma,\vartheta)})=\overline{(\rho,\chi)}$. This shows that we can define the bijection $\Omega$ from the statement for all pairs $(\sigma,\vartheta)\in\C^d(B,Z)_+$ such that $\sigma$ satisfies (iv) or (v).

Next, we consider chains $\sigma$ satisfying condition (iii). Let $\mathcal{S}$ be the set of $p$-subgroups $Q$ of $G$ satisfying $Z<Q$ and $Q^g\cap N<D\cap N$ for some $g\in G$. Denote by $\mathcal{S}/G$ the set of $G$-orbits $\overline{Q}$ of $p$-subgroups $Q\in\mathcal{S}$. Observe that $\n_A(Z)_B$ acts on $\mathcal{S}/G$. In fact, if $x\in\n_A(Z)_B$ and $\overline{Q}\in\mathcal{S}/G$, then we find $g\in G$ such that $Q^g\cap N<D\cap N$ and so $Q^{gx}\cap N<D^x\cap N$. Since $B^x=B$, we deduce that $D^x$ is a defect group of $B$ and hence there exists some $h\in G$ such that $D^{xh}=D$. It follows that $Q^{gxh}\cap N<D\cap N$. Moreover, as $G\unlhd \n_A(Z)_B$, we can find $g'\in G$ such that $gx=xg'$ and therefore $Q^{xg'h}\cap N<D\cap N$. This shows that $\overline{Q}^x\in\mathcal{S}/G$ as claimed. We now fix an $\n_A(Z)_B$-transversal $\mathcal{T}_0$ in $\mathcal{S}/G$. For any $\overline{Q}\in\mathcal{T}_0$, we define the subset $\C^d_{\overline{Q}}(B,Z)_\pm$ of $\C^d(B,Z)_\pm$ consisting of those pairs $(\sigma,\vartheta)$ satisfying $\overline{D_1}=\overline{Q}$, where we recall that $D_1$ denotes the second term of $\sigma$. We claim that there exists a $G\n_A(Z,Q)_B$-equivariant bijection
\begin{equation}
\label{eq:Auxiliary theorem 1}
\Omega_{\overline{Q}}:\C^d_{\overline{Q}}(B,Z)_+/G\to\C^d_{\overline{Q}}(B,Z)_-/G
\end{equation}
such that
\[\left(A_{\sigma,\vartheta},G_\sigma,\vartheta\right)\iso{G}\left(A_{\rho,\chi},G_\rho,\chi\right)\]
for every $(\sigma,\vartheta)\in\C^d_{\overline{Q}}(B,Z)_+$ and $(\rho,\chi)\in\Omega_{\overline{Q}}(\overline{(\sigma,\vartheta)})$. For this, let $\C^d_Q(B,Z)_\pm$ be the set of pairs $(\sigma,\vartheta)\in\C^d(B,Z)_\pm$ satisfying $D_1=Q$ and notice that it is enough to obtain an $\n_A(Z,Q)_B$-equivariant bijection
\begin{equation}
\label{eq:Auxiliary theorem 2}
\Omega_{Q}:\C^d_Q(B,Z)_+/\n_G(Q)\to\C^d_Q(B,Z)_-/\n_G(Q)
\end{equation}
such that
\[\left(\n_A(Q)_{\sigma,\vartheta},\n_G(Q)_\sigma,\vartheta\right)\iso{\n_G(Q)}\left(\n_A(Q)_{\rho,\chi},\n_G(Q)_\rho,\chi\right)\]
for every $(\sigma,\vartheta)\in\C^d_Q(B,Z)_+$ and $(\rho,\chi)\in\Omega_Q(\overline{(\sigma,\vartheta)})$. In fact, once we have $\Omega_Q$, we can define
\[\Omega_{\overline{Q}}\left(\overline{(\sigma,\vartheta)}\right):=\overline{(\rho,\chi)}\]
for any $(\sigma,\vartheta)\in\C^d_Q(B,Z)_+$ and $(\rho,\chi)\in\C^d_Q(B,Z)_-$ whose $\n_G(Q)$-orbits correspond via $\Omega_Q$. Furthermore, the $\n_G(Q)$-block isomorphisms of character triples induced by the bijection \eqref{eq:Auxiliary theorem 2} imply the $G$-block isomorphisms of character triples required for \eqref{eq:Auxiliary theorem 1}. This follows by using \cite[Lemma 2.11]{Ros22} and since, being $Q$ a term of $\sigma$ and $\rho$, we have $\n_A(Q)_\sigma=A_\sigma$ and $\n_A(Q)_\rho=A_\rho$. Next, we construct the bijection \eqref{eq:Auxiliary theorem 2} proceeding by induction on the order of $G$. We assume without loss of generality that $Z=\O_p(G)$, for if $Z<\O_p(G)$ than the conclusion of the statement follows from the argument of \cite[Lemma 2.3]{Ros22}. In particular, we get $\n_G(Q)<G$ because $Z<Q$. Moreover, thanks to \cite[Proposition 6.10]{Spa17}, we can assume that all $p$-chains $\sigma$ that we are considering satisfy $D(\sigma)\leq G_\sigma$, that is, $\sigma$ is a normal $p$-chain. Let $B_Q$ be the union of those blocks $B'$ of $\n_G(Q)$ satisfying $(B')^G=B$ and observe that the set $\C^d_Q(B,Z)$ is in bijection with the set
\[\C^d(B_Q,Q):=\coprod\limits_{B'\in B_Q}\C^d(B',Q).\]
The bijection can be described as follows: if $(\sigma,\vartheta)\in\C^d_Q(B,Z)$, then consider the $p$-chain $\sigma_0$ of $\n_G(Q)$ with initial term $Q$ obtained by removing $Z$ from $\sigma$ and map $(\sigma,\vartheta)$ to $(\sigma_0,\vartheta)$. Observe that $\sigma_0$ is a chain of $\n_G(Q)$ since $\sigma$ is a normal $p$-chain, and that the assignment $(\sigma,\vartheta)\mapsto(\sigma_0,\vartheta)$ inverts the sign of $p$-chains and preserves the action of $\n_A(Z,Q)$. Applying Lemma \ref{lem:Bijections for union of blocks} to $Q\unlhd \n_G(Q)\unlhd \n_A(Z,Q)$ and $\mathcal{B}=B_Q$, we deduce that it is enough to find a bijection
\[\Omega_{B'}:\C^d(B',Q)_+/\n_G(Q)\to\C^d(B',Q)_-/\n_G(Q)\]
for every $B'$ belonging to $B_Q$ that satisfies
\[\left(\n_A(Q)_{\sigma,\vartheta},\n_G(Q)_\sigma,\vartheta\right)\iso{\n_G(Q)}\left(\n_A(Q)_{\rho,\chi},\n_G(Q)_\rho,\chi\right)\]
for every $(\sigma,\vartheta)\in\C^d(B',Q)_+$ and $(\rho,\chi)\in\Omega_{B'}(\overline{(\sigma,\vartheta)})$. Notice that here we denote by $\overline{(\sigma,\vartheta)}$ the $\n_G(Q)$-orbit of $(\sigma,\vartheta)$. Recalling that $\n_G(Q)<G$, the map $\Omega_{B'}$ is obtained by induction since the hypothesis of the above statement are satisfied with respect to $Q\leq \n_N(Q)Q\leq \n_G(Q)\leq \n_A(Q)$ and the block $B'$ by the argument used in\cite[Theorem 3.2.2]{Rob02}. This proves our claim and we now have constructed a bijection $\Omega_{\overline{Q}}$ as in \eqref{eq:Auxiliary theorem 1} for every $\overline{Q}\in\mathcal{T}_0$. Fix a $G\n_A(Z,Q)_B$-transversal $\mathcal{T}_{\overline{Q}}^+$ in $\C^d_{\overline{Q}}(B,Z)_+/G$ and notice that $\mathcal{T}_{\overline{Q}}^-:=\Omega_{\overline{Q}}(\mathcal{T}_{\overline{Q}}^+)$ is a $G\n_A(Z,Q)_B$-transversal in $\C^d_{\overline{Q}}(B,Z)_-/G$. By the definition of $\mathcal{S}$, the set $\C^d_{\mathcal{S}}(B,Z)_\pm$ of pairs $(\sigma,\vartheta)\in\C^d(B,Z)_\pm$ such that $\overline{D_1}\in\mathcal{S}/G$ coincides with the set of pairs $(\sigma,\vartheta)\in\C^d(B,Z)_\pm$ whose corresponding chain $\sigma$ is (up to $G$-conjugation) of type (iii). It follows from our choices that
\[\coprod_{\overline{Q}\in\mathcal{T}_0}\mathcal{T}_{\overline{Q}}^+\]
and
\[\coprod_{\overline{Q}\in\mathcal{T}_0}\mathcal{T}_{\overline{Q}}^-\]
are $\n_A(Z)_B$-transversals in $\C^d_{\mathcal{S}}(B,Z)_+/G$ and $\C^d_{\mathcal{S}}(B,Z)_-/G$ respectively. Hence, we obtain an $\n_A(Z)_B$-equivariant bijection
\[\Omega_{\mathcal{S}}:\C^d_{\mathcal{S}}(B,Z)_+/G\to\C^d_{\mathcal{S}}(B,Z)_-/G\]
by setting
\[\Omega_{\mathcal{S}}\left(\overline{(\sigma,\vartheta)}^x\right):=\Omega_{\overline{Q}}\left(\overline{(\sigma,\vartheta)}\right)^x\]
for every $\overline{Q}\in\mathcal{T}_0$, $\overline{(\sigma,\vartheta)}\in\mathcal{T}^+_{\overline{Q}}$ and $x\in\n_A(Z)_B$.

To conclude the proof, it remains to consider pairs $(\sigma,\vartheta)\in\C^d(B,Z)$ whose corresponding chains are of type (i) or (ii). The only $G$-orbits remaining in $\C^d(B,Z)_+/G$ are of the form $\overline{(\sigma,\vartheta)}$ with $\sigma=\{Z\}$ while the remaining orbits in $\C^d(B,Z)_-/G$ are of the form $\overline{(\rho,\chi)}$ with $\rho=\{Z<D\cap N\}$. Hence, in order to conclude the proof it is enough to exhibit an $\n_A(Z,D\cap N)_B$-equivariant bijection
\begin{equation}
\label{eq:Auxiliary theorem 3}
\left\lbrace\vartheta\in\irr^d(\n_G(Z))\enspace\middle|\enspace \bl(\vartheta)^G=B\right\rbrace\to\left\lbrace\chi\in\irr^d(\n_G(D\cap N))\enspace\middle|\enspace \bl(\chi)^G=B\right\rbrace
\end{equation}
inducing $G$-block isomorphisms of character triples. However, the set on the left hand side coincides with $\irr^d(B)$ while the set on the right hand side coincides with $\irr^d(B_0)$ where $B_0$ is the Harris-Kn\"orr correspondent of $B$ in $\n_G(D\cap N)$ with respect to the nilpotent block $b$ of $N$ with defect group $D\cap N$. Therefore, the bijection \eqref{eq:Auxiliary theorem 3} is provided by Corollary \ref{cor:Bijections over nilpotent blocks}.
\end{proof}

\section{Equivalent forms of the Character Triple Conjecture}
\label{sec:Equivalence of conjectures}

In this section, we show that the two versions of the Character Triple Conjecture from Section \ref{sec:Conjectures}, namely Conjecture \ref{conj:CTC non-central Z} and Conjecture \ref{conj:CTC}, are logically equivalent, that is, if one holds for every finite group then so does the other, and vice versa. Notice that, although Conjecture \ref{conj:CTC non-central Z} implies Conjecture \ref{conj:CTC} for each given block (see Proposition \ref{prop:CTC non-central implies CTC}), the two statements are not equivalent block-by-block. For any finite group $G$, we denote by $\mathcal{S}(G)$ the set of simple groups $S$ \emph{involved} in $G$, that is, there exist subgroups $N\unlhd H\leq G$ such that $H/N\simeq S$. Moreover, for each prime number $p$, we denote by $\mathcal{S}_p(G)$ the subset of $\mathcal{S}(G)$ consisting of those non-abelian simple groups whose order is divisible by $p$. The main result of this section can then be stated as follows.

\begin{theo}
\label{thm:CTC equivalent to CTC non-central}
Let $G\unlhd A$ be finite groups and $p$ a prime number. If Conjecture \ref{conj:CTC} holds at the prime $p$ for every finite group $G_1\unlhd A_1$ such that $\mathcal{S}_p(G_1)\subseteq \mathcal{S}_p(G)$ and $|A_1:\z(G_1)|\leq|A:\z(G)|$, then Conjecture \ref{conj:CTC non-central Z} holds at the prime $p$ for $G\unlhd A$.
\end{theo}

As an immediate consequence of Theorem \ref{thm:CTC equivalent to CTC non-central} and Proposition \ref{prop:CTC non-central implies CTC}, we deduce that Conjecture \ref{conj:CTC non-central Z} and Conjecture \ref{conj:CTC} are logically equivalent.

\begin{cor}
\label{cor:Equivalent conjectures}
Conjecture \ref{conj:CTC non-central Z} holds at the prime $p$ for every finite group if and only if Conjecture \ref{conj:CTC} holds at the prime $p$ for every finite group.
\end{cor}

Before proceeding to the proof of Theorem \ref{thm:CTC equivalent to CTC non-central}, we explain how its statement can be used when trying to reduce Conjecture \ref{conj:CTC} to quasi-simple groups. For this, suppose that $G\unlhd A$ is a minimal (with respect to $|A:\z(G)|$) counterexample to Conjecture \ref{conj:CTC} and suppose that $X\unlhd Y$ satisfies $\mathcal{S}_p(X)\subseteq \mathcal{S}_p(G)$ and $|Y:\z(X)|<|A:\z(G)|$. By the minimality of $G\unlhd A$, we know that Conjecture \ref{conj:CTC} holds for $X\unlhd Y$ and also for every $X_1\unlhd Y_1$ satisfying $\mathcal{S}_p(X_1)\subseteq\mathcal{S}_p(X)$ and $|Y_1:\z(X_1)|\leq |Y:\z(X)|$. But then, Theorem \ref{thm:CTC equivalent to CTC non-central} tells us that in this case even the stronger statement of Conjecture \ref{conj:CTC non-central Z} holds for $X\unlhd Y$, a fact that can then be used to restrict the structure of $G$ even further. This will be a crucial step in Section \ref{sec:Cancellation} and, in particular, in the proof of Proposition \ref{prop:Forcing bijections in minimal counterexample}.

In order to prove Theorem \ref{thm:CTC equivalent to CTC non-central}, we essentially need to show that a minimal counterexample to Conjecture \ref{conj:CTC non-central Z} would satisfy $U\leq \z(G)$, where $U$ is the normal $p$-subgroup of $G$ used as starting term for our $p$-chains. Our arguments below are inspired by earlier work of Dade \cite[Section 17]{Dad94}, Robinson \cite{Rob96}, and Eaton \cite{Eat04} on the equivalence of Dade's Conjecture and Robinson's Ordinary Weight Conjecture.

\subsection{Product chains and stable reduction for Conjecture \ref{conj:CTC non-central Z}}
\label{sec:Product chains}

We collect some definitions and results from \cite[Section 3-4]{Dad94}. To start, we recall what it means for a $p$-chain $\sigma$ to avoid a given subgroup $H$ of $G$. This notion has been rephrased in \cite[p.319]{Rob96} (see also \cite[p.644]{Eat04}) where it is referred to as \textit{deficiency of $p$-chains}. Let $\sigma$ be a $p$-chain of $G$ starting with the normal $p$-subgroup $U$ and consider any subgroup $H\leq G$. Since $U$ normalises all the terms of $\sigma$, we deduce that $U\cap H\leq D(\sigma)\cap H_\sigma$, where as usual $D(\sigma)$ denotes the largest term of $\sigma$. Moreover, if $\sigma$ is a normal $p$-chain, then $D(\sigma)\cap H_\sigma=D(\sigma)\cap H$. We recall once again that we may, and do, assume that all $p$-chains we are dealing with are normal (see \cite[Proposition 6.10]{Spa17}).

\begin{defin}
\label{def:Chain avoidance}
Let $U\unlhd G$ be a $p$-subgroup and consider $H\leq G$. A normal $p$-chain $\sigma\in\mathfrak{N}(G,U)$ \emph{avoids} $H$ if $U\cap H=D(\sigma)\cap H$ or equivalently if $D(\sigma)\cap H\leq U$. We denote by $\mathfrak{A}_H(G,U)$ the set of all such $p$-chains.
\end{defin}

As described in \cite[Proposition 3.6]{Dad94}, for a fixed subgroup $H\leq G$, one can recover some of the $p$-chains of $G$ as a product of a $p$-chain of $G$ that avoids $H$ and a $p$-chain of $H$. More precisely, we have the following lemma.

\begin{lem}
\label{lem:Product of chains definition}
Let $U\unlhd G$ be a $p$-subgroup and consider $H\leq G$. For every $\upsilon\in\mathfrak{A}_H(G,U)$ and $\varrho\in\mathfrak{N}(H_\upsilon,H\cap U)$ there is a $p$-chain $\upsilon\ast\varrho\in\mathfrak{N}(G,U)$ of length $|\upsilon\ast\varrho|=|\upsilon|+|\varrho|$. The $p$-chain $\upsilon\ast\varrho$ is uniquely determined by $\upsilon$ and $\varrho$ and its stabiliser satisfies $H_{\upsilon\ast\varrho}=H_\upsilon\cap H_\varrho$. More generally, if $G\unlhd A$, then $\n_A(H)_{\upsilon\ast\varrho}=\n_A(H)_\upsilon\cap \n_A(H)_\varrho$.
\end{lem}

\begin{proof}
This follows immediately from \cite[Proposition 3.6 and Proposition 3.10]{Dad94}.
\end{proof}

In what follows, we denote by $\mathfrak{N}(G:H,U)$ the set of all products $\upsilon\ast\varrho$ where $\upsilon\in\mathfrak{A}_H(G,U)$ and $\varrho\in\mathfrak{N}(H_\upsilon,U\cap H)$. While in general $\mathfrak{N}(G:H,U)$ is a proper subset of $\mathfrak{N}(G,U)$, the following lemma shows that, for most purposes, we can get rid of all the $p$-chains lying outside this subset.

\begin{lem}
\label{lem:Product of chains reduction}
Let $U\unlhd G$ be a $p$-subgroup and consider $H\leq G\unlhd A$. Then there exists a self-inverse $\n_A(U,H)$-equivariant bijection
\[\gamma:\mathfrak{N}(G,U)\setminus \mathfrak{N}(G:H,U)\to\mathfrak{N}(G,U)\setminus \mathfrak{N}(G:H,U)\]
such that $|\gamma(\sigma)|=|\sigma|\pm 1$ for every $\sigma\in\mathfrak{N}(G,U)\setminus \mathfrak{N}(G:H,U)$.
\end{lem}

\begin{proof}
This is essentially a reformulation of \cite[Theorem 4.3]{Dad94}. We refer the reader to the proof of that theorem for the construction of the map $\gamma$ (see \cite[p.110-113]{Dad94}).
\end{proof}

We now fix a character $\varphi\in\irr(U)$ and consider the above situation with respect to the subgroup $H=G_\varphi$. Notice in this case that $U$ is contained in $H$. We want to reduce Conjecture \ref{conj:CTC non-central Z} to the case where the character $\varphi$ is $A$-invariant. For this, we will use Lemma \ref{lem:Product of chains reduction} to construct bijections between the orbit sets of $\C^d(B,U,\varphi)_\pm$ via Clifford theory and by only considering $p$-chains contained in $\mathfrak{N}(G:G_\varphi,U)$ hence reducing the problem to dealing with $p$-chains of the stabiliser $G_\varphi$. The reader should bear with us a little longer while we introduce some further notation. 

For every block $B$ of $G$ and $d\geq 0$, we define the set $\C^d(B:G_\varphi,U,\varphi)$ consisting of those pairs $(\sigma,\vartheta)$ where $\sigma\in \mathfrak{N}(G:G_\varphi,U)$ and $\vartheta\in\irr^d(G_\sigma\mid\varphi)$ satisfies $\bl(\vartheta)^G=B$. Moreover, if $\upsilon$ is a fixed $p$-chain belonging to $\mathfrak{A}_{G_\varphi}(G,U)$, then we denote by $\C^d(B:G_\varphi,U,\varphi)_\upsilon$ the set of pairs $(\sigma,\vartheta)\in\C^d(B:G_\varphi,U,\varphi)$ such that $\sigma=\upsilon\ast\varrho$ for some $\varrho\in\mathfrak{N}(G_{\varphi,\upsilon},U)$. Next, let $B_{\varphi,\upsilon}$ be the union of all blocks $B'$ of $G_{\varphi,\upsilon}$ such that $(B')^G$ is defined and coincides with $B$. Then
\[\C^d(B_{\varphi,\upsilon},U,\varphi)=\coprod\limits_{B'\in B_{\varphi,\upsilon}}\C^d(B',U,\varphi)\]
is the set of pairs $(\varrho,\eta)$ with $\varrho\in\mathfrak{N}(G_{\varphi,\upsilon},U)$ and $\eta\in\irr^d(G_{\varphi,\upsilon,\varrho}\mid \varphi)$ such that $\bl(\eta)^G=B$. Notice here that $\bl(\eta)^G$ is defined according to \cite[Lemma 3.2]{Kno-Rob89}, \cite[Corollary 6.2]{Nav98}, and using \cite[Lemma 5.3.4]{Nag-Tsu89}. Then, if we denote by $\C^d(B_{\varphi,\upsilon},U,\varphi)_\upsilon$ the set of those elements $(\upsilon\ast\varrho,\eta)$ such that $(\varrho,\eta)\in\C^d(B_{\varphi,\upsilon},U,\varphi)$, we obtain an $\n_A(U)_{B,\varphi,\upsilon}$-equivariant bijection
\begin{align}
\label{eq:Product of chain, induction yields bijection of sets}
\C^d(B_{\varphi,\upsilon},U,\varphi)_\upsilon&\to \C^d(B:G_\varphi,U,\varphi)_\upsilon
\\
(\sigma,\eta)&\mapsto (\sigma,\eta^{G_\sigma})\nonumber
\end{align}
by Clifford theory and recalling that $G_{\varphi,\upsilon\ast\varrho}=G_{\varphi,\upsilon,\varrho}$ according to Lemma \ref{lem:Product of chains definition}. In the rest of this section we consider the following hypothesis.

\begin{hyp}
\label{hyp:Product chains and stable reduction}
Let $G\unlhd A$ be finite groups, $U$ a $p$-subgroup of $G$ normalised by $A$, and consider $\varphi\in\irr(U)$, $B$ a block of $G$ and $d\geq 0$. Suppose that for every $\upsilon\in\mathfrak{A}_{G_\varphi}(G,U)$ and $B'\in B_{\varphi,\upsilon}$ there exists an $A_{B',\varphi,\upsilon}$-equivariant bijection
\[\Theta_{B',\upsilon}:\C^d(B',U,\varphi)_+/G_{\varphi,\upsilon}\to\C^d(B',U,\varphi)_-/G_{\varphi,\upsilon}\]
such that
\[\left(A_{\varphi,\upsilon,\varrho,\eta},G_{\varphi,\upsilon,\varrho},\eta\right)\iso{G_{\varphi,\upsilon}}\left(A_{\varphi,\upsilon,\varpi,\kappa},G_{\varphi,\upsilon,\varpi},\kappa\right)\]
for every $(\varrho,\eta)\in\C^d(B',U,\varphi)_+$ and $(\varpi,\kappa)\in\Theta_{B',\upsilon}(\overline{(\varrho,\eta)})$ and where we denote by $\overline{(\varrho,\eta)}$ the $G_{\varphi,\upsilon}$-orbit of $(\varrho,\eta)$.
\end{hyp}

Our aim is now to use the above hypothesis to construct bijections between the orbit sets of $\C^d(B:G_\varphi,U,\varphi)_\pm$ (see Lemma \ref{lem:Product chain construct bijections 3} below) and then, via an application of Lemma \ref{lem:Product of chains reduction}, to show how these can be used to construct analogous bijections between the orbit sets of $\C^d(B,U,\varphi)_\pm$. As a first step, we fix a $p$-chain $\upsilon\in\mathfrak{A}_{G_\varphi}(G,U)$ and consider product chains of the form $\upsilon\ast\varrho$ with $\varrho$ running over $\mathfrak{N}(G_{\varphi,\upsilon},U)$. 

\begin{lem}
\label{lem:Product chain construct bijections 1}
Assume Hypothesis \ref{hyp:Product chains and stable reduction} and fix $\upsilon\in\mathfrak{A}_{G_\varphi}(G,U)$. Then there exists an $A_{B,\varphi,\upsilon}$-equivariant bijection
\[\Theta_{B,\upsilon}:\C^d(B_{\varphi,\upsilon},U,\varphi)_{\upsilon,+}/G_{\varphi,\upsilon}\to\C^d(B_{\varphi,\upsilon},U,\varphi)_{\upsilon,-}/G_{\varphi,\upsilon}\]
such that
\[\left(A_{\varphi,\upsilon\ast\varrho,\eta},G_{\varphi,\upsilon\ast\varrho},\eta\right)\iso{G_{\varphi,\upsilon}}\left(A_{\varphi,\upsilon\ast\varpi,\kappa},G_{\varphi,\upsilon\ast\varpi},\kappa\right)\]
for every $(\upsilon\ast\varrho,\eta)\in\C^d(B_{\varphi,\upsilon},U,\varphi)_{\upsilon,+}$ and $(\upsilon\ast\varpi,\kappa)\in\Theta_{B,\upsilon}(\overline{(\varrho,\eta)})$.
\end{lem}

\begin{proof}
Recall that $B_{\varphi,\upsilon}$ is the set of all blocks $B'$ of $G_{\varphi,\sigma}$ such that $(B')^G=B$. Then, by applying Lemma \ref{lem:Bijections for union of blocks} we can combine the bijections $\Theta_{B',\upsilon}$ given by Hypothesis \ref{hyp:Product chains and stable reduction} and obtain an $A_{B,\varphi,\upsilon}$-equivariant bijection
\[\Theta_{B,\upsilon}':\C^d(B_{\varphi,\upsilon},U,\varphi)_+/G_{\varphi,\upsilon}\to\C^d(B_{\varphi,\upsilon},U,\varphi)_-/G_{\varphi,\upsilon}\]
Notice that, although the statement of Lemma \ref{lem:Bijections for union of blocks} does not take into account characters lying above the given irreducible character $\varphi\in\irr(U)$, the same argument can be applied here since all the bijections $\Theta_{B',\upsilon}$ preserve the character $\varphi$. Now, we define an $A_{B,\varphi,\upsilon}$-equivariant bijection
\[\Theta_{B,\upsilon}:\C^d(B_{\varphi,\upsilon},U,\varphi)_{\upsilon,+}/G_{\varphi,\upsilon}\to\C^d(B_{\varphi,\upsilon},U,\varphi)_{\upsilon,-}/G_{\varphi,\upsilon}\]
by setting
\[\Theta_{B,\upsilon}\left(\overline{(\upsilon\ast\varrho,\eta)}\right):=\overline{(\upsilon\ast\varpi,\kappa)}\]
for every $(\varrho,\eta)\in\C^d(B_{\varphi,\upsilon},U,\varphi)_+$ and $(\varpi,\kappa)\in\Theta'_{B,\upsilon}(\overline{(\varrho,\eta)})$. To cocnlude, observe that the $G_{\varphi,\upsilon}$-block isomorphisms of character triples required in the statement follow from those given by Hypothesis \ref{hyp:Product chains and stable reduction} after noticing that $A_{\varphi,\upsilon\ast\varrho}=A_{\varphi,\upsilon,\varrho}$ for every $\varrho\in\mathfrak{N}(G_{\varphi,\upsilon},U)$ thanks to Lemma \ref{lem:Product of chains definition}.
\end{proof}

Induction of characters yields a bijection between $\C^d(B_{\varphi,\upsilon},U,\varphi)_\upsilon$ and $\C^d(B:G_\varphi,U,\varphi)_\upsilon$ as explained in \eqref{eq:Product of chain, induction yields bijection of sets}. We can then use the maps $\Theta_{B,\upsilon}$ from the previous lemma to obtain the following result.

\begin{lem}
\label{lem:Product chain construct bijections 2}
Assume Hypothesis \ref{hyp:Product chains and stable reduction} and fix $\upsilon\in\mathfrak{A}_{G_\varphi}(G,U)$. Then there exists an $A_{B,\varphi,\upsilon}$-equivariant bijection
\[\Omega_{B:G_\varphi}^\upsilon:\C^d(B:G_\varphi,U,\varphi)_{\upsilon,+}/G_{\varphi,\upsilon}\to\C^d(B:G_{\varphi},U,\varphi)_{\upsilon,-}/G_{\varphi,\upsilon}\]
such that
\[\left(A_{\upsilon\ast\varrho,\vartheta},G_{\upsilon\ast\varrho},\vartheta\right)\iso{G_{\upsilon}}\left(A_{\upsilon\ast\varpi,\chi},G_{\upsilon\ast\varpi},\chi\right)\]
for every $(\upsilon\ast\varrho,\vartheta)\in\C^d(B:G_{\varphi},U,\varphi)_{\upsilon,+}$ and $(\upsilon\ast\varpi,\chi)\in\Omega_{B:G_\varphi}^\upsilon(\overline{(\upsilon\ast\varrho,\vartheta)})$.
\end{lem}

\begin{proof}
Write $\sigma=\upsilon\ast\varrho$ for some $\varrho\in\mathfrak{N}(G_{\varphi,\upsilon},U)$ and consider $\vartheta\in\irr(G_\sigma\mid \vartheta)$. By Clifford's theorem there exists a unique irreducible character $\eta\in\irr(G_{\varphi,\sigma}\mid \varphi)$ such that $\eta^{G_\sigma}=\vartheta$. Moreover, if $g\in G_{\varphi,\upsilon}$, then $\sigma^g=\upsilon\ast\varrho^g$ and $\eta^g\in\irr(G_{\varphi,\sigma^g}\mid \varphi)$ is the Clifford correspondent of $\vartheta^g\in\irr(G_{\sigma^g}\mid \varphi)$. In particular, the bijection \eqref{eq:Product of chain, induction yields bijection of sets} induces an $A_{B,\varphi,\upsilon}$-equivariant bijection
\begin{align*}
\C^d(B_{\varphi,\upsilon},U,\varphi)_\upsilon/G_{\varphi,\upsilon}&\to \C^d(B:G_\varphi,U,\varphi)_\upsilon/G_{\varphi,\upsilon}
\\
\overline{(\sigma,\eta)}&\mapsto \overline{(\sigma,\eta^{G_\sigma})}.
\end{align*}
If $\Theta_{B,\upsilon}$ is the map given by Lemma \ref{lem:Product chain construct bijections 1}, then we define $\Omega_{B:G_\varphi}^\upsilon$ by setting
\[\Omega_{B:G_\varphi}^\upsilon\left(\overline{(\sigma,\vartheta)}\right):=\overline{(\rho,\chi)}\]
for every $(\sigma,\vartheta)\in\C^d(B:G_\varphi,U,\varphi)_{\upsilon,+}$ and $(\rho,\chi)\in\C^d(B:G_\varphi,U,\varphi)_{\upsilon,-}$ such that $(\rho,\kappa)\in\Theta_{B,\upsilon}(\overline{(\sigma,\eta)})$ and where $\eta\in\irr(G_{\varphi,\sigma}\mid \varphi)$ and $\kappa\in\irr(G_{\varphi,\rho}\mid \varphi)$ are the Clifford correspondents over $\varphi$ of $\vartheta$ and $\chi$ respectively. Write now $\sigma=\upsilon\ast\varrho$ and $\rho=\upsilon\ast\varpi$ so that
\[\left(A_{\varphi,\upsilon\ast\varrho,\eta},G_{\varphi,\upsilon\ast\varrho},\eta\right)\iso{G_{\varphi,\upsilon}}\left(A_{\varphi,\upsilon\ast\varpi,\kappa},G_{\varphi,\upsilon\ast\varpi},\kappa\right)\]
according to Lemma \ref{lem:Product chain construct bijections 1}. Then, applying \cite[Proposition 2.8]{Ros22} we deduce that
\[\left(A_{\upsilon\ast\varrho,\vartheta},G_{\upsilon\ast\varrho},\vartheta\right)\iso{G_{\upsilon}}\left(A_{\upsilon\ast\varpi,\chi},G_{\upsilon\ast\varpi},\chi\right)\]
and hence
\[\left(A_{\upsilon\ast\varrho,\vartheta},G_{\upsilon\ast\varrho},\vartheta\right)\iso{G}\left(A_{\upsilon\ast\varpi,\chi},G_{\upsilon\ast\varpi},\chi\right)\]
thanks to \cite[Lemma 2.11]{Ros22}. To verify the condition on the centralisers of defect groups required in \cite[Lemma 2.11]{Ros22}, suppose that $Q$ is a defect group of $\bl(\vartheta)$ and notice that the last term $D(\upsilon\ast\varrho)$ of $\upsilon\ast\varrho$ is contained in $\O_p(G_{\upsilon\ast\varrho})\leq Q$. Then $\c_{GA_{\upsilon\ast\varrho}}(Q)\leq A_{\upsilon\ast\varrho}$ and we deduce that $\c_{GA_{\upsilon\ast\varrho,\vartheta}}(Q)\leq GA_{\upsilon\ast\varrho,\vartheta}$ as required.
\end{proof}

Finally, assuming that Hypothesis \ref{hyp:Product chains and stable reduction} holds, we can construct bijections between the orbit sets of ${\C^d(B:G_\varphi,U,\varphi)}_\pm$.

\begin{lem}
\label{lem:Product chain construct bijections 3}
Assume Hypothesis \ref{hyp:Product chains and stable reduction}. Then there exists an $A_{B,\varphi}$-equivariant bijection
\[\Omega_{B:G_\varphi}:\C^d(B:G_\varphi,U,\varphi)_+/G_\varphi\to\C^d(B:G_{\varphi},U,\varphi)_-/G_\varphi\]
such that
\[\left(A_{\sigma,\vartheta},G_\sigma,\vartheta\right)\iso{G}\left(A_{\rho,\chi},G_\rho,\chi\right)\]
for every $(\sigma,\vartheta)\in\C^d(B:G_{\varphi},U,\varphi)_+$ and $(\rho,\chi)\in\Omega_{B:G_\varphi}(\overline{(\sigma,\vartheta)})$.
\end{lem}

\begin{proof}
To start, fix an $A_{B,\varphi}$-transversal $\mathcal{A}$ in $\mathfrak{A}_{G_\varphi}(G,U)$. For each $\upsilon\in\mathcal{A}$ we consider the map $\Omega_{B:G_\varphi}^\upsilon$ given by Lemma \ref{lem:Product chain construct bijections 2} and fix an $A_{B,\varphi,\upsilon}$-transversal $\mathcal{T}^\upsilon_+$ in $\C^d(B:G_\varphi,U,\varphi)_{\upsilon,-}$. By choosing an element $(\upsilon\ast\varpi,\chi)$ in each $G_{\varphi,\upsilon}$-orbit $\Omega^\upsilon_{B:G_\varphi}(\overline{(\upsilon\ast\varrho,\vartheta)})$ for every $(\upsilon\ast\varrho,\vartheta)\in\mathcal{T}_+^\upsilon$ (and where $\overline{(\upsilon\ast\varrho,\vartheta)}$ is the $G_{\varphi,\upsilon}$-orbit of $(\upsilon\ast\varrho,\vartheta)$) we obtain an $A_{B,\varphi,\upsilon}$-transversal in $\C^d(B:G_\varphi,U,\varphi)_{\upsilon,-}$ which we denote by $\mathcal{T}_-^\upsilon$. Then, the sets (now of $G_\varphi$-orbits)
\[\mathcal{T}_+:=\left\lbrace\overline{(\upsilon\ast\varrho,\vartheta)}\enspace\middle|\enspace(\upsilon\ast\varrho,\vartheta)\in\mathcal{T}^\upsilon_+,\upsilon\in\mathcal{A}\right\rbrace\]
and
\[\mathcal{T}_-:=\left\lbrace\overline{(\upsilon\ast\varpi,\chi)}\enspace\middle|\enspace(\upsilon\ast\varpi,\chi)\in\mathcal{T}^\upsilon_-,\upsilon\in\mathcal{A}\right\rbrace\]
are $A_{B,\varphi}$-transversals in $\C^d(B:G_\varphi,U,\varphi)_+/G_\varphi$ and $\C^d(B:G_\varphi,U,\varphi)_-/G_\varphi$ respectively in bijection with each other. We now obtain an $A_{B,\varphi}$-equivariant bijection $\Omega_{B:G_\varphi}$ by setting
\[\Omega_{B:G_\varphi}\left(\overline{(\sigma,\vartheta)}^x\right):=\overline{(\rho,\chi)}^x\]
for every $x\in A_{B,\varphi}$ and every $G_\varphi$-orbits $\overline{(\sigma,\vartheta)}\in\mathcal{T}_+$ and $\overline{(\rho,\chi)}\in\mathcal{T}_-$ whose representatives can be chosen of the form $(\sigma,\vartheta)=(\upsilon\ast\varrho,\vartheta)$ and $(\rho,\chi)=(\upsilon\ast\varpi,\chi)$ for certain $\upsilon\in\mathcal{A}$ and $\varrho,\varpi\in\mathfrak{N}(G_{\varphi,\upsilon},U)$ and whose $G_{\varphi,\upsilon}$-orbits correspond via $\Omega^\upsilon_{B:G_\varphi}$. Moreover, it follows from the above choices that the $G$-block isomorphisms given by $\Omega^\upsilon_{B:G_\varphi}$ imply that
\[\left(A_{\sigma,\vartheta},G_\sigma,\vartheta\right)\iso{G}\left(A_{\rho,\chi},G_\rho,\chi\right)\]
for every $(\sigma,\vartheta)\in\C^d(B:G_{\varphi},U,\varphi)_+$ and $(\rho,\chi)\in\Omega_{B:G_\varphi}(\overline{(\sigma,\vartheta)})$ by choosing $(\sigma,\vartheta)=(\upsilon\ast\varrho,\vartheta)$ and $(\rho,\chi)=(\upsilon\ast\varpi,\chi)$ as above.
\end{proof}

As claimed before, we can now use the cancellation of $p$-chains given by Lemma \ref{lem:Product of chains reduction} to show that in order to control the set $\C^d(B,U,\varphi)$ it is enough to consider pairs in $\C^d(B:G_\varphi,U,\varphi)$, that is, pairs whose associated $p$-chain is the product of a $p$-chain avoiding $G_\varphi$ and a $p$-chain contained in $G_\varphi$.

\begin{prop}
\label{prop:Product of chains reduuction with pairs}
Assume Hypothesis \ref{hyp:Product chains and stable reduction}. Then there exists an $A_{B,\varphi}$-equivariant bijection
\[\Omega_B:\C^d(B,U,\varphi)_+/G_\varphi\to\C^d(B,U,\varphi)_-/G_\varphi\]
such that
\[\left(A_{\sigma,\vartheta},G_\sigma,\vartheta\right)\iso{G}\left(A_{\rho,\chi},G_\rho,\chi\right)\]
for every $(\sigma,\vartheta)\in\C^d(B,U,\varphi)_+$ and $(\rho,\chi)\in\Omega_{B}(\overline{(\sigma,\vartheta)})$.
\end{prop}

\begin{proof}
Let $\gamma$ be the map given by Lemma \ref{lem:Product of chains reduction} applied with $H=G_\varphi$ and denote by $\mathcal{M}$ the set of pairs $(\sigma,\vartheta)\in\C^d(B,U,\varphi)$ such that $\sigma\in\mathfrak{N}(G,U)\setminus\mathfrak{N}(G:G_\varphi,U)$. For any such $\sigma$, Lemma \ref{lem:Product of chains reduction} implies that $G_{\varphi,\sigma}=G_{\varphi,\gamma(\sigma)}$. Then, if $\vartheta\in\irr(G_\sigma\mid \varphi)$ there exists a unique $\gamma(\vartheta)\in\irr(G_{\gamma(\sigma)}\mid \varphi)$ whose Clifford correspondent in $G_{\varphi,\sigma}=G_{\varphi,\gamma(\sigma)}$ over $\varphi$ coincide with the one of $\vartheta$. Using this observation, we can then construct an $A_{B,\varphi}$-equivariant bijection
\begin{align*}
\Gamma:\mathcal{M}_+/G_\varphi &\to\mathcal{M}_-/G_\varphi
\\
\overline{(\sigma,\vartheta)}&\mapsto\overline{(\gamma(\sigma),\gamma(\vartheta))}
\end{align*}
where we use the notation $\overline{(\sigma,\vartheta)}$ to indicate the $G_\varphi$-orbit of $(\sigma,\vartheta)$. Observe here that $(\gamma(\sigma),\gamma(\vartheta))\in\mathcal{M}_-$ since $|\gamma(\sigma)|=|\sigma|\pm 1$ according to Lemma \ref{lem:Product of chains reduction}. If we denote by $\zeta\in\irr(G_{\varphi,\sigma}\mid \varphi)$ the (common) Clifford correspondence of $\vartheta$ and $\gamma(\vartheta)$, then we have
\[\left(A_{\varphi,\sigma,\zeta},G_{\varphi,\sigma},\zeta\right)\iso{G_\varphi}\left(A_{\varphi,\gamma(\sigma),\zeta},G_{\varphi,\gamma(\sigma)},\zeta\right)\]
which together with \cite[Proposition 2.8]{Ros22} (and recalling that $\vartheta=\zeta^{G_\sigma}$ and $\gamma(\vartheta)=\zeta^{G_{\gamma(\sigma)}}$) yields
\[\left(A_{\sigma,\vartheta},G_{\sigma},\vartheta\right)\iso{G}\left(A_{\gamma(\sigma),\gamma(\vartheta)},G_{\gamma(\sigma)},\gamma(\vartheta)\right).\]
Finally, noticing that $\mathcal{M}$ is the complement of $\C^d(B:G_\varphi,U,\varphi)$ in $\C^d(B,U,\varphi)$ we can define $\Omega_B$ by combining $\Gamma$ with the bijection $\Omega_{B:G_\varphi}$ given by Lemma \ref{lem:Product chain construct bijections 3}.
\end{proof}

In Section \ref{sec:Proof of equivalence of conjectures}, we will use the above corollary to obtain a reduction for Conjecture \ref{conj:CTC non-central Z} to the case where the character $\varphi\in\irr(U)$ is invariant under the action of $A$. Once this is established, we can consider a character correspondence, introduced by Dade in \cite{Dad94}, to reduce Conjecture \ref{conj:CTC non-central Z} to the case in which $U$ is contained in the centre of $G$ and hence obtain Theorem \ref{thm:CTC equivalent to CTC non-central}. In the next section, we collect some of the properties of Dade's correspondence.

\subsection{Dade's correspondence over normal $p$-subgroups}
\label{sec:Dade correspondence}

We now consider the case in which the irreducible character $\varphi\in\irr(U)$ from Conjecture \ref{conj:CTC non-central Z} is assumed to be $A$-invariant. Under this assumption, we can construct an extension $\wt{A}$ of $A/U$ by a central $p$-subgroup $\wt{U}$ of $\wt{A}$ inducing correspondences of characters, $p$-blocks, and $p$-chains. The central extensions studied here can be seen as a special case (in which the normal subgroup has $p$-power order) of a more general situation that was first introduced in \cite{Dad94} and have then been reformulated in character theoretic terms by Eaton and Robinson \cite[Section 1]{Eat-Rob02}. The character correspondence resulting from Dade's work extends the one introduced by Fong in \cite{Fon61}. In \cite[Section 4]{Ros22} we have shown that the latter correspondence is compatible with isomorphisms of character triples. The aim of this section is prove an analogue of that result to the case described above. More precisely, we consider the following situation.

\begin{hyp}
\label{hyp:Dade's correspondence}
Let $U$ be a normal $p$-subgroup of $A$ and consider an $A$-invariant character $\varphi\in\irr(U)$. Fix a projective representation $\Pr$ of $A$ associated with the character triple $(A,U,\varphi)$ and let $\alpha$ denote the factor set of $\Pr$. Recall that $\alpha$ can be identified with a factor set of the quotient $A/U$ and that we can further assume that $\alpha(x,y)^{|U|^2}=1$ for every $x,y\in A$ according to \cite[Theorem 3.5.7]{Nag-Tsu89}. We denote by $S$ the $p$-group generated by the values of $\alpha$ and by $\wh{A}$ the $p$-central extension of $A$ by $S$ defined by $\Pr$, or better by $\alpha$, as defined in \cite[Section 5.3]{Nav18}. Recall that $\wh{A}$ consists of pairs $(x,s)$ for $x\in A$ and $s\in S$ and consider the canonical projection
\begin{align*}
\epsilon:\wh{A}&\to A
\\
(x,s)&\mapsto x.
\end{align*}
Set $\wh{X}:=\epsilon^{-1}(X)$ for every $X\leq A$ and observe that, since $\alpha$ is trivial on $U\times U$, we have a direct product decomposition $\wh{X}=X_0\times S$ whenever $X\leq U$ and where $X_0:=\{(x,1)\mid x\in X\}$. In this case, the projection $\epsilon$ induces an isomorphism of $X_0$ with $X$ and we identify any character $\vartheta\in\irr(X)$ with the corresponding character $\vartheta_0$ of $X_0$ defined by $\vartheta_0(x,1):=\vartheta(x)$ for every $x\in X$. Furthermore, the character $\vartheta_0$ can be identified via inflation with $\wh{\vartheta}:=\vartheta_0\times 1_S\in\irr(\wh{X})$. Let $\wh{\lambda}$ be the linear character of $\wh{U}$ given by $\wh{\lambda}(x,s):=s^{-1}$ for every $x\in U$ and $s\in S$ and denote by $\wh{\Pr}$ the irreducible representation of $\wh{A}$ given by $\wh{\Pr}(x,s):=s\Pr(x)$ for every $x\in A$ and $s\in S$. If $\tau$ denotes the irreducible character afforded by $\wh{\Pr}$, then $\tau$ extends $\wh{\varphi}\wh{\lambda}^{-1}\in\irr(\wh{U})$ and $\varphi_0\in\irr(U_0)$. Finally, define $\wt{X}:=\wh{X}U_0/U_0$ for every $X\leq A$ and denote by $\wt{\varphi}$ the character $\wh{\lambda}$ viewed as a character of $\wt{U}=\wh{U}/U_0$. In other words, set $\wt{\varphi}(U_0(x,s)):=s^{-1}$ for every $x\in U$ and $s\in S$.
\end{hyp}

A proof of the following properties of Dade's correspondence can be found in Eaton's PhD thesis \cite[Section 3.2]{Eat99}. We refer the reader to the discussion given in \cite[Section 1]{Eat-Rob02} for a more general situation.

\begin{prop}[Dade's correspondence]
\label{prop:Dade's correspondence}
Consider the setting of Hypothesis \ref{hyp:Dade's correspondence}and let $U\leq H\leq A$. Then:
\begin{enumerate}
\item $\wt{H}$ is a $p$-central extension of $H/U$ by the central $p$-subgroup $\wt{U}\simeq S$;
\item there exists a bijection
\begin{align}
\label{eq:Dade's correspondence character correspondence}
\irr\left(H\enspace\middle|\enspace\varphi\right)&\to\irr\left(\wt{H}\enspace\middle|\enspace \wt{\varphi}\right)
\\
\psi&\mapsto \wt{\psi}\nonumber
\end{align}
such that, if $\wh{\psi}$ is the inflation to $\wh{H}$ of the character of $\wh{H}/S$ corresponding to $\psi$ via the isomorphism $\wh{H}/S\simeq H$ and $\wt{\psi}'$ is the inflation to $\wh{H}$ of the character $\wt{\psi}$ of $\wt{H}=\wh{H}/U_0$, then $\wh{\psi}=\tau_{\wh{H}}\wt{\psi}'$;
\item for every block $B$ of $H$ there is a set of blocks $\mathcal{D}_\varphi(B)$ of $\wt{H}$ such that the bijection \eqref{eq:Dade's correspondence character correspondence} restricts to a bijection
\begin{equation}
\label{eq:Dade's correspondence character correspondence with blocks}
\irr\left(B\enspace\middle|\enspace\varphi\right)\to\irr\left(\mathcal{D}_\varphi(B)\enspace\middle|\enspace \wt{\varphi}\right)
\end{equation}
and $\mathcal{D}_\varphi(B)$ is uniquely determined by $B$ and $\varphi$ via \eqref{eq:Dade's correspondence character correspondence with blocks};
\item for every $\psi\in\irr(H\mid \varphi)$ we have $d(\psi)=d(\wt{\psi})+d(\varphi)-\log_p(|\wt{U}|)$; and
\item for every $\wh{x}\in\wh{A}$ and $\psi\in\irr(H\mid \psi)$ we have $\wt{H^x}=\wt{H}^{\wt{x}}$ and $\wt{\psi^x}=(\wt{\psi})^{\wt{x}}$, where we set $x:=\epsilon(\wh{x})$ and $\wt{x}:=U_0\wh{x}$. In particular $\mathcal{D}_\varphi(B^x)$ coincides with the set of blocks $\wt{B}^{\wt{x}}$ for $\wt{B}\in\mathcal{D}_\varphi(B)$.
\end{enumerate}
\end{prop}

\begin{proof}
The first statement follows directly from the definition, while (ii) and (iii) follow from Lemma 3.6 and Proposition 3.10 in \cite{Eat99} respectively. The statement of (iv) follows by a direct inspection of the character degrees of $\psi$ and $\wt{\psi}$, while (v) is obtained by arguing as in the proof of \cite[Theorem 4.2 (v)]{Ros22}.
\end{proof}

In the above situation, we refer to any block $\wt{B}$ in the set $\mathcal{D}_\varphi(B)$ as \emph{a Dade correspondent} of $B$ with respect to $\varphi$. Similarly, we say that $B$ is \emph{the Dade correspondent}, with respect to $\varphi$, of any $\wt{B}\in\mathcal{D}_\varphi(B)$. Since $U$ is a $p$-group, we deduce that the collection of sets $\mathcal{D}_\varphi(B)$, for $B$ running over all blocks of $H$, is a partition of the set of blocks of $\wt{H}$. In the general case (where $U$ is not a $p$-group), only those blocks of $\wt{H}$ covering the block to which the restriction $\wt{\varphi}_{\O_{p'}(\wt{U})}$ belongs can be obtained as Dade correspondents (of some block $B$ covering the block to which $\varphi$ belongs) with respect to $\varphi$.

For later use and for the reader's convenience, we recall the following compatibility of Dade's correspondence with block induction.

\begin{lem}
\label{lem:Dade's correspondence, block induction}
Consider the setting of Hypothesis \ref{hyp:Dade's correspondence}. Let $U\leq J\leq H\leq A$, $\wt{C}$ be a block of $\wt{J}$, and suppose that $\wt{B}:=\wt{C}^{\wt{H}}$ is defined. If $C\in\Bl(J)$ and $B\in\Bl(H)$ are the Dade correspondent blocks, with respect to $\varphi$, of $\wt{C}$ and $\wt{B}$ respectively, then $C^H$ is defined and coincides with $B$.
\end{lem}

\begin{proof}
This follows from the second part of \cite[Proposition 3.10]{Eat99}.
\end{proof}

From now on, we assume furthermore that $U\leq G\unlhd A$. First, we consider the relation between the sets of $p$-chains $\mathfrak{N}(G,U)$ and $\mathfrak{N}(\wt{G},\wt{U})$.

\begin{lem}
\label{lem:Dade's correspondence, chains}
Consider the setting of Hypothesis \ref{hyp:Dade's correspondence} and let $U\leq G\unlhd A$. For every $\sigma\in\mathfrak{N}(G,U)$, we denote by $\wt{\sigma}$ the $p$-chain of $\wt{G}$ whose terms are given by the images $\wt{D}_i$ of the terms $D_i$ of the $p$-chain $\sigma$. Then the map
\begin{align*}
\mathfrak{N}(G,U)&\to\mathfrak{N}\left(\wt{G},\wt{U}\right)
\\
\sigma &\mapsto \wt{\sigma}
\end{align*}
is bijective, equivariant with respect to the action of $A/U\simeq \wt{A}/\wt{U}$, and preserves the length of $p$-chains. In particular $\wt{G_\sigma}=\wt{G}_{\wt{\sigma}}$ for every $\sigma\in\mathfrak{N}(G,U)$.
\end{lem}

\begin{proof}
The first statement follows from the correspondence theorem after noticing that $G/U\simeq \wt{G}/\wt{U}$ and that $U$ and $\wt{U}$ are $p$-subgroups. Consider $\sigma\in\mathfrak{N}(G,U)$ and notice that $(G/U)_{\sigma}\simeq(\wt{G}/\wt{U})_{\wt{\sigma}}$ by the equivariance property of the above bijection. Since $(G/U)_\sigma=G_\sigma/U$ and $(\wt{G}/\wt{U})_{\wt{\sigma}}=\wt{G}_{\wt{\sigma}}/\wt{U}$, we deduce that $\wt{G_\sigma}=\wt{G}_{\wt{\sigma}}$.
\end{proof}

The identification of $p$-chains described in Lemma \ref{lem:Dade's correspondence, chains}, together with Dade's correspondence of blocks and the character bijections given in \eqref{eq:Dade's correspondence character correspondence with blocks}, allows us to construct a bijection between the sets $\C^d(B,U,\varphi)$ and $\C^{\wt{d}}(\mathcal{D}_\varphi(B),\wt{U},\wt{\varphi})$ that we now describe. Recall that for a set of blocks $\mathcal{B}$ of a finite group $X$, we denote by $\irr(\mathcal{B})$ the set of irreducible characters belonging to any of the blocks belonging to the set $\mathcal{B}$. Furthermore, if $\sigma$ is a $p$-chain of $X$, then $\irr(\mathcal{B}_\sigma)$ denotes the set of irreducible characters of the stabiliser $X_\sigma$ whose block correspond via block induction to a block in the set $\mathcal{B}$. With this notation, and given $d\geq 0$, a $p$-subgroup $U\leq X$, and $\varphi\in\irr(U)$, the set $\C^d(\mathcal{B},U,\varphi)$ consists of pairs $(\sigma,\vartheta)$ with $\sigma\in\mathfrak{N}(X,U)$ and $\vartheta\in\irr^d(\mathcal{B}_\sigma\mid \varphi):=\irr(\mathcal{B}_\sigma)\cap \irr^d(X_\sigma\mid \varphi)$. In the case where $X\unlhd Y$, we denote by $Y_{\mathcal{B}}$ the stabiliser in $Y$ of the set $\mathcal{B}$, that is, the set of those elements $y\in Y$ such that $B^y\in\mathcal{B}$ for every $B\in\mathcal{B}$.

\begin{lem}
\label{lem:Dade's correspondence, pairs}
Consider the setting of Hypothesis \ref{hyp:Dade's correspondence} and let $U\leq G\unlhd A$. Consider a block $B$ of $G$, $d\geq 0$ and define $\wt{d}:=d-d(\varphi)+\log_p(|\wt{U}|)$. Then the map
\begin{align*}
\C^d\left(B,U,\varphi\right)/G &\to \C^{\wt{d}}\left(\mathcal{D}_\varphi(B),\wt{U},\wt{\varphi}\right)/\wt{G}
\\
\overline{(\sigma,\vartheta)}&\mapsto\overline{(\wt{\sigma},\wt{\vartheta})}
\end{align*}
is bijective and equivariant with respect to the action of $A_B/U\simeq \wt{A}_{\mathcal{D}_\varphi(B)}/\wt{U}$. Here, we denote by $\overline{(\sigma,\vartheta)}$ and $\overline{(\wt{\sigma},\wt{\vartheta})}$ the $G$-orbit of $(\sigma,\vartheta)$ and the $\wt{G}$-orbit of $(\wt{\sigma},\wt{\vartheta})$ respectively.
\end{lem}

\begin{proof}
Let $\sigma\in\mathfrak{N}(G,U)$ and $\vartheta\in\irr(G_\sigma\mid \varphi)$. By Lemma \ref{lem:Dade's correspondence, chains}, we know that $\wt{G_\sigma}=\wt{G}_{\wt{\sigma}}$ and that $\wt{\vartheta}\in\irr(\wt{G}_{\wt{\sigma}}\mid \wt{\varphi})$. According to Proposition \ref{prop:Dade's correspondence} (iv) we know that $d(\vartheta)=d$ if and only if $d(\wt{\vartheta})=\wt{d}$. Furthermore, the character $\vartheta$ belongs to $\irr(B_\sigma)$ if and only if $\wt{\vartheta}$ belongs to $\irr(\mathcal{D}_\varphi(B)_{\wt{\sigma}})$ thanks to Lemma \ref{lem:Dade's correspondence, block induction}. Therefore, the association $\vartheta\mapsto\wt{\vartheta}$ yields a bijection
\[\irr^d\left(B_\sigma\enspace\middle|\enspace \varphi\right)\to \irr^{\wt{d}}\left(\mathcal{D}_\varphi(B)_{\wt{\sigma}}\enspace\middle|\enspace\wt{\varphi}\right)\]
which is equivariant with respect to the action of $A_B/U\simeq \wt{A}_{\mathcal{D}_\varphi(B)}/\wt{U}$ thanks to Proposition \ref{prop:Dade's correspondence} (v). This observation, combined with the bijection given by Lemma \ref{lem:Dade's correspondence, chains}, yields the desired result.
\end{proof}

Suppose now that, with the notation of Lemma \ref{lem:Dade's correspondence, pairs}, we have $d>d(\varphi)$ and hence $\wt{d}>\log_p(|\wt{U}|)$. Let $\sigma\in\mathfrak{N}(G,U)$ and suppose that there exists a Dade correspondent $\wt{B}$ of $B$ with respect to $\varphi$ and $\wt{\vartheta}\in\irr(\wt{B}_{\wt{\sigma}})$ such that $d(\wt{\vartheta})=\wt{d}$. In this case, we claim that the defect groups of $\wt{B}$ strictly contain the central $p$-subgroup $\wt{U}$. In fact, if $\wt{Q}$ is a defect group of $\bl(\wt{\vartheta})$, then we can find a defect group $\wt{D}$ of $\wt{B}$ such that $\wt{Q}\leq\wt{D}$ (see \cite[Lemma 4.13]{Nav98}). Since $\wt{U}\leq \O_p(\wt{G}_{\wt{\sigma}})\leq \wt{Q}$ we deduce that $\wt{U}\leq \wt{D}$. On the other hand, \cite[Theorem 4.6]{Nav98} implies that $\log_p(|\wt{U}|)<\wt{d}\leq \log_p(|\wt{Q}|)\leq \log_p(|\wt{D}|)$ and hence $\wt{U}<\wt{D}$. In particular, this shows that the hypothesis of Conjecture \ref{conj:CTC} are satisfied with respect to the block $\wt{B}$ and the groups $\wt{U}$, $\wt{G}$, and $\wt{A}$. In the following proposition, we show that if Conjecture  Conjecture \ref{conj:CTC} holds for all such blocks $\wt{B}$, then Conjecture \ref{conj:CTC non-central Z} holds for the Dade correspondent $B$ with respect to any $d>d(\varphi)$.

\begin{prop}
\label{prop:Dade correspondence, CTC implies CTC non-central}
Consider the setting of Hypothesis \ref{hyp:Dade's correspondence} and let $U\leq G\unlhd A$. Consider a block $B$ of $G$, $d>d(\varphi)$ and suppose that Conjecture \ref{conj:CTC} holds for $\wt{G}\unlhd \wt{A}$ with respect to the $p$-subgroup $\wt{U}$, any $\wt{B}\in\mathcal{D}_\varphi(B)$ and $\wt{d}:=d-d(\varphi)+\log_p(|\wt{U}|)$. Then Conjecture \ref{conj:CTC non-central Z} holds for $G\unlhd A$ with respect to $U$, $\varphi$, $B$ and $d$.
\end{prop}

\begin{proof}
Let $\wt{B}$ be a Dade correspondent of $B$ with respect to $\varphi$. If $\C^{\wt{d}}(\wt{B},\wt{U})$ is non-empty, then $\wt{U}$ is strictly contained in the defect groups of $\wt{B}$ as explained in the discussion preceding this proposition. In this case, by our assumption there exists an $\wt{A}_{\wt{B}}$-equivariant bijection
\[\wt{\Omega}_{\wt{B}}:\C^{\wt{d}}\left(\wt{B},\wt{U}\right)_+/\wt{G}\to\C^{\wt{d}}\left(\wt{B},\wt{U}\right)_-/\wt{G}\]
satisfying
\begin{equation}
\label{eq:Dade correspondence, CTC implies CTC non-central, 1}
\left(\wt{A}_{\wt{\sigma},\wt{\vartheta}},\wt{G}_{\wt{\sigma}},\wt{\vartheta}\right)\iso{\wt{G}}\left(\wt{A}_{\wt{\rho},\wt{\chi}},\wt{G}_{\wt{\rho}},\wt{\chi}\right)
\end{equation}
for every $(\wt{\sigma},\wt{\vartheta})\in\C^{\wt{d}}(\wt{B},\wt{U})_+$ and $(\wt{\rho},\wt{\chi})\in\wt{\Omega}_{\wt{B}}(\overline{(\wt{\sigma},\wt{\vartheta})})$ and where $\overline{(\wt{\sigma},\wt{\vartheta})}$ denotes the $\wt{G}$-orbit of $(\wt{\sigma},\wt{\vartheta})$. Applying Lemma \ref{lem:Bijections for union of blocks} and \cite[Lemma 3.4]{Spa17} we obtain an $\wt{A}_{\mathcal{D}_\varphi(B)}$-equivariant bijection
\[\wt{\Omega}:\C^{\wt{d}}\left(\mathcal{D}_\varphi(B),\wt{U},\wt{\varphi}\right)_+/\wt{G}\to\C^{\wt{d}}\left(\mathcal{D}_\varphi(B),\wt{U},\wt{\varphi}\right)_-/\wt{G}\]
that satisfies the condition on isomorphisms of character triples given in \eqref{eq:Dade correspondence, CTC implies CTC non-central, 1}. Then, we construct an $A_B$-equivariant bijection 
\[\Omega:\C^d(B,U,\varphi)_+/G\to\C^d(B,U,\varphi)_-/G\]
by combining $\wt{\Omega}$ with the bijection given by Lemma \ref{lem:Dade's correspondence, pairs}. More precisely, $\Omega$ maps the $G$-orbit of $(\sigma,\vartheta)$ to that of $(\rho,\chi)$ if $\wt{\Omega}$ maps the $\wt{G}$-orbit of $(\wt{\sigma},\wt{\vartheta})$ to that of $(\wt{\rho},\wt{\chi})$. Notice furthermore that the sets $\C^d(B,U,\varphi)_\pm$ coincide with $\underline{\C}^d(B,U,\varphi)_\pm$ since by assumption $d>d(\varphi)$. Therefore, in order to conclude the proof, it remains to show that
\begin{equation}
\label{eq:Dade correspondence, CTC implies CTC non-central, 2}
\left(A_{\sigma,\vartheta},G_\sigma,\vartheta\right)\iso{G}\left(A_{\rho,\chi},G_{\rho},\chi\right)
\end{equation}
for any $(\sigma,\vartheta)\in\C^d(B,U,\varphi)_+$ and $(\rho,\chi)\in\Omega(\overline{(\sigma,\vartheta)})$ and where here we denote by $\overline{(\sigma,\vartheta)}$ the $G$-orbit of $(\sigma,\vartheta)$. The group theoretic requirements for \eqref{eq:Dade correspondence, CTC implies CTC non-central, 2} are consequences of the analogous requirements for \eqref{eq:Dade correspondence, CTC implies CTC non-central, 1}. In fact, we know by \eqref{eq:Dade correspondence, CTC implies CTC non-central, 1} that $\wt{G}\wt{A}_{\wt{\sigma},\wt{\vartheta}}=\wt{G}\wt{A}_{\wt{\rho},\wt{\chi}}$ so that $\wt{GA_{\sigma,\vartheta}}/\wt{U}=\wt{GA_{\rho,\chi}}/\wt{U}$ and hence $GA_{\sigma,\vartheta}=GA_{\rho,\chi}$. Without loss of generality, we assume from now on that $A=GA_{\sigma,\vartheta}=GA_{\rho,\chi}$. To obtain the condition on defect groups, suppose that $D$ is a defect group of $\bl(\vartheta)$ and notice that $D(\sigma)\leq \O_p(G_\sigma)\leq D$. Then $\c_{GA_{\sigma,\vartheta}}(D)\leq GA_{\sigma,\vartheta}\cap A_\sigma=A_{\sigma,\vartheta}(G\cap A_\sigma)=A_{\sigma,\vartheta}$. The same argument also shows that $\c_{GA_{\rho,\chi}}(Q)\leq A_{\rho,\chi}$ for any defect group $Q$ of $\bl(\chi)$. We now construct projective representations associated with $(A_{\sigma,\vartheta},G_\sigma,\vartheta)$ and $(A_{\rho,\chi},G_\rho,\chi)$ following the argument used in \cite[Theorem 4.4]{Ros22}. We fix a pair of projective representations $(\wt{\mathcal{R}}_1,\wt{\mathcal{R}}_2)$ associated with the $\wt{G}$-block isomorphism of character triples \eqref{eq:Dade correspondence, CTC implies CTC non-central, 1} and let $\wt{\alpha}_i$ denote the factor set of $\wt{\mathcal{R}}_i$. By inflation we can consider $\wt{\mathcal{R}}_1$ as a projective representation of $\wh{A_{\sigma,\vartheta}}$ and $\wt{\mathcal{R}}_2$ as a projective representation of $\wh{A_{\rho,\chi}}$. More precisely, we define
\[\wt{\mathcal{R}}_1'(x):=\wt{\mathcal{R}}_1\left(U_0x\right)\hspace{15pt}\text{and}\hspace{15pt}\wt{\mathcal{R}}_2'(y):=\wt{\mathcal{R}}_2\left(U_0y\right)\]
for every $x\in\wh{A_{\sigma,\vartheta}}$ and $y\in\wh{A_{\rho,\chi}}$. As explained in the proof of \cite[Theorem 4.4]{Ros22}, we can choose projective representations $\mathcal{R}_1$ associated with $(A_{\sigma,\vartheta},G_\sigma,\vartheta)$ and $\mathcal{R}_2$ associated with $(A_{\rho,\chi},G_\rho,\chi)$ such that
\begin{equation}
\label{eq:Dade correspondence, CTC implies CTC non-central, 3}
\wh{\mathcal{R}}_1=\wh{\Pr}_1\otimes\wt{\mathcal{R}}_1'\hspace{15pt}\text{and}\hspace{15pt}\wh{\mathcal{R}}_2=\wh{\Pr}_2\otimes\wt{\mathcal{R}}_2'
\end{equation}
where $\wh{\Pr}_1$ and $\wh{\Pr}_2$ denote the restrictions of $\wh{\Pr}$ to $\wh{A_{\sigma,\vartheta}}$ and to $\wh{A_{\rho,\chi}}$ respectively while $\wh{\mathcal{R}}_1$ and $\wh{\mathcal{R}}_2$ correspond to $\mathcal{R}_1$ and $\mathcal{R}_2$ via the epimorphism $\epsilon$, that is, we define
\[\wh{\mathcal{R}}_1(x):=\mathcal{R}_1\left(\epsilon(x)\right)\hspace{15pt}\text{and}\hspace{15pt}\wh{\mathcal{R}}_2(y):=\mathcal{R}_2\left(\epsilon(y)\right)\]
for every $x\in\wh{A_{\sigma,\vartheta}}$ and $y\in\wh{A_{\rho,\chi}}$. Let $\alpha_1$ and $\alpha_2$ be the factor sets of $\mathcal{R}_1$ and of $\mathcal{R}_2$ respectively. Since $\wt{\alpha}_1$ and $\wt{\alpha}_2$ coincide via the isomorphism $\wt{A}_{\wt{\sigma},\wt{\vartheta}}/\wt{G}_{\wt{\sigma}}\simeq \wt{A}_{\wt{\rho},\wt{\chi}}/\wt{G}_{\wt{\rho}}$ it follows from \eqref{eq:Dade correspondence, CTC implies CTC non-central, 3} that $\alpha_1$ and $\alpha_2$ coincide via the isomorphism $A_{\sigma,\vartheta}/G_\sigma\simeq A_{\rho,\chi}/G_\rho$. Next, observe that by \eqref{eq:Dade correspondence, CTC implies CTC non-central, 1} the projective representations $\wt{\mathcal{R}}_1$ and $\wt{\mathcal{R}}_2$ define the same scalar function on $\c_{\wt{A}}(\wt{G})$. By the above discussion, we know that $\c_A(G)\leq \c_A(D)\cap \c_A(Q)\leq A_{\sigma,\vartheta}\cap A_{\rho,\chi}$ and hence $\c_{\wh{A}}(\wh{G})\leq \epsilon^{-1}(\c_A(G))\leq \wh{A_{\sigma,\vartheta}}\cap\wh{A_{\rho,\chi}}$ by \cite[Theorem 4.1 (d)]{Spa17}. In particular, it follows that the projective representations $\wt{\mathcal{R}}_1'$ and $\wt{\mathcal{R}}_2'$ define the same scalar function on $\c_{\wh{A}}(\wh{G})$. As a consequence of \eqref{eq:Dade correspondence, CTC implies CTC non-central, 3}, we conclude that $\wh{\mathcal{R}}_1$ and $\wh{\mathcal{R}}_2$ define the same scalar function on $\c_{\wh{A}}(\wh{G})$ and therefore the same is true for $\mathcal{R}_1$ and $\mathcal{R}_2$ since $\epsilon(\c_{\wh{A}}(\wh{G}))=\c_A(G)$ by applying once again \cite[Theorem 4.1 (d)]{Spa17}. We now fix $G\leq J\leq A$ and set $J_1:=J_{\sigma,\vartheta}$ and $J_2:=J_{\rho,\chi}$. By \cite[Theorem 3.3]{Spa17}, we obtain bijections
\[\varsigma_{J_1}:\irr(J_1\mid \vartheta)\to\irr(J_2\mid \chi)\]
and
\[\wt{\varsigma}_{\wt{J}_1}:\irr\left(\wt{J}_1\enspace\middle|\enspace \wt{\vartheta}\right)\to\irr\left(\wt{J}_2\enspace\middle|\enspace \wt{\chi}\right)\]
associated to our choices of projective representations $(\mathcal{R}_1,\mathcal{R}_2)$ and $(\wt{\mathcal{R}}_1,\wt{\mathcal{R}}_2)$ respectively. Then, proceeding as in the proof of \cite[Theorem 4.4]{Ros22} and applying Proposition \ref{prop:Dade's correspondence} (ii) instead of \cite[Theorem 4.2 (iv)]{Ros22}, we deduce that $\wt{\varsigma_{J_1}(\psi)}=\wt{\varsigma}_{\wt{J}_1}(\wt{\psi})$ for every $\psi\in\irr(J_1\mid \vartheta)$. Finally, since $\bl(\wt{\varsigma}_{\wt{J}_1}(\wt{\psi}))^{\wt{J}}=\bl(\wt{\psi})^{\wt{J}}$ by \eqref{eq:Dade correspondence, CTC implies CTC non-central, 1}, using Lemma \ref{lem:Dade's correspondence, block induction} we conclude that $\bl(\sigma_{J_1}(\psi))^J=\bl(\psi)^J$. This completes the proof of \eqref{eq:Dade correspondence, CTC implies CTC non-central, 2} and hence of the proposition.
\end{proof}

\subsection{Proof of Theorem \ref{thm:CTC equivalent to CTC non-central}}
\label{sec:Proof of equivalence of conjectures}

We now come to the proof of Theorem \ref{thm:CTC equivalent to CTC non-central}. For this, we fix $G\unlhd A$ satisfying the assumption of Theorem \ref{thm:CTC equivalent to CTC non-central} and proceed by induction on $|A:\z(G)|$. In particular, the following hypothesis holds for $G\unlhd A$.

\begin{hyp}
\label{hyp:Equivalence of conjectures}
Let $G\unlhd A$ be finite groups and suppose that Conjecture \ref{conj:CTC non-central Z} holds at the prime $p$ for every $G_1\unlhd A_1$ such that:
\begin{itemize}
\item $\mathcal{S}_p(G_1)\subseteq \mathcal{S}_p(G)$; and
\item $|A_1:\z(G_1)|<|A:\z(G)|$.
\end{itemize}
\end{hyp}

As a first step towards our proof of Theorem \ref{thm:CTC equivalent to CTC non-central}, we obtain a reduction to the case in which the character $\varphi\in\irr(U)$ considered in Conjecture \ref{conj:CTC non-central Z} is $A$-invariant.

\begin{prop}
\label{prop:CTC non-central, reduction to stable varphi}
Suppose that Hypothesis \ref{hyp:Equivalence of conjectures} holds for $G\unlhd A$ and that Conjecture \ref{conj:CTC non-central Z} fails to hold for a block $B$ of $G$ with respect to a normal $p$-subgroup $U$ of $G$, a character $\varphi\in\irr(U)$, and a non-negative integer $d$. Then $U=\O_p(G)\unlhd A$, the character $\varphi$ is $A$-invariant, and $d>d(\varphi)$.
\end{prop}

\begin{proof}
First, notice that if $U<\O_p(G)$ then Conjecture \ref{conj:CTC non-central Z} would hold with respect to $U$ by the very same argument used in the proof of \cite[Lemma 2.3]{Ros22}. Hence, our choices imply that $U=\O_p(G)\unlhd A$. Next, observe that the set $\underline{\C}^d(B,U,\varphi)$ is empty whenever $d=d(\varphi)$ and hence, given our assumption, we must have $d>d(\varphi)$. If we assume that $\varphi$ is not $A$-invariant, then  $|A_\varphi:\z(G_\varphi)|<|A:\z(G)|$. In particular, by Hypothesis \ref{hyp:Equivalence of conjectures}, we deduce that Conjecture \ref{conj:CTC non-central Z} holds with respect to $G_{\varphi,\upsilon}\unlhd A_{\varphi,\upsilon}$ for every $\upsilon\in\mathfrak{A}_{G_\varphi}(G,U)$. Hence, for every block $B'$ of $G_{\varphi,\upsilon}$, there exists an $A_{B',\varphi,\upsilon}$-equivariant bijection
\[\underline{\Theta}_{B',\upsilon}:\underline{\C}^d(B',U,\varphi)_+/G_{\varphi,\upsilon}\to\underline{\C}^d(B',U,\varphi)_-/G_{\varphi,\upsilon}\]
such that
\[\left(A_{\varphi,\upsilon,\varrho,\eta},G_{\varphi,\upsilon,\varrho},\eta\right)\iso{G_{\varphi,\upsilon}}\left(A_{\varphi,\upsilon,\varpi,\kappa},G_{\varphi,\upsilon,\varpi},\kappa\right)\]
for every $(\varrho,\eta)\in\C^d(B',U,\varphi)_+$ and $(\varpi,\kappa)\in\underline{\Theta}_{B',\upsilon}(\overline{(\varrho,\eta)})$. Since $d>d(\varphi)$, it follows that $\omega^d(B'_\varrho,U)^{\rm c}=\irr^d(B'_\varrho)$ for every $\varrho\in\mathfrak{N}(G_{\varphi,\upsilon},U)$ and therefore $\underline{\C}^d(B',U,\varphi)_\pm=\C^d(B',U,\varphi)_\pm$. As a consequence of this observation and the existence of the above bijections, we conclude that Hypothesis \ref{hyp:Product chains and stable reduction} is satisfied. Now, applying Proposition \ref{prop:Product of chains reduuction with pairs}, we get an $A_{B,\varphi}$-equivariant bijection
\[\Omega_B:\C^d(B,U,\varphi)_+/G_\varphi\to\C^d(B,U,\varphi)_-/G_\varphi\]
such that
\[\left(A_{\sigma,\vartheta},G_\sigma,\vartheta\right)\iso{G}\left(A_{\rho,\chi},G_\rho,\chi\right)\]
for every $(\sigma,\vartheta)\in\C^d(B,U,\varphi)_+$ and $(\rho,\chi)\in\Omega_{B}(\overline{(\sigma,\vartheta)})$. Noticing once again that $\underline{\C}^d(B,U,\varphi)_\pm=\C^d(B,U,\varphi)_\pm$ (because $d>d(\varphi)$), we conclude that Conjecture \ref{conj:CTC non-central Z} holds for our choices of $B$, $U$, $\varphi$ and $d$. This contradiction shows that $\varphi$ must be $A$-invariant.
\end{proof}

We are now able to prove Theorem \ref{thm:CTC equivalent to CTC non-central}.

\begin{proof}[Proof of Theorem \ref{thm:CTC equivalent to CTC non-central}]
Suppose that Conjecture \ref{conj:CTC non-central Z} fails to hold for $G\unlhd A$ with respect to a block $B$ of $G$, a $p$-subgroup $U\leq G$, an irreducible character $\varphi\in\irr(U)$, and a non-negative integer $d$. Proceeding by induction on $|A:\z(G)|$, we deduce that Hypothesis \ref{hyp:Equivalence of conjectures} holds for $G\unlhd A$. Then, Proposition \ref{prop:CTC non-central, reduction to stable varphi} implies that $U=\O_p(G)\unlhd A$, that $\varphi$ is $A$-invariant, and that $d>d(\varphi)$. We are now in the situation considered in Hypothesis \ref{hyp:Dade's correspondence} and we can apply the results obtained in Section \ref{sec:Dade correspondence}. Since $\mathcal{S}_p(\wt{G})\subseteq \mathcal{S}_p(G)$ and $|\wt{A}:\z(\wt{G})|\leq |A:\z(G)|$, our assumption implies that Conjecture \ref{conj:CTC} holds for $\wt{G}\unlhd \wt{A}$. But then applying Proposition \ref{prop:Dade correspondence, CTC implies CTC non-central} we deduce that $d\leq d(\varphi)$, a contradiction. This shows that Conjecture \ref{conj:CTC non-central Z} holds for $G\unlhd A$ with respect to any choice of $B$, $U$, $\varphi$, and $d$, and we are done.
\end{proof}

\section{A cancellation theorem for the Character Triple Conjecture}
\label{sec:Cancellation}

In this section, we prove a cancellation theorem for the Character Triple Conjecture inspired by work of Staszewski--Robinson \cite{Rob-Sta90} and Robinson \cite{Rob02} on the Alperin Weight Conjecture and Dade's Conjecture respectively. In order to state this result, we need to introduce some additional notation. Let $Z\leq N\unlhd G$ be finite groups with $Z$ a central $p$-subgroup of $G$. For any block $B$ of $G$ and any non-negative integer $d$, we define the subset $\C^d_N(B,Z)$ of pairs $(\sigma,\vartheta)$ of $\C^d(B,Z)$ such that $\sigma$ is a $p$-chain of $N$, that is,
\[\C^d_N(B,Z):=\left\lbrace (\sigma,\vartheta)\in\C^d(B,Z) \enspace\middle|\enspace D(\sigma)\leq N\right\rbrace\]
and denote its complement by
\[\D^d_N(B,Z):=\C^d(B,Z)\setminus\C^d_N(B,Z).\]
As usual, using the notion of length of $p$-chains, we define $\C^d_N(B,Z)_\pm:=\C^d_N(B,Z)\cap \C^d(B,Z)_\pm$ and $\D^d_N(B,Z)_\pm:=\D^d_N(B,Z)\cap \C^d(B,Z)_\pm$.
Suppose now that $G$ is embedded as a normal subgroup in a finite group $A$ that normalises $N$. Notice then that the sets $\C^d_N(B,Z)_\pm$ and $\mathcal{D}_N^d(B,Z)_\pm$ are $\n_A(Z)_B$-stable. In this case, the task of proving the Character Triple Conjecture can be subdivided into showing that there exist $\n_A(Z)_B$-equivariant bijections
\begin{equation}
\label{eq:Cancellation subset}
\C^d_N(B,Z)_+/G\to \C^d_N(B,Z)_-/G
\end{equation}
and
\begin{equation}
\label{eq:Cancellation complement}
\D^d_N(B,Z)_+/G\to \D^d_N(B,Z)_-/G
\end{equation}
inducing $G$-block isomorphisms of character triples. The aim of this section is to show how to construct the latter bijection \eqref{eq:Cancellation complement} in the case where $G$ is a minimal counterexample to the conjecture and $B$ covers a block of $N$ whose defect groups strictly contain $Z$. In Section \ref{sec:Reduction}, in order to prove Theorem \ref{thm:Main reduction}, we will apply the above discussion to the case where $N$ coincides with the generalised Fitting subgroup and then conclude by showing how to construct the bijection \eqref{eq:Cancellation subset} by assuming the conjecture for all quasi-simple components of $G$.

The following hypothesis makes precise what we mean by a minimal counterexample to Conjecture \ref{conj:CTC}. We point out that the results of \cite[Section 5]{Ros22} hold for any pair $G\unlhd A$ satisfying Hypothesis \ref{hyp:Minimal counterexample} below. This observation will be used without further reference.

\begin{hyp}
\label{hyp:Minimal counterexample}
Let $G\unlhd A$ be finite groups and suppose that Conjecture \ref{conj:CTC} holds at the prime $p$ for every $G_1\unlhd A_1$ such that:
\begin{itemize}
\item $\mathcal{S}_p(G_1)\subseteq \mathcal{S}_p(G)$; and
\item $|A_1:\z(G_1)|<|A:\z(G)|$.
\end{itemize}
\end{hyp}

We are now able to state the main result of this section.

\begin{theo}
\label{thm:Minimal counterexample}
Assume Hypothesis \ref{hyp:Minimal counterexample} for $G\unlhd A$ and suppose that Conjecture \ref{conj:CTC} fails to hold with respect to some $p$-subgroup $Z\leq \z(G)$, a $p$-block $B$ of $G$ whose defect groups strictly contain $Z$, and a non-negative integer $d$. If $Z\leq N\unlhd A$ is a non-central subgroup of $G$, then $B$ covers some block of $N$ with defect groups strictly containing $Z$ and there exists an $\n_A(Z)_B$-equivariant bijection
\[\Lambda_{N,Z}^B:\mathcal{D}^d_N(B,Z)_+/G\to\D^d_N(B,Z)_-/G\]
such that
\[\left(A_{\sigma,\vartheta},G_\sigma,\vartheta\right)\iso{G}\left(A_{\rho,\chi},G_\rho,\chi\right)\]
for every $(\sigma,\vartheta)\in\D^d_N(B,Z)_+$ and any $(\rho,\chi)\in\Lambda_{N,Z}^B(\overline{(\sigma,\vartheta)})$.
\end{theo}

\subsection{Brauer pairs and non-stable blocks}

In this section we show that if Hypothesis \ref{hyp:Minimal counterexample} holds for $G\unlhd A$ and $B$ is a block of $G$ for which Conjecture \ref{conj:CTC} fails to holds, then every block of the normal subgroup $N$ of $G$ that is covered by $B$ is $A$-invariant. For this purpose, we consider a decomposition of the set $\C^d(B,Z)$ given by considering $B$-Brauer pairs.

Recall that $(P,b_P)$ is a Brauer pair of $G$ if $P$ is a $p$-subgroup of $G$ and $b_P$ is a $p$-block of $\c_G(P)$. There is a partial order relation defined on the set of Brauer pairs of $G$, here denoted by $\subseteq$, with the following property: for any Brauer pair $(P,b_P)$ and any $p$-subgroup $Q$ of $P$, there exists a unique Brauer pair $(Q,b_Q)$ such that $(Q,b_Q)\subseteq(P,b_P)$. In particular, any Brauer pair $(P,b_P)$ determines a uniquely defined block $B$ of $G$ such that $(\{1\},B)\subseteq (P,b_P)$. In this case, we say that $(P,b_P)$ is a $B$-Brauer pair. The order relation $\subseteq$ restricts to the subset of all $B$-Brauer pairs. Then, all maximal $B$-Brauer pairs are $G$-conjugate and of the form $(D,b_D)$ for a defect group $D$ of $B$ (see \cite[Theorem 3.10]{Alp-Bro79}). We refer the reader to \cite{Alp-Bro79} for a more detailed introduction to Brauer pairs.

For any central $p$-subgroup $Z$ of $G$ and any block $B$ of $G$ we define the set
\[\mathfrak{N}_B(G,Z):=\left\lbrace (\sigma,b_\sigma)\enspace
\middle|\enspace\sigma\in\mathfrak{N}(G,Z), (\{1\},B)\subseteq(D(\sigma),b_\sigma)\right\rbrace\]
where we recall that $D(\sigma)$ denotes the final term of the chain $\sigma$. By the properties of Brauer pairs described in the previous paragraph, any $(\sigma,b_\sigma)$ defines a unique chain of Brauer pairs $(Z,b_Z)\subseteq (D_1,b_{D_1})\subseteq \dots\subseteq (D(\sigma),b_\sigma)$. Notice that $G$ acts on the set $\mathfrak{N}_B(G,Z)$ and that the stabiliser $G_{(\sigma,b_\sigma)}$ coincides with the inertial subgroup of $b_\sigma$ in the $p$-chain stabiliser $G_\sigma$. Since $D(\sigma)$ is a normal $p$-subgroup of $G_\sigma$ (recall that we are only considering normal $p$-chains), any block of $G_\sigma$ is $\c_G(D(\sigma))$-regular and therefore the block $\beta_{(\sigma,b_\sigma)}:=b_\sigma^{G_\sigma}$ is defined and is the unique block of $G_\sigma$ covering $b_\sigma$ (see \cite[Theorem 9.19 and Lemma 9.20]{Nav98}). We now define the set of triples
\[\C_\bullet^d(B,Z):=\left\lbrace(\sigma,b_\sigma,\vartheta)\enspace\middle|\enspace(\sigma,b_\sigma)\in\mathfrak{N}_B(G,Z), \vartheta\in\irr^d\left(\beta_{(\sigma,b_\sigma)}\right)\right\rbrace\] 
and denote by $\C^d_\bullet(B,Z)/G$ the set of orbits in $\C^d_\bullet(B,Z)$ under the action of $G$ induced by conjugation. As usual we denote by $\overline{(\sigma,b_\sigma,\vartheta)}$ the $G$-orbit of $(\sigma,b_\sigma,\vartheta)$. The following lemma establishes a connection between $\C^d(B,Z)$ and $\C^d_\bullet(B,Z)$ and allows us to reformulate the Character Triple Conjecture in terms of the set $\C_\bullet^d(B,Z)$.

\begin{lem}
\label{lem:Brauer pairs version of CTC}
Let $G\unlhd A$ and consider a block $B$ of $G$ and a central $p$-subgroup $Z$ of $G$. Then the map
\begin{align*}
\C^d_\bullet(B,Z)/G&\to\C^d(B,Z)/G
\\
\overline{(\sigma,b_\sigma,\vartheta)}&\mapsto\overline{(\sigma,\vartheta)}
\end{align*}
is an $\n_A(Z)_B$-equivariant bijection.
\end{lem}

\begin{proof}
By definition, we know that the map commutes with the action of $\n_A(Z)_B$. Recall by the above discussion that every block of $G_\sigma$ is $\c_G(D(\sigma))$-regular and therefore there is a partition
\[\irr^d(B_\sigma)=\coprod\limits_{b_\sigma}\irr^d\left(\beta_{(\sigma,b_\sigma)}\right)\]
where the union runs over the blocks $b_\sigma$ of $\c_G(D(\sigma))$ up to $G_\sigma$-conjugation such that $(\{1\},B)\subseteq (D(\sigma),b_\sigma)$. In particular, for every $(\sigma,\vartheta)\in\C^d(B,Z)$ there is a block $b_\sigma$ such that $(\sigma,b_\sigma,\vartheta)\in\C^d_\bullet(B,Z)$ and $\overline{(\sigma,b_\sigma,\vartheta)}$ maps to $\overline{(\sigma,\vartheta)}$. This shows that the map is surjective. On the other hand, if $(\sigma_1,b_1,\vartheta_1)$ and $(\sigma_2,b_2,\vartheta_2)$ are triples in $\C^d_\bullet(B,Z)$ such that $\overline{(\sigma_1,\vartheta_1)}=\overline{(\sigma_2,\vartheta_2)}$, then there exists $g\in G$ such that $(\sigma_1,\vartheta_1)=(\sigma_2,\vartheta_2)^g$. Now $\vartheta_1$ lies both in the block $\beta_{(\sigma_1,b_1)}$ and in the block $\beta_{(\sigma_2,b_2)^g}$ and it follows that $b_1$ and $b_2^g$ are conjugate by an element of $G_{\sigma_1}$. Since $\sigma_1$ and $\vartheta_1$ are invariant under the action of $G_{\sigma_1}$ we obtain $\overline{(\sigma_1,b_1,\vartheta_1)}=\overline{(\sigma_2,b_2,\vartheta_2)}$ and this shows that the map is injective.
\end{proof}

Our next proposition shows that the set $\C^d_\bullet(B,Z)$ is well behaved with respect to the Fong--Reynolds correspondence. For this, we adapt the argument used in \cite[Section 3.3]{Rob02} to our situation. In particular, we show that Robinson's argument is compatible with the action of automorphisms.

\begin{prop}
\label{prop:Stable reduction, bijection}
Let $G\unlhd A$ and consider a block $B$ of $G$ and a central $p$-subgroup $Z$ of $G$. Let $N$ be a normal subgroup of $A$ contained in $G$ and $b$ a block of $N$ covered by $B$. Then there exists an $\n_A(Z)_{B,b}$-equivariant bijection
\[\Upsilon:\C^d_\bullet(B,Z)/G\to\C^d_\bullet(B',Z)/G_b\]
where $B'$ is the block of $G_b$ covering $b$ and corresponding to $B$ via the Fong--Reynolds correspondence.
\end{prop}

\begin{proof}
Without loss of generality we assume that $Z$ and $B$ are invariant under the action of $A$. Let $A':=A_b$ and $H':=A'\cap H$ for every $H\leq A$. In particular $A=GA'$. Notice that $Z\leq G'$ and that $A'$ acts on $\C^d_\bullet(B',Z)/G'$ because $B'$ is $A'$-invariant. As explained in \cite[p.210-211]{Rob02}, there exists a maximal $B$-Brauer pair $(D,b_D)$ and a maximal $B'$-Brauer pair $(D,b'_D)$ such that there exists a bijection
\begin{equation}
\label{eq:Stable reduction, bijection 1}
(Q,b_Q)\mapsto(Q,b_Q')
\end{equation}
between the set of $B$-Brauer pairs $(Q,b_Q)\subseteq (D,b_D)$ and the $B'$-Brauer pairs $(Q,b_Q')\subseteq (D,b_D')$. In addition, given $(Q,b_Q)\subseteq (D,b_D)$ and $g\in G$ such that $(Q,b_Q)^g\subseteq (D,b_D)$, then we can write $g=cx$ for some $c\in\c_G(Q)$ and $g\in G'$ (this relies on \cite[Theorem 4.10 and Proposition 4.24]{Alp-Bro79}).

If we define $\mathcal{N}_{(D,b_D)}$ to be the set of pairs $(\sigma,b_\sigma)\in\mathfrak{N}_B(G,Z)$ such that $(D(\sigma),b_\sigma)\subseteq (D,b_D)$ and similarly $\mathcal{N}_{(D,b'_D)}$ to be the set of pairs $(\sigma',b'_{\sigma'})\in\mathfrak{N}_{B'}(G',Z)$ such that $(D(\sigma'),b'_{\sigma'})\subseteq (D,b_D')$, then the correspondence \eqref{eq:Stable reduction, bijection 1} induces a bijection
\begin{align}
\label{eq:Stable reduction, bijection 2}
\mathcal{N}_{(D,b_D')}&\to\mathcal{N}_{(D,b_D)}
\\
(\sigma,b_\sigma)&\mapsto(\sigma,b_\sigma').\nonumber
\end{align} 
Given $(\sigma,b_\sigma)$ and $(\sigma,b'_\sigma)$ as above, we recall that there exists a unique block $\beta_{(\sigma,b_\sigma)}$ of $G_\sigma$ covering $b_\sigma$ and a unique block $\beta_{(\sigma,b'_\sigma)}'$ of $G'_\sigma$ covering $b_\sigma'$. We claim that $\beta_{(\sigma,b_\sigma)}=(\beta'_{(\sigma,b_\sigma')})^{G_\sigma}$ and that induction of characters yields a bijection between the irreducible characters of $\beta_{(\sigma,b_\sigma)}$ and those of $\beta'_{(\sigma,b_\sigma')}$. Notice that in this case the bijection would preserve the defect of characters (this is because a character and its induction have the same defect by the usual degree formula for induced characters). To prove the claim, consider the Brauer pairs $(D(\sigma),b_{\sigma})$ and $(D(\sigma),b_{\sigma}')$ and observe that by the construction of \eqref{eq:Stable reduction, bijection 1} there exists a block $b^\star_{\sigma}$ of $\c_N(D(\sigma))$ covered by $b_{\sigma}$. Moreover, the stabiliser of $b_{\sigma}^\star$ in $A_\sigma$ is contained in $A'_\sigma$ (see \cite[p.211]{Rob02}) and the block $b_\sigma'$ is defined to be the unique block of $\c_{G'}(D(\sigma))$ that covers $b^\star_{\sigma}$ and induces $b_\sigma$ (see \cite[Theorem 9.14]{Nav98}). Now $\beta_{(\sigma,b_\sigma)}$ and $\beta'_{(\sigma,b_\sigma')}$ both cover $b^\star_\sigma$ and our claim follows from the Fong--Reynolds correspondence \cite[Theorem 9.14]{Nav98}.

The above description also shows that $A_{(\sigma,b_\sigma)}=\c_G(D(\sigma))A'_{(\sigma,b_\sigma')}$ by a Frattini argument. In fact, since $b_\sigma$ covers $b_\sigma^\star$, we have $A_{(\sigma,b_\sigma)}=\c_G(D(\sigma))A_{(\sigma,b_\sigma),b^\star_\sigma}$ and recalling that $A_{\sigma,b^\star_\sigma}\leq A'$ we deduce that $A_{(\sigma,b_\sigma),b^\star_\sigma}=A'_{(\sigma,b_\sigma')}$ by the uniqueness property of \cite[Theorem 9.14]{Nav98}. In particular, observe that $A'_{(\sigma,b_\sigma)}=A'_{(\sigma,b_\sigma')}$. Therefore, we have constructed a defect preserving $A'_{(\sigma,b_\sigma')}$-equivariant bijection
\begin{align}
\label{eq:Stable reduction, bijection 3}
\irr\left(\beta'_{(\sigma,b_\sigma')}\right)&\to\irr\left(\beta_{(\sigma,b_\sigma)}\right)
\\
\vartheta'&\mapsto (\vartheta')^{G_\sigma}.\nonumber
\end{align}
Next, we show that for every $(\sigma,b_\sigma)\in\mathcal{N}_{(D,b_D)}$ and $a\in A$ such that $(\sigma,b_\sigma)^a\in\mathcal{N}_{(D,b_D)}$ there exists $a'\in A'$ such that $(\sigma,b_\sigma)^a=(\sigma,b_\sigma)^{a'}$. For this, it is enough to prove that if $(Q,b_Q)\subseteq (D,b_D)$ and $a\in A$ satisfy $(Q,b_Q)^a\subseteq (D,b_D)$ then $(Q,b_Q)^a=(Q,b_Q)^{a'}$ for some $a'\in A'$. Recalling that $A=GA'$ and that $A'_{(D,b_D)}=A'_{(D,b_D')}$, we can write $A=GA'=G(G'A'_{(D,b_D')})=GA'_{(D,b_D)}$ by using the fact that all maximal $B'$-Brauer pairs are $G'$-conjugate. Then there exists $g\in G$ and $a_1'\in A'_{(D,b_D)}$ such that $a=a_1'g$. Now we have $(Q,b_Q)^{a_1'}\subseteq (D,b_D)$ and $g\in G$ such that $((Q,b_Q)^{a_1'})^g=(Q,b_Q)\subseteq (D,b_D)$ and hence by the properties of the correspondence \eqref{eq:Stable reduction, bijection 1} we deduce that there exist $c\in \c_G(Q^{a_1'})$ and $x\in G'$ such that $g=cx$. If we define $a':=a_1'x\in A'$, then we obtain $(Q,b_Q)^a=(Q,b_Q)^{a'}$ as desired.

We now fix, as we may, an $A$-transversal $\mathcal{T}$ in $\mathfrak{N}_B(G,Z)$ contained in $\mathcal{N}_{(D,b_D)}$ and denote by $\mathcal{T}'$ the subset of $\mathcal{N}_{(D,b_D')}$ corresponding to $\mathcal{T}$ under \eqref{eq:Stable reduction, bijection 2}. By the previous paragraph we deduce that $\mathcal{T}'$ is an $A'$-transversal in $\mathfrak{N}_{B'}(G',Z)$. Furthermore, for every $(\sigma,b_\sigma')\in\mathcal{T'}$ corresponding to $(\sigma,b_\sigma)\in\mathcal{T}$ we choose an $A'_{(\sigma,b_\sigma')}$-transversal $\mathcal{T}'_{(\sigma,b_\sigma')}$ in $\irr^d(\beta'_{(\sigma,b_\sigma')})$ and denote by $\mathcal{T}_{(\sigma,b_\sigma)}$ the subset of $\irr^d(\beta_{(\sigma,b_\sigma)})$ obtained by inducing the characters of $\mathcal{T}'_{(\sigma,b_\sigma')}$ to $G_\sigma$. By the properties of \eqref{eq:Stable reduction, bijection 3} it follows that $\mathcal{T}_{(\sigma,b_\sigma)}$ is an $A_{(\sigma,b_\sigma)}$-transversal in $\irr^d(\beta_{(\sigma,b_\sigma)})$ and we conclude that the sets
\[\mathcal{S}:=\left(\lbrace\overline{(\sigma,b_\sigma,\vartheta)}\enspace\middle|\enspace(\sigma,b_\sigma)\in\mathcal{T}, \vartheta\in\mathcal{T}_{(\sigma,b_\sigma)}\right\rbrace\]
and
\[\mathcal{S}':=\left(\lbrace\overline{(\sigma,b_\sigma',\vartheta')}\enspace\middle|\enspace(\sigma,b_\sigma')\in\mathcal{T}', \vartheta'\in\mathcal{T}_{(\sigma,b_\sigma')}'\right\rbrace\]
are $A'$-transversals in $\C^d_\bullet(B,Z)/G$ and $\C^d_\bullet(B',Z)/G'$ respectively (recall once more that $A=GA'$) in bijection with each others. Although clear from the context, we remark that the notation $\overline{(\sigma,b_\sigma',\vartheta')}$ represents the $G'$-orbit and not the $G$-orbit. We conclude by defining
\[\Upsilon\left(\overline{(\sigma,b_\sigma,\vartheta)}^a\right):=\overline{(\sigma,b_\sigma',\vartheta')}^a\]
for every $\overline{(\sigma,b_\sigma,\vartheta)}\in\mathcal{S}$ corresponding to $\overline{(\sigma,b_\sigma',\vartheta')}\in\mathcal{S}'$ and every $a\in A'$.
\end{proof}

Next, we show that the bijection given by Proposition \ref{prop:Stable reduction, bijection} preserves isomorphisms of characters triples.

\begin{cor}
\label{cor:Stable reduction, character triples}
Consider the setting of Proposition \ref{prop:Stable reduction, bijection}. For $i=1,2$ consider $\overline{(\sigma_i,b_{\sigma_i},\vartheta_i)}\in\C^d_\bullet(B,Z)/G$ and write $\Upsilon(\overline{(\sigma_i,b_{\sigma_i},\vartheta_i)})=\overline{(\sigma_i',b_{\sigma_i'}',\vartheta_i')}$. If
\[\left(A_{b,\sigma_1',\vartheta_1'},G_{b,\sigma_i'},\vartheta_1'\right)\iso{G_b}\left(A_{b,\sigma_2',\vartheta_2'},G_{b,\sigma_2'},\vartheta_2'\right)\]
then
\[\left(A_{\sigma_1,\vartheta_1},G_{\sigma_i},\vartheta_1\right)\iso{G}\left(A_{\sigma_2,\vartheta_2},G_{\sigma_2},\vartheta_2\right).\]
\end{cor}

\begin{proof}
First notice that we can write $\Upsilon(\overline{(\sigma_i,b_{\sigma_i},\vartheta_i)})=\overline{(\sigma_i,b_{\sigma_i}',\vartheta_i')}$ thanks to the construction given above. Then the result follows immediately from \cite[Proposition 2.8]{Ros22}.
\end{proof}

We can finally prove the main result of this section.

\begin{prop}
\label{prop:Stable reduction}
Suppose that Hypothesis \ref{hyp:Minimal counterexample} holds for $G\unlhd A$ and that Conjecture \ref{conj:CTC} fails to hold for a block $B$ of $G$. If $N$ is a normal subgroup of $A$ contained in $G$ and $b$ is a block of $N$ covered by $B$, then $b$ is $A$-invariant.
\end{prop}

\begin{proof}
By assumption there exist a $p$-subgroup $Z\leq \z(G)$ strictly contained in the defect groups of $B$ and an integer $d\geq 0$ such that Conjecture \ref{conj:CTC} fails to hold with respect to $Z$, $G$, $A$, $B$, and $d$. Let $A':=A_b$ and define $H':=H\cap A'$ for every $H\leq A$. Notice that $\mathcal{S}_p(G')\subseteq \mathcal{S}_p(G)$ and that $\z(G)\leq \z(G')$. If $A'<A$, then $|A':\z(G')|<|A:\z(G)|$ and we obtain from Hypothesis \ref{hyp:Minimal counterexample} that Conjecture \ref{conj:CTC} holds for $G'\unlhd A'$. Let $B'$ be the block of $G'$ corresponding to $B$ via the Fong--Reynolds correspondence over $b$ and notice that $Z\leq \z(G')$. Since $B'$ and $B$ have common defect groups, we deduce that $Z$ is strictly contained in the defect groups of $B'$. Then, there exists an $\n_{A'}(Z)_{B'}$-equivariant bijection
\[\Omega':\C^d(B',Z)_+/G'\to\C^d(B',Z)_-/G'\]
such that
\begin{equation}
\label{eq:Stable reduction 1}
\left(A'_{\sigma',\vartheta'},G'_{\sigma'},\vartheta'\right)\iso{G'}\left(A'_{\rho',\chi'},G'_{\rho'},\chi'\right)
\end{equation}
for every $(\sigma',\vartheta')\in\C^d(B',Z)_+$ and every $(\rho',\chi')\in\Omega'(\overline{(\sigma',\vartheta')})$. Denote by $\C^d_\bullet(B',Z)_\pm$ and $\C^d_\bullet(B,Z)_\pm$ the subsets of $\C^d_\bullet(B',Z)$ and $\C^d_\bullet(B,Z)$ respectively whose elements have a $p$-chain belonging to $\mathfrak{N}(G',Z)_\pm$ and $\mathfrak{N}(G,Z)_\pm$ respectively. Combining $\Omega'$ with the bijection given by Lemma \ref{lem:Brauer pairs version of CTC} we obtain an $\n_{A'}(Z)_{B'}$-equivariant bijection
\[\Omega'_\bullet:\C^d_\bullet(B',Z)_+/G'\to\C^d_\bullet(B',Z)_-/G'\]
from which we deduce, thanks to Proposition \ref{prop:Stable reduction, bijection}, that the map
\[\Omega_\bullet:\C^d_\bullet(B,Z)_+/G\to\C^d_\bullet(B,Z)_-/G\]
given by
\[\Omega_\bullet\left(\overline{(\sigma,\vartheta)}\right):=\Upsilon^{-1}\circ\Omega'_\bullet\circ\Upsilon\left(\overline{(\sigma,\vartheta)}\right)\]
for every $(\sigma,\vartheta)\in\C^d_\bullet(B,Z)_+$ is an $\n_{A'}(Z)_{B'}$-equivariant bijection. Noticing that $\n_A(Z)_B=G\n_{A'}(Z)_{B'}$, we deduce that $\Omega_\bullet$ is $\n_A(Z)_B$-equivariant and hence Lemma \ref{lem:Brauer pairs version of CTC} yields an $\n_A(Z)_B$-equivariant bijection
\[\Omega:\C^d(B,Z)_+/G\to\C^d(B,Z)_-/G.\]
Furthermore, applying Corollary \ref{cor:Stable reduction, character triples} together with \eqref{eq:Stable reduction 1} we deduce that
\[\left(A_{\sigma,\vartheta},G_{\sigma},\vartheta\right)\iso{G}\left(A_{\rho,\chi},G_{\rho},\chi\right)\]
for every $(\sigma,\vartheta)\in\C^d(B,Z)_+$ and every $(\rho,\chi)\in\Omega(\overline{(\sigma,\vartheta)})$. This contradicts our assumptions and therefore $A'=A$.
\end{proof}

\subsection{Covering blocks of non-central normal subgroups}

As before, consider finite groups $Z\leq N\leq G\unlhd A$ with $N$ normal in $A$. In this section, we show that if Hypothesis \ref{hyp:Minimal counterexample} holds and Conjecture \ref{conj:CTC} fails to hold with respect to a block $B$ of $G$, then either $N$ is contained in the centre of $G$ or $Z$ is strictly contained in the defect groups of any block of $N$ covered by $B$.

For this, suppose that $B$ covers a block $b$ of $N$ with defect group $Z$. In this case, notice that $Z$ is central in $N$ and according to \cite[Theorem 9.12]{Nav98} there exists a unique character $\mu$ belonging to $b$ and satisfying $Z\leq \ker(\mu)$. Observe furthermore that $b$ is a nilpotent block (see, for instance, \cite[(1.ex.1)]{Bro-Pui80}) and that $b$ is $A$-invariant according to Proposition \ref{prop:Stable reduction}. The structure of the covering block $B$ can then be understood by the work of K\"ulshammer--Puig \cite{Kul-Pui90} and Puig--Zhou \cite{Pui-Zho12} on extensions of nilpotent blocks. In particular, there exists an extension $L$ of $G/N$ by (an isomorphic copy of) $Z$, together with a factor set $\alpha$ of $L$, and an isomorphism $\z(\mathcal{O}Gb)\simeq \z(\mathcal{O}_\alpha L)$ where $\mathcal{O}_\alpha L$ denotes the twisted group algebra of $L$ by $\alpha$. Moreover, assuming that $\mu$ is $G$-invariant, we deduce that the factor set $\alpha$ is trivial and hence that the block $B$ of $G$ corresponds to a unique block, say $C$, of $L$. Similarly, by replacing $G$ with $A$, we obtain an extension $M$ of $A/N$ by $Z$ such that $L\unlhd M$. Then, if we define $\overline{A}:=A_B/G\simeq M_C/L$, the results of \cite{Pui-Zho12} (see also \cite[Theorem 1.1]{Coc-Mar-Tod20}) show that the block extensions $\mathcal{O}A_BB$ and $\mathcal{O}M_CC$ are $\overline{A}$-graded basic Morita equivalent. We now show that if Conjecture \ref{conj:CTC} holds for the block $C$, then it holds for the block $B$.

\begin{lem}
\label{lem:Kulshammer-Puig and CTC}
Let $G\unlhd A$ be finite groups, $B$ a block of $G$ and suppose that $Z=\O_p(G)\leq$ $\z(G)$. Let $Z\leq N\leq G$ with $N\unlhd A$ and consider a block $b$ of $N$ covered by $B$ and with defect group $Z$. Assume that the unique character $\mu$ belonging to $b$ and satisfying $Z\leq \ker(\mu)$ extends to $A$ and consider the K\"ulshammer--Puig extension $M$ of $A/N$ by $Z$ and the analogue extension $L$ of $G/N$ by $Z$. Let $C$ be the unique block of $L$ corresponding to $B$. If Conjecture \ref{conj:CTC} holds for $C$ with respect to $Z\leq L\unlhd M$, then Conjecture \ref{conj:CTC} holds for $B$ with respect to $Z\leq G\unlhd A$.
\end{lem}

\begin{proof}
Using the argument given in \cite[p.213-214]{Rob02}, we know that there exists a bijection between the sets $\mathcal{N}_B(G,Z)/G$ and $\mathcal{N}_C(L,Z)/L$ sending the $G$-orbit of $(\sigma,b_\sigma)$ to the $L$-orbit of $(\wt{\sigma},\wt{b}_{\wt{\sigma}})$. Moreover, this bijection commutes with the action of $A_B$ and $M_C$ and preserves the length of corresponding $p$-chains, that is, $|\sigma|=|\wt{\sigma}|$. Consider now the block $\beta_{(\sigma,b_\sigma)}:=b_\sigma^{G_\sigma}$ and fix a character $\vartheta\in\irr(\beta_{(\sigma,b_\sigma)})$ with defect $d(\vartheta)=d$, so that the triple $(\sigma,b_\sigma,\vartheta)$ belongs to $\C_\bullet^d(B,Z)$. Let $b(\sigma)$ be the unique block of $ND(\sigma)$ that covers $b$, where we recall once again that $D(\sigma)$ is the largest term of the $p$-chain $\sigma$. Since $b$ is nilpotent and $A$-invariant, it follows that $b(\sigma)$ is nilpotent and has $D(\sigma)$ as defect group (see \cite[Theorem 2]{Cab87}). Furthermore, as shown in \cite[p.214]{Rob02}, we have $\n_{NG_\sigma}(D(\sigma))=G_\sigma$ and the induced block $\beta'_{(\sigma,b_\sigma)}:=(\beta_{(\sigma,b_\sigma)})^{NG_\sigma}$ covers $b(\sigma)$. It then follows that $\beta_{(\sigma,b_\sigma)}$ and $\beta'_{(\sigma,b_\sigma)}$ are Harris--Kn\"orr correspondents over the nilpotent block $b(\sigma)$. If we set $A(\sigma):=A_\sigma/G_\sigma\simeq NA_\sigma/NG_\sigma$, then we deduce that $\mathcal{O}NA_\sigma b(\sigma)$ is $A(\sigma)$-graded basic Morita equivalent to $\mathcal{O}A_\sigma b(\sigma)_0$, where $b(\sigma)_0$ is the Brauer correspondent of $b(\sigma)$ in $D(\sigma)\n_N(D(\sigma))$, and Corollary \ref{cor:Bijections over nilpotent blocks} yields a defect preserving $A_{\sigma,\beta_{(\sigma,b_\sigma)}}$-equivariant bijection between $\irr(\beta_{(\sigma,b_\sigma)})$ and $\irr(\beta'_{(\sigma,b_\sigma)})$ such that, if $\vartheta$ corresponds to $\vartheta'$, then we have
\[\left(NA_{\sigma,\vartheta'},NG_\sigma,\vartheta'\right)\iso{NG_\sigma}\left(A_{\sigma,\vartheta},G_\sigma,\vartheta\right)\]
which implies
\begin{equation}
\label{eq:Kulshammer-Puig and CTC, 1}
\left(NA_{\sigma,\vartheta'},NG_\sigma,\vartheta'\right)\iso{G}\left(A_{\sigma,\vartheta},G_\sigma,\vartheta\right)
\end{equation}
according to \cite[Lemma 2.11]{Ros22}. Next, let $(\wt{\sigma},\wt{b}_{\wt{\sigma}})$ be an element of $\mathcal{N}_C(L,Z)$ whose $L$-orbit corresponds to the $G$-orbit of $(\sigma,b_\sigma)$. Then, the induced block $\wt{\beta}_{(\wt{\sigma},\wt{b}_{\wt{\sigma}})}:=(\wt{b}_{\wt{\sigma}})^{L_{\wt{\sigma}}}$ is the unique block of the extension $L_{\wt{\sigma}}$ of $NG_\sigma/N$ by $Z$ corresponding to the block $\beta'_{(\sigma,b_\sigma)}$ and, as explained above and noticing that $A(\sigma,b_\sigma):=NA_{\sigma,\beta'_{\sigma,b_\sigma}}/NG_\sigma\simeq M_{\wt{\sigma},\wt{\beta}_{(\wt{\sigma},\wt{b}_{\wt{\sigma}})}}/L_{\wt{\sigma}}$, there exists an $A(\sigma,b_\sigma)$-graded basic Morita equivalence between $\mathcal{O}NA_{\sigma,\beta'_{(\sigma,b_\sigma)}}\beta'_{(\sigma,b_\sigma)}$ and $\mathcal{O}M_{\wt{\sigma},\wt{\beta}_{(\wt{\sigma},\wt{b}_{\wt{\sigma}})}}\wt{\beta}_{(\wt{\sigma},\wt{b}_{\wt{\sigma}})}$. In particular, we obtain a defect preserving and $A(\sigma,b_\sigma)$-equivariant bijection between the character sets $\irr(\beta'_{(\sigma,b_\sigma)})$ and $\irr(\wt{\beta}_{(\wt{\sigma},\wt{b}_{\wt{\sigma}})})$. Then, if we denote by $\wt{\vartheta}$ the character of $\wt{\beta}_{(\wt{\sigma},\wt{b}_{\wt{\sigma}})}$ corresponding to $\vartheta'\in\irr(\beta'_{\sigma,b_\sigma})$ and recalling that $\sigma$ and $\wt{\sigma}$ have the same length, we can define an equivariant bijection
\begin{equation}
\label{eq:Kulshammer-Puig and CTC, 2}
C_\bullet^d(B,Z)_\pm/G\to\C_\bullet^d(C,Z)_\pm/L
\end{equation}
by sending the $G$-orbit of $(\sigma,b_\sigma,\vartheta)$ to the $L$-orbit of $(\wt{\sigma},\wt{b}_{\wt{\sigma}}, \wt{\vartheta})$. Now, by our assumption and using Lemma \ref{lem:Brauer pairs version of CTC}, we know that there exists an $M_C$-equivariant bijection
\[\Omega_\bullet^{C,d}:\C_\bullet^d(C,Z)_+/L\to \C_\bullet^d(C,Z)_-/L\]
such that
\begin{equation}
\label{eq:Kulshammer-Puig and CTC, 3}
\left(M_{\wt{\sigma}_1,\wt{\vartheta}_1},L_{\wt{\sigma}_1},\wt{\vartheta}_1\right)\iso{L}\left(M_{\wt{\sigma}_2,\wt{\vartheta}_2},L_{\wt{\sigma}_2},\wt{\vartheta}_2\right)
\end{equation}
for every $(\wt{\sigma}_1,\wt{\vartheta}_1)\in\C^d_\bullet(C,Z)_+$ and $(\wt{\sigma}_2,\wt{\vartheta}_2)\in\C^d_\bullet(C,Z)_-$ whose $L$-orbits correspond via $\Omega_\bullet^{C,d}$. To conclude, by \eqref{eq:Kulshammer-Puig and CTC, 2}  and again using Lemma \ref{lem:Brauer pairs version of CTC}, it is enough to show that
\[\left(A_{\sigma_1,\vartheta_1},G_{\sigma_1},\vartheta_1\right)\iso{G}\left(A_{\sigma_2,\vartheta_2},G_{\sigma_2},\vartheta_2\right)\]
where the $G$-orbit of $(\sigma_i,b_{\sigma_i},\vartheta_i)$ corresponds to the $L$-orbit of $(\wt{\sigma}_i,\wt{b}_{\wt{\sigma_i}},\wt{\vartheta}_i)$ as explained above. Furthermore, by the transitivity of the $G$-block isomorphism of character triples, it suffices to show that
\[\left(NA_{\sigma_1,\vartheta'_1},NG_{\sigma_1},\vartheta'_1\right)\iso{G}\left(NA_{\sigma_2,\vartheta'_2},G_{\sigma_2},\vartheta'_2\right)\]
according to \eqref{eq:Kulshammer-Puig and CTC, 1} and where $\vartheta'_i$ is the character of the block $\beta'_{(\sigma_i,b_{\sigma_i})}$ corresponding to the character $\vartheta_i$ of the block $\beta_{(\sigma_i,b_{\sigma_i})}$. We now wish to exploit the $L$-block isomorphism given in \eqref{eq:Kulshammer-Puig and CTC, 3} together with the $A(\sigma_i,b_{\sigma_i})$-graded basic Morita equivalence between $\mathcal{O}NA_{\sigma_i,\beta'_{(\sigma_i,b_{\sigma_i})}}\beta'_{(\sigma_i,b_{\sigma_i})}$ and $\mathcal{O}M_{\wt{\sigma}_i, \wt{\beta}_{(\wt{\sigma}_i,\wt{b}_{\wt{\sigma}_i})}}\wt{\beta}_{(\wt{\sigma}_i,\wt{b}_{\wt{\sigma}_i})}$ through the work of \cite{Mar-Min21}. However, notice that here we consider a more general situation for the condition on block induction required to prove the $G$-block isomorphism that is not covered by the Harris--Kn\"orr correspondence considered in \cite{Mar-Min21}. For this, consider intermediate groups $NG_{\sigma_i}\leq J_i\leq NA_{\sigma_i,\vartheta_i'}$ where $G\leq J\leq GA_{\sigma_i,\vartheta'_i}$ and $J_i:=NA_{\sigma_i,\vartheta'_i}\cap J$. Observe here that $GA_{\sigma_1,\vartheta'_1}=GA_{\sigma_2,\vartheta'_2}$ since $LM_{\wt{\sigma}_1,\wt{\vartheta}_1}=LM_{\wt{\sigma}_2,\wt{\vartheta}_2}$ and $LM_{\wt{\sigma}_i,\wt{\vartheta}_i}/L$ is isomorphic to $GA_{\sigma_i,\vartheta'_i}/G$. Moreover, $A(\sigma_i,\vartheta'_i):=NA_{\sigma_i,\vartheta'_i}/NG_{\sigma_i}\simeq M_{\wt{\sigma}_i,\wt{\vartheta}_i}/L_{\wt{\sigma}_i}\leq A(\sigma_i,b_{\sigma_i})$ and \cite[Corollary 4.3]{Coc-Mar17} shows that $\mathcal{O}NA_{\sigma_i,\vartheta'_i}\beta'_{\sigma_i,b_{\sigma_i}}$ and $\mathcal{O}M_{\wt{\sigma}_i,\wt{\vartheta}_i}\wt{\beta}_{\wt{\sigma}_i,\wt{b}_{\wt{\sigma}_i}}$ are $A(\sigma_i,\vartheta'_i)$-graded basic Morita equivalent. Let now $(\Qr_1,\Qr_2)$ be projective representation associated to the $L$-block isomorphism \eqref{eq:Kulshammer-Puig and CTC, 3} and denote by $\Pr_i$ the projective representation associated with $\vartheta'_i$ that corresponds to $\Qr_i$ via the latter equivalence (see, for instance, \cite[Remark 6.5]{Mar-Min21-central}). Furthermore, let $L\leq K\leq LM_{\wt{\sigma}_i,\wt{\vartheta}_i}$ be the subgroup corresponding to $J$ via the isomorphism $GA_{\sigma_i,\vartheta'_i}/G\simeq LM_{\wt{\sigma}_i,\wt{\vartheta}_i}/L$. Let $\psi_i\in\irr(J_i\mid \vartheta'_i)$ and denote by $b_i$ the block to which $\psi_i$ belongs. Consider a projective representation $\mathcal{R}$ of $J/G$, set $\mathcal{R}_i:=\mathcal{R}_{J_i}$, and suppose that $\psi_i$ is afforded by the product $\Pr_{i,J_i}\otimes \mathcal{R}_i$ (see \cite[Theorem 3.3]{Spa17}). We need to show that $b_1^J=b_2^J$. Let $c_i$ be the block of $K_i$ corresponding to $b_i$ and where $K_i:=K\cap M_{\wt{\sigma}_i,\wt{\vartheta}_i}$. Denote by $\wt{\psi}_i\in\irr(K_i\mid \wt{\vartheta}_i)$ the character belonging to the block $c_i$ and corresponding to the character $\psi_i$ of $b_i$. Now, if we identify $\mathcal{R}$ with a projective representation $\wt{\mathcal{R}}$ of $K/L$ via the isomorphism $K/L\simeq J/G$, then we deduce that $\wt{\psi}_i$ is afforded by $\Qr_i\otimes \wt{\mathcal{R}}_i$ where $\wt{\mathcal{R}}_i:=\wt{\mathcal{R}}_{K_i}$. To conclude, we now apply a property of basic Morita equivalences (see, for instance, \cite{Zho06}). For this, set $C_i:=c_i^K$, let $B_i$ be the block of $J$ corresponding to $C_i$, and notice in particular that $B_i$ and $C_i$ are basic Morita equivalent. We claim that $b_i^J=B_i$. Consider defect pointed groups $P_{i,\gamma_i}$ and $\wt{P}_{i,\wt{\gamma}_i}$ of $J_{\{B_i\}}$ and $K_{\{C_i\}}$ respectively  such that there exists a bijection between the local pointed groups $Q_{i,\delta_i}$ contained in $P_{i,\gamma_i}$ and the local pointed groups $\wt{Q}_{i,\wt{\delta}_i}$ contained in $\wt{P}_{i,\wt{\delta}_i}$ (see \cite[1.5]{Zho06}). Every $Q_{i,\delta_i}$ determines a unique block, denoted by $b_{i,\delta_i}$, of $\c_J(Q_i)$ (see \cite[1.4]{Zho06}). Notice then that $(Q_i,b_{i,\delta_i})$ is a $B_i$-subpair. Similarly, every $\wt{Q}_{i,\wt{\delta}_i}$ determines a unique block of $\c_K(\wt{Q}_i)$, which we denote by $\wt{b}_{i,\wt{\delta}_i}$. Moreover, whenever $Q_{i}\c_J(Q_{i})\leq H\leq \n_J(Q_{i,\delta_i})$, we can find a unique block $b_{i,\delta_i}^\star$ of $Q_i\c_J(Q_i)$ covering $b_{i,\delta_i}$ and then the induced block $(b_{i,\delta_i}^\star)^H$ is the unique block of $H$ covering $b_{i,\delta_i}^\star$. A similar observation holds for $\wt{b}_{i,\wt{\delta}_i}$. In particular, consider a local pointed group $D(\wt{\sigma}_i)_{\wt{\delta}_i}$, corresponding to $D(\sigma_i)_{\delta_i}$, such that $c_i$ covers the block $\wt{b}_{i,\wt{\delta}_i}^\star$. Now, $c_i=(\wt{b}_{i,\wt{\delta}_i}^\star)^{K_i}$ corresponds to $b_i=(b_{i,\delta_i}^\star)^{J_i}$ and, as $(D(\sigma_i),b_{i,\delta_i})$ is a $B_i$-subpair, it follows that $b_i^J=B_i$ as claimed. Finally, recalling that $B_i$ corresponds to $C_i$ and noticing that $C_1=C_2$ according to \eqref{eq:Kulshammer-Puig and CTC, 3}, we deduce that $B_1=B_2$ as desired.
\end{proof}

Next, still assuming that $\mu$ extends to $A$ so that the K\"ulshammer--Puig cocycle is trivial, we show that if the block $B$ covers a block of $N$ with defect group $Z$, then $N$ must be contained in the centre of $G$. More precisely, we have the following result.

\begin{prop}
\label{prop:Reduction central defect groups}
Suppose that Hypothesis \ref{hyp:Minimal counterexample} holds for $G\unlhd A$. Consider a $p$-subgroup $Z\leq \z(G)$ and a block $B$ of $G$ whose defect groups strictly contain $Z$ and for which Conjecture \ref{conj:CTC} fails to hold. Let $Z\leq N\leq G$ with $N\unlhd A$ and $b$ be an $A$-invariant block of $N$ covered by $B$ with defect group $Z$. If the unique irreducible character $\mu$ of $b$ containing $Z$ in its kernel extends to $A$, then $N\leq \z(G)$.
\end{prop}

\begin{proof}
We assume that $N$ is not contained in $\z(G)$ and show that Conjecture \ref{conj:CTC} holds for $B$, against the above assumption. By \cite[Lemma 2.3]{Ros22} we know that $Z=\O_p(G)$ and therefore $\O_p(G)\leq N$. Then, $\z(G)<N\z(G)\leq N\O_{p'}(G)$ and we get $|A:N\O_{p'}(G)|\leq |A:N\z(G)|<|A:\z(G)|$. Consider now the groups $L\unlhd M$ and the block $C$ of $L$ from the statement of Lemma \ref{lem:Kulshammer-Puig and CTC} and observe that $Z$ is strictly contained in the defect groups of $C$. This is because $B$ and $C$ are basic Morita equivalent and hence have isomorphic defect groups. Notice furthermore that, since $Z\leq \z(G)$, it follows that $Z\leq \z(L)$ and $L$ is a central extension of $G/N$. In particular, we have $\mathcal{S}_p(L)\subseteq \mathcal{S}_p(G/N)\subseteq \mathcal{S}_p(G)$. Set $T:=\O_{p'}(L)$ and notice that $TZ/Z=\O_{p'}(L/Z)$. Then $|M:TZ|=|M/Z:TZ/Z|=|A/N:\O_{p'}(G/N)|\leq |A:N\O_{p'}(G)|$ and we deduce from Hypothesis \ref{hyp:Minimal counterexample} that Conjecture \ref{conj:CTC} holds for every $L_1\unlhd M_1$ such that $\mathcal{S}_p(L_1)\subseteq \mathcal{S}_p(L)$ and $|M_1:\z(L_1)|\leq|M:ZT|$. If $c$ is a block of $T$ covered by $C$ and $C'$ is its Fong--Reynolds correspondent in $L_c$, then proceeding as in \cite[Section 4]{Ros22} we construct a central extension $\wt{M}$ of $M_c/T$ containing a normal subgroup $\wt{L}$ and a block $\wt{C}$ of $\wt{L}$ corresponding to $C'$ (see \cite[Theorem 4.2]{Ros22}). According to the discussion preceding \cite[Corollary 4.5]{Ros22}, and noticing that $C'$ has defect groups in common with $C$ which strictly contain $Z$, the $p$-subgroup $Z$ of $L_c$ corresponds to a $p$-subgroup of $\wt{L}$ denoted by $\wt{Z}_p$ and strictly contained in the defect groups of $\wt{C}$. Now, since $|\wt{M}:\z(\wt{L})|\leq|M_c:\z(L_c)T|\leq |M:ZT|$ and $\mathcal{S}_p(\wt{L})\subseteq \mathcal{S}_p(L)$, the above discussion implies that Conjecture \ref{conj:CTC} holds for the block $\wt{C}$. Applying \cite[Corollary 4.7]{Ros22} we deduce that Conjecture \ref{conj:CTC} holds for $C'$ and hence, proceeding as in the proof of \cite[Proposition 5.1]{Ros22}, we obtain Conjecture \ref{conj:CTC} for the block $C$. We can then apply Lemma \ref{lem:Kulshammer-Puig and CTC} to show that Conjecture \ref{conj:CTC} holds for the block $B$ as desired.
\end{proof}

Finally, using the results obtained in the previous section, we extend the above statement to the case where the block $b$ is not necessarily $A$-stable and the character $\mu$ might not extend to $A$.

\begin{cor}
\label{cor:Reduction central defect groups}
Suppose that Hypothesis \ref{hyp:Minimal counterexample} holds for $G\unlhd A$ and consider a $p$-subgroup $Z\leq \z(G)$ and a block $B$ of $G$ whose defect groups strictly contain $Z$ and for which Conjecture \ref{conj:CTC} fails to hold. If $Z\leq N\leq G$ with $N\unlhd A$ and $b$ is a block of $N$ covered by $B$ with defect group $Z$, then $N\leq \z(G)$.
\end{cor}

\begin{proof}
As before, notice that $b$ has central defect groups and there exists a unique irreducible character $\mu$ belonging to $b$ whose kernel contains $Z$. First, observe that $b$ is $A$-invariant thanks to Proposition \ref{prop:Stable reduction} and then so is $\mu$ by its uniqueness property. Next, fix a projective representation $\Pr$ associated with the character triple $(A,N,\mu)$ and denote by $\alpha$ the factor set of $\Pr$ (not to be confused with the K\"ulshammer--Puig cocycle). As noted in \cite[p.212]{Rob02} it is no loss of generality to assume that $\alpha$ takes values in the set of $p'$-roots of unity. Now, consider the associated central extension $\wh{A}$ of $A$ by the cyclic $p'$-subgroup $S$ generated by the values of $\alpha$ (see \cite[Theorem 1.12]{Spa18} or \cite[Theorem 5.6]{Nav18}). The group $\wh{A}$ consists of pairs $(a,s)$ where $a\in A$ and $s\in S$ and we denote by $\epsilon:\wh{A}\to A$ the projection $(a,s)\mapsto a$ with kernel $S$. For every $X\leq A$, we define $\wh{X}:=\epsilon^{-1}(X)$ and $X_0:=\{(x,1)\mid x\in X\}$. Since $\alpha$ is a factor set of $A/N$, we deduce that $\wh{X}=X_0\times S$ whenever $X\leq N$. In this case, there is a natural isomorphism between $X_0$ and $X$. We denote by $\mu_0$ the irreducible character of $N_0$ corresponding to $\mu$. Our choice of $\Pr$ induces an irreducible representation of $\wh{A}$ affording an extension $\wh{\mu}$ of $\mu_0$. By \cite[Theorem 5.8.8]{Nag-Tsu89} there are blocks $\wh{B}$ of $\wh{G}$ and $b_0$ of $N_0$ corresponding (via $\epsilon$) to $B$ and $b$ respectively. Notice that $b_0$ is covered by $\wh{B}$ and has defect group $Z_0$ while the defect groups of $\wh{B}$ strictly contain $Z_0$. Furthermore Hypothesis \ref{hyp:Minimal counterexample} holds for $\wh{G}\unlhd \wh{A}$: using the fact that $\wh{G}$ is a central extension of $G$ we obtain $\mathcal{S}_p(\wh{G})=\mathcal{S}_p(G)$, while as $\z(\wh{G})=\wh{\z(G)}$ (see \cite[Theorem 4.1 (d)]{Spa17}) we have $|\wh{A}:\z(\wh{G})|=|A:\z(G)|$. Now, if Conjecture \ref{conj:CTC} fails for $\wh{B}$, then the hypothesis considered in the statement is satisfied with respect to $Z_0\leq N_0\leq \wh{G}\leq \wh{A}$ and $\wh{B}$ covering $b_0$. Moreover, if the result holds for these choices, then $N_0\leq \z(\wh{G})$ and therefore $N\leq \z(G)$. Therefore, replacing $A$ with $\wh{A}$, it is no loss of generality to assume that $\mu$ extends to $A$ in which case the result follows from Proposition \ref{prop:Reduction central defect groups}.

It remains to show that Conjecture \ref{conj:CTC} fails for $\wh{B}$. First, notice that according to \cite[Proposition 2.3 and (2.5)]{Dad94} there exists a bijection between $p$-chains $\sigma\in\mathfrak{N}(G,Z)$ and $p$-chains $\wh{\sigma}\in\mathfrak{N}(\wh{G},Z_0)$ such that $\wh{G_\sigma}=\wh{G}_{\wh{\sigma}}$. By \cite[Proposition 2.4 (b)]{Nav-Spa14I} there is a bijection between the blocks $C$ of $G_\sigma$ whose induction $C^G$ coincides with $B$ and the blocks $\wh{C}$ of $\wh{G}_{\wh{\sigma}}$ whose induction $\wh{C}^{\wh{G}}=\wh{B}$. Moreover, \cite[Theorem 5.8.8]{Nag-Tsu89} implies that the characters of $\wh{C}$ are obtained from the characters of $C$ via inflation through the map $\epsilon$. Since $p$ does not divide the order of $S$ it follows that this correspondence preserves the $p$-defect of characters. Thus, we have a bijection
\[\Gamma:\C^d(B,Z)/G\to \C^d\left(\wh{B},Z_0\right)/\wh{G}\]
which commutes with the action of $\n_A(Z)_B$ and that of $\n_{\wh{A}}(Z_0)_{\wh{B}}$ and preserves the length of $p$-chains. Now, if Conjecture \ref{conj:CTC} holds for $\wh{B}$, then we obtain an $\n_{\wh{A}}(Z_0)_{\wh{B}}$-equivariant bijection
\[\wh{\Omega}:\C^d\left(\wh{B},Z_0\right)_+/\wh{G}\to \C^d\left(\wh{B},Z_0\right)_-/\wh{G}\]
for every $d\geq 0$ that satisfies the required condition on character triples. Combining $\wh{\Omega}$ with $\Gamma$ we obtain an $\n_A(Z)_B$-equivariant bijection
\[\Omega:\C^d(B,Z)_+/G\to\C^d(B,Z)_-/G\]
for every $d\geq 0$. To obtain the required isomorphism of character triples for $\Omega$, we apply \cite[Corollary 4.5]{Spa17} together with the fact that analogous isomorphisms of character triples are given by $\wh{\Omega}$. Notice that the conditions on the centralisers assumed in \cite[Corollary 4.5]{Spa17}  are satisfied since $S$ is a central $p'$-subgroup and applying \cite[Theorem 4.1 (d)]{Spa17}. This shows that Conjecture \ref{conj:CTC} holds for $B$, against our assumption. Therefore, Conjecture \ref{conj:CTC} must fail for $\wh{B}$ as required.
\end{proof}

\subsection{Proof of Theorem \ref{thm:Minimal counterexample}}
\label{sec:Minimal counterexample proof}

In this section, we finally prove Theorem \ref{thm:Minimal counterexample}. We start with an elementary group-theoretic result which requires the use of coprime group actions. We refer the reader to the content of \cite[Chapter 5.3]{Gor}.

\begin{lem}
\label{lem:Coprime actions}
Let $U\leq D$ be $p$-subgroups of the finite group $G$ such that $\c_D(U)\leq U$ and consider $U\unlhd N\unlhd G$. If $\O^p(N)\cap U\leq \z(G)$, then $\O^p(N)\cap D\leq \z(G)$.
\end{lem}

\begin{proof}
Let $M:=\O^p(N)$ and observe that $[M,U]\leq M\cap U\leq \z(G)\cap U$. In particular $[M,U,M]=1=[U,M,M]$ and the three subgroups lemma implies that $[M,M]$ centralises $U$. Now $M/[M,M]$ is a $p'$-group acting on the $p$-group $U$ and stabilising the normal chain $\z(G)\cap U\unlhd U$ (see \cite[Lemma 3.1]{Gor}). We can then apply \cite[Theorem 3.2]{Gor} to conclude that $M$ centralises $U$. Using the assumption $\c_D(U)\leq U$, we finally get $M\cap D=M\cap \c_D(U)\leq M\cap U\leq \z(G)$ as required.
\end{proof}

The above lemma allows us to construct certain nilpotent blocks to which Proposition \ref{prop:Auxiliary theorem} can be applied. This idea, together with Theorem \ref{thm:CTC equivalent to CTC non-central}, allows us to obtain the following result. Notice that the hypothesis assumed below is stronger than Hypothesis \ref{hyp:Minimal counterexample}.

\begin{prop}
\label{prop:Forcing bijections in minimal counterexample}
Let $G\unlhd A$ be finite groups and suppose that Conjecture \ref{conj:CTC} holds for every finite group $G_1\unlhd A_1$ such that $\mathcal{S}_p(G_1)\subseteq \mathcal{S}_p(G)$ and $|A_1:\z(G_1)|\leq|A:\z(G)|$. Let $B$ be a block of $G$ and $U\unlhd G$ a $p$-subgroup. Suppose that $U\leq N\leq G$ with $N\unlhd A$ and $\O^p(N)\cap U\leq \z(G)$ and that $B$ covers a block $b$ of $N$ whose defect groups strictly contain $U$. Then for every $d\geq 0$ there exists an $\n_A(U)_B$-equivariant bijection
\[\Omega:\C^d(B,U)_+/G\to\C^d(B,U)_-/G\]
such that
\[\left(A_{\sigma,\vartheta},G_\sigma,\vartheta\right)\iso{G}\left(A_{\rho,\chi},G_\rho,\chi\right)\]
for every $(\sigma,\vartheta)\in\C^d(B,U)_+$ and $(\rho,\chi)\in\Omega(\overline{(\sigma,\vartheta)})$.
\end{prop}

\begin{proof}
In the above setting, the hypothesis of Theorem \ref{thm:CTC equivalent to CTC non-central} are satisfied and hence Conjecture \ref{conj:CTC non-central Z} holds for $G\unlhd A$. For every $\varphi\in\irr(U)$, we then obtain an $\n_A(U)_{B,\varphi}$-equivariant bijection
\[\underline{\Omega}_\varphi:\underline{\C}^d(B,U,\varphi)_+/G_\varphi\to\underline{\C}^d(B,U,\varphi)_-/G_\varphi
\]
satisfying the condition on isomorphisms of character triples given by Conjecture \ref{conj:CTC non-central Z}. If we fix an $\n_A(U)_B$-transversal $\mathcal{U}$ in $\irr(U)$ and an $\n_A(U)_{B,\varphi}$-transversal $\mathcal{T}_\varphi^+$ in $\underline{\C}^d(B,U,\varphi)_+$ for every $\varphi\in\mathcal{U}$, then we obtain an $\n_A(U)_B$-transversal
\[\mathcal{T}^+:=\left\lbrace\overline{(\sigma,\vartheta)}\enspace\middle|\enspace (\sigma,\vartheta)\in\mathcal{T}_\varphi^+,\varphi\in\mathcal{U}\right\rbrace\]
in $\underline{\C}^d(B,U)_+/G$. For every $\varphi\in\mathcal{U}$, we construct an $\n_A(U)_{B,\varphi}$-transversal $\mathcal{T}_\varphi^-$ in $\underline{\C}^d(B,U,\varphi)_-$ in bijection with $\mathcal{T}_\varphi^+$ via $\underline{\Omega}_\varphi$ (more precisely, the corresponding $G_\varphi$-orbits are mapped via $\underline{\Omega}_\varphi$) and hence obtain an $\n_A(U)_B$-transversal
\[\mathcal{T}^-:=\left\lbrace\overline{(\sigma,\vartheta)}\enspace\middle|\enspace (\sigma,\vartheta)\in\mathcal{T}_\varphi^-,\varphi\in\mathcal{U}\right\rbrace\]
in bijection with $\mathcal{T}^+$. Then, defining
\[\underline{\Omega}\left(\overline{(\sigma,\vartheta)}^x\right):=\overline{(\rho,\chi)}^x\]
for every $(\sigma,\vartheta)\in\mathcal{T}_\varphi^+$ whose $G_\varphi$-orbit corresponds to that of $(\rho,\chi)\in\mathcal{T}^-_\varphi$ via $\underline{\Omega}_\varphi$, every $\varphi\in\mathcal{U}$ and any $x\in\n_A(U)_B$, we obtain an $\n_A(U)_B$-equivariant bijection
\[\underline{\Omega}:\underline{\C}^d(B,U)_+/G\to\underline{\C}^d(B,U)_-/G\]
satisfying the condition on isomorphisms of character triples required in the statement above (see the proof of Proposition \ref{prop:CTC non-central implies CTC}). Now, if $\omega^d(B_\sigma,U)$ is empty for every $\sigma\in\mathfrak{N}(G,U)$, then $\underline{\C}^d(B,U)=\C^d(B,U)$ and we define $\Omega:=\underline{\Omega}$ to conclude the proof. On the other hand, if this is not the case, then $\c_D(U)\leq U$ for any defect group $D$ of $B$ according to Lemma \ref{lem:CTC non-central implies CTC}. Since by assumption $\O^p(N)\cap U\leq \z(G)$, we deduce that $\O^p(N)\cap D\leq \z(G)$ from Lemma \ref{lem:Coprime actions}. As a consequence, any block of $\O^p(N)$ covered by $B$ is nilpotent (in fact, it has central defect groups) and it follows from \cite[Theorem 2]{Cab87} that $b$ is nilpotent. In this case, Proposition \ref{prop:Auxiliary theorem} gives the desired result.
\end{proof}

We are finally able to prove Theorem \ref{thm:Minimal counterexample}.

\begin{proof}[Proof of Theorem \ref{thm:Minimal counterexample}]
We consider $G\unlhd A$, $Z\leq \z(G)$ a $p$-subgroup strictly contained in the defect groups of a block $B$ of $G$, and a non-negative integer $d$. In addition, we assume that Hypothesis \ref{hyp:Minimal counterexample} holds for $G\unlhd A$ and that Conjecture \ref{conj:CTC} fails with respect to these choices. Let $N$ be a non-central subgroup of $G$ such that $Z\leq N\unlhd A$ and consider a block $b$ of $N$ covered by $B$. According to Proposition \ref{prop:Stable reduction} and Corollary \ref{cor:Reduction central defect groups} we know that $b$ is $A$-invariant and $Z$ is strictly contained in its defect groups.

Suppose that $\sigma$ is a $p$-chain of $G$ with starting term $Z$ and whose largest term $D(\sigma)$ is not contained in $N$. If $\sigma=\{Z=D_0<D_1<\dots<D_n=D(\sigma)\}$, then we can find a minimal $1\leq j_\sigma\leq n$ such that $D_{j_\sigma}\nleq N$ and in this case we have $D_{j_\sigma-1}\leq D_{j_\sigma}\cap N<D_{j_\sigma}$. We now partition the set $\mathcal{D}^d_N(B,Z)$ into two subsets $\mathcal{D}^d_N(B,Z)^\circ$ and $\mathcal{D}^d_N(B,Z)^\bullet$ where $(\sigma,\vartheta)\in\mathcal{D}^d_N(B,Z)^\circ$ if $Z<D_{j_\sigma}\cap N$ while $(\sigma,\vartheta)\in\mathcal{D}^d_N(B,Z)^\bullet$ if $Z=D_{j_\sigma}\cap N$ (in which case $j_\sigma=1$). Notice that this partition is $\n_A(Z)_B$-stable and therefore it is enough to define $\Lambda_{N,Z}^B$ on each of the two subsets.

First, we define $\Lambda_{N,Z}^B$ with respect to the subset $\mathcal{D}^d_N(B,Z)^\circ$. Fix some $(\sigma,\vartheta)\in\mathcal{D}_N^d(B,Z)_+$ and consider $j_\sigma$ as defined above. If $D_{j_\sigma-1}=D_{j_\sigma}\cap N$, we define the $p$-chain $\rho$ obtained by deleting the term $D_{j_\sigma-1}$ from $\sigma$. In this case, $\rho$ is a $p$-chain of $G$ starting with $Z$ and such that $D(\rho)\nleq N$. Moreover, if we write $\rho=\{Z=Q_0<Q_1<\dots <Q_n=D(\rho)\}$, then we have $Q_{j_\rho}=D_{j_\sigma}$ and, since $Z<D_{j_\sigma}\cap N$ by assumption, we deduce that $Z<D_{j_\sigma}\cap N=Q_{j_\rho}\cap N$. On the other hand if $D_{j-1}< D_j\cap N$, then we define $ \rho$ to be the $p$-chain obtained by adding $D_{j_\sigma}\cap N$ as an intermediate term in the $p$-chain $\sigma$. As before $\rho$ is a $p$-chain of $G$ starting with $Z$ and such that $D(\rho)\nleq N$, $Q_{j_\rho}=D_{j_\sigma}$ and $Z<Q_{j_\rho}\cap N$. In both cases observe that $|\rho|=|\sigma|\pm 1$ and that the stabilisers $A_\sigma$ and $A_\rho$, and therefore also $G_\sigma$ and $G_\rho$, coincide. We can then define
\[\Lambda_{N,Z}^B\left(\overline{\sigma,\vartheta}\right):=\overline{(\rho,\vartheta)}\]
to get an $\n_A(Z)_B$-equivariant bijection
\begin{equation}
\label{eq:Proof on minimal counterexample 1}
\Lambda_{N,Z}^B:\mathcal{D}^d_N(B,Z)_+^\circ/G\to\mathcal{D}^d_N(B,Z)_-^\circ/G
\end{equation}
which satisfies the condition on block isomorphisms of character triples required by Theorem \ref{thm:Minimal counterexample}.

We now consider the subset $\mathcal{D}^d_N(B,Z)^\bullet$. If $\sigma=\{Z=D_0<D_1<\dots<D_n=D(\sigma)\}$ is a $p$-chain appearing in this situation, then $Z=D_1\cap N$. Let $\mathcal{U}$ be an $\n_A(Z)_B$-transversal in the set of $p$-subgroups $U$ of $G$ such that $Z<U$ and $Z=U\cap N$. For every $U\in\mathcal{U}$ we define the set $\C^d(B,Z)_U$ consisting of those pairs $(\sigma,\vartheta)\in\C^d(B,Z)$ such that the second term of the $p$-chain $\sigma$ coincides with $U$. Observe that by the properties of $U$ such pairs $(\sigma,\vartheta)$ belong to $\mathcal{D}^d_N(B,Z)^\bullet$. Our aim is now to construct bijections of the form
\[\Delta_{Z,U}^B:\C^d(B,Z)_{U,+}/\n_G(U)\to\C^d(B,Z)_{U,-}/\n_G(U)\]
which we then combine in order to define $\Lambda_{N,Z}^B$ on the set $\mathcal{D}^d_{N}(B,Z)^\bullet$.

Let $B'$ be a block of $\n_G(U)$ such that $(B')^G=B$. We claim that the assumption of Proposition \ref{prop:Forcing bijections in minimal counterexample} is satisfied with respect to $U$, $U\n_N(U)$, $\n_G(U)$, $\n_A(U)$ and $B'$. First, observe that since Conjecture \ref{conj:CTC} fails for our choice of $Z$, we have $Z=\O_p(G)$ by \cite[Lemma 2.3]{Ros22} and hence $|\n_A(U):\z(\n_G(U))|<|A:\z(G)|$ since $Z<U$. In particular, Conjecture \ref{conj:CTC} holds for every $G_1\unlhd A_1$ such that $\mathcal{S}_p(G_1)\subseteq \mathcal{S}_p(\n_G(U))$ and $|A_1:\z(G_1)|\leq |\n_A(U):\z(\n_G(U))|$ thanks to Hypothesis \ref{hyp:Minimal counterexample}. Next, $\O^p(U\n_N(U))=\O^p(\n_N(U))\leq N$ and therefore $\O^p(U\n_N(U))\cap U\leq N\cap U=Z\leq \z(\n_G(U))$, where the equality $Z=N\cap U$ holds since $U\in\mathcal{U}$. It remains to show that $B'$ covers a block of $U\n_N(U)$ whose defect groups strictly contain $U$. Recall, by the first paragraph, that $Z$ is strictly contained in the defect groups of a block $b$ of $N$ covered by $B$. Since $Z=N\cap U$, it follows that $U$ is not a defect group of $B$ and hence Brauer's First Main Theorem shows that $U$ is not a defect group of $B'$. Similarly, the unique block $c$ of $UN$ covering $b$ does not have $U$ has a defect group and nor does any block $b'$ of $U\n_N(U)=\n_{UN}(U)$ such that $(b')^{UN}=c$. Since any such block $b'$ is covered by $B'$ (see, for instance, \cite[Theorem B]{Kos-Spa15}), we deduce that the assumptions of Proposition \ref{prop:Forcing bijections in minimal counterexample} are satisfied with respect to $U$, $U\n_N(U)$, $\n_G(U)$, $\n_A(U)$ and $B'$ as claimed. We can now apply Proposition \ref{prop:Forcing bijections in minimal counterexample} to obtain an $\n_A(U)_{B'}$-equivariant bijection
\[\Delta^{B'}_U:\C^d(B',U)_+/\n_G(U)\to\C^d(B',U)_-/\n_G(U)\] 
such that
\begin{equation}
\label{eq:Proof on minimal counterexample 2}
\left(\n_A(U)_{\sigma,\vartheta},\n_G(U)_\sigma,\vartheta\right)\iso{\n_G(U)}\left(\n_A(U)_{\rho,\chi},\n_G(U)_\rho,\chi\right)
\end{equation}
for every $(\sigma,\vartheta)\in\C^d(B',U)_+$ and $(\rho,\chi)\in\Delta^{B'}_U(\overline{(\rho,\chi)})$ and where $\overline{(\sigma,\vartheta)}$ denotes here the $\n_G(U)$-orbit of $(\sigma,\vartheta)$. Noticing that $\n_A(U)_\sigma=A_\sigma$ for every $p$-chain $\sigma$ starting with $U$ and applying \cite[Lemma 2.11]{Ros22} to \eqref{eq:Proof on minimal counterexample 2} we deduce that
\begin{equation}
\label{eq:Proof on minimal counterexample 3}
\left(A_{\sigma,\vartheta},G_\sigma,\vartheta\right)\iso{G}\left(A_{\rho,\chi},G_\rho,\chi\right).
\end{equation}
If we denote by $B_U$ the union of those blocks of $\n_G(U)$ whose Brauer's induction to $G$ coincides with $B$, then it is immediate to verify that $\C^d(B_U,U)$ coincides with $\C^d(B,U)$ once we realise that $G_\sigma\leq \n_G(U)$ whenever $\sigma$ is a $p$-chain starting with $U$. Due to this observation, we can apply Lemma \ref{lem:Bijections for union of blocks} and combine the bijections $\Delta_U^{B'}$ to obtain an $\n_A(U)_B$-equivariant bijection
\[\Delta_U^B:\C^d(B,U)_+/\n_G(U)\to\C^d(B,U)_-/\n_G(U).\]
Consider now a pair $(\sigma,\vartheta)\in\C^d(B,Z)_{U,\pm}$ and observe that if we define $\sigma'$ to be the $p$-chain obtained by removing $Z$ from the chain $\sigma$, then $(\sigma',\vartheta)\in\C^d(B,U)_\mp$. This construction yields an $\n_A(Z,U)_B$-equivariant bijection between $\C^d(B,Z)_{U,\pm}$ and $\C^d(B,U)_\mp$ and, together with the map $\Delta_U^B$ given above, allows us to construct an $\n_A(Z,U)_B$-equivariant bijection
\[\Delta_{Z,U}^B:\C^d(B,Z)_{U,+}/\n_G(U)\to\C^d(B,Z)_{U,-}/\n_G(U).\]
Furthermore, we deduce from \eqref{eq:Proof on minimal counterexample 3} that
\begin{equation}
\label{eq:Proof on minimal counterexample 4}
\left(A_{\sigma,\vartheta},G_\sigma,\vartheta\right)\iso{G}\left(A_{\rho,\chi},G_\rho,\chi\right)
\end{equation}
for every $(\sigma,\vartheta)\in\C^d(B,Z)_{U,+}$ and $(\rho,\chi)\in\Delta^B_{Z,U}(\overline{(\sigma,\vartheta)})$. We recall one more time that here $\overline{(\sigma,\vartheta)}$ denotes the $\n_G(U)$-orbit of $(\sigma,\vartheta)$.

We now define the map $\Lambda_{N,Z}^B$ with respect to the set $\mathcal{D}^d_N(B,Z)^\bullet$. We fix an $\n_A(Z,U)_B$-transversal $\mathcal{T}_U^+$ in $\C^d(B,Z)_{U,+}$ and define a set $\mathcal{T}_U^-$ by picking exactly one pair $(\rho,\chi)$ from each $\n_G(U)$-orbit $\Delta_{Z,U}^B(\overline{(\sigma,\vartheta)})$ and while $(\sigma,\vartheta)$ runs over the set $\mathcal{T}_U^+$. Since $\Delta_{Z,U}^B$ is $\n_A(Z,U)_B$-equivariant, we deduce that $\mathcal{T}_U^-$ is an $\n_A(Z,U)_B$-transversal in $\C^d(B,Z)_{U,-}$. Then the sets
\[\mathcal{T}^+:=\left\lbrace \overline{(\sigma,\vartheta)}\enspace\middle|\enspace(\sigma,\vartheta)\in\mathcal{T}_U^+, U\in\mathcal{U}\right\rbrace\]
and
\[\mathcal{T}^-:=\left\lbrace \overline{(\rho,\chi)}\enspace\middle|\enspace(\rho,\chi)\in\mathcal{T}_U^-, U\in\mathcal{U}\right\rbrace\]
are $\n_A(Z)_B$-transversal in the sets $\mathcal{D}^d_N(B,Z)^\bullet_+/G$ and $\mathcal{D}^d_N(B,Z)_-/G$ respectively and where we now use the notation $\overline{(\sigma,\vartheta)}$ and $\overline{(\rho,\chi)}$ to indicate the corresponding $G$-orbits. We finally define
\[\Lambda_{N,Z}^B\left(\overline{(\sigma,\vartheta)}^x\right)=\overline{(\rho,\chi)}^x\]
for every $(\sigma,\vartheta)\in\mathcal{T}_U^+$ corresponding to $(\rho,\chi)\in\mathcal{T}_U^-$ via $\Delta_{Z,U}^B$, every $U\in\mathcal{U}$, and every $x\in\n_A(Z)_B$. This gives us an $\n_A(Z)_B$-equivariant bijection
\begin{equation}
\label{eq:Proof on minimal counterexample 5}
\Lambda_{N,Z}^B:\mathcal{D}^d_N(B,Z)_+^\bullet/G\to\mathcal{D}^d_N(B,Z)_-^\bullet/G
\end{equation}
that satisfies the required condition on block isomorphisms of character triples thanks to \eqref{eq:Proof on minimal counterexample 4}. Now the result follows by combining the bijections given in \eqref{eq:Proof on minimal counterexample 1} and \eqref{eq:Proof on minimal counterexample 5} together with the remark at the beginning of the proof.
\end{proof}

\section{The reduction}
\label{sec:Reduction}

In this section, we finally prove a reduction theorem for the Character Triple Conjecture. The strategy for this can be described as follows. Let $G\unlhd A$ be finite groups and $B$ a block of $G$ for which the conjecture fails to hold. Assuming that $G\unlhd A$ is minimal with respect to $|A:\z(G)|$, it follows from Theorem \ref{thm:Minimal counterexample} that it is enough to find a subgroup $N$ of $G$ such that $Z\leq N\unlhd A$ and $N\nleq \z(G)$ and to construct a bijection as in \eqref{eq:Cancellation subset}. As we will see later, we can apply this argument to $N:=\F^\star(G)$. Then the problem reduces to showing that the conjecture holds for $N$ once it is assumed for the quasi-simple components of $G$. This is exactly what we show in Theorem \ref{thm:CTC over simple groups} below.

\subsection{Consequences for covering groups}

We now review some of the arguments from \cite[Section 7 and 8]{Spa17}. In particular, we obtain a version of \cite[Corollary 8.2]{Spa17} compatible with block isomorphisms of character triples and where we relax the hypothesis in order to include groups that are not necessarily perfect. We start by stating a reformulation of \cite[Theorem 7.2]{Spa17}. 

\begin{lem}
\label{lem:Perfect}
Let $S$ be a non-abelian simple group such that Conjecture \ref{conj:CTC} holds at $p$ for every covering group of $S$. Then Conjecture \ref{conj:CTC} holds at $p$ for every covering group of $S^n$ for every $n\geq 1$.
\end{lem}

\begin{proof}
By our assumption the inductive condition for Dade's Conjecture as stated in \cite[Definition 6.7]{Spa17} holds for $S$. Then \cite[Theorem 7.2]{Spa17} implies that Conjecture \ref{conj:CTC+} holds for $\wh{S}^n$ with respect to $\wh{S}^n\unlhd \wh{S}^n\rtimes \aut(\wh{S}^n)$. Since $\wh{S}^n$ is the universal covering group of $S^n$, Lemma \ref{lem:CTC vs CTC+} shows that Conjecture \ref{conj:CTC} holds for every covering group of $S^n$.
\end{proof}

Next, we extend the above result to direct products of, possibly non-isomorphic, non-abelian simple groups.

\begin{prop}
\label{prop:Perfect}
Let $S_1,\dots,S_r$ be non-isomorphic non-abelian simple groups and suppose that Conjecture \ref{conj:CTC} holds at $p$ for every covering group of every $S_i$. Then Conjecture \ref{conj:CTC} holds at $p$ for every covering group of $S_1^{n_1}\times \dots\times S_r^{n_r}$ for every $n_1,\dots,n_r\geq 1$. 
\end{prop}

\begin{proof}
Let $K$ be a covering group of $S_1^{n_1}\times \dots\times S_r^{n_r}$, so that $K$ is a perfect group and $K/\z(K)\simeq S_1^{n_1}\times \dots\times S_r^{n_r}$. Fix a $p$-subgroup $V\leq \z(K)$, a block $C$ of $K$ with defect groups strictly containing $V$, a non-negative integer $d$, and suppose that $K\unlhd A$. We want to prove Conjecture \ref{conj:CTC} for $V\leq K\unlhd A$, $C$ and $d$.

Let $\wh{S}_i$ be the universal covering group of $S_i$ and define $\wh{K}_i:=\wh{S}_i^{n_i}$ and $\wh{K}:=\bigtimes_i\wh{K}_i$. Set $\wh{A}_i:=\wh{K}_i\rtimes \aut(\wh{K}_i)$ and notice that $\wh{A}:=\wh{K}\rtimes \aut(\wh{K})=\bigtimes_i\wh{A}_i$. Since $\wh{K}$ is the universal covering group of $K$ (see \cite[Exercise 2, Chapter 11]{Asc86} and \cite[Corollary B.10]{Nav18}), there exists an epimorphism $\epsilon:\wh{K}\to K$ with central kernel $W$. Under $\epsilon$, the $p$-subgroup $V\leq \z(K)$ corresponds to a unique $p$-subgroup $W\leq \wh{V}\leq \z(\wh{K})$ and we can write $\wh{V}=\wh{V}_1\times \dots\times \wh{V}_r$ for $p$-subgroups $\wh{V}_i\leq \z(\wh{K}_i)$. The isomorphism $\wh{K}/\wh{V}\simeq K/V$ induces a length preserving bijection
\begin{equation}
\label{eq:Perfect 1}
\mathfrak{N}\left(\wh{K},\wh{V}\right)/\wh{K}\to\mathfrak{N}(K,V)/K
\end{equation}
such that, if the $\wh{K}$-orbit of $\wh{\sigma}\in\mathfrak{N}(\wh{K},\wh{V})$ corresponds to the $K$-orbit of $\sigma\in\mathfrak{N}(K,V)$, then the stabiliser $\wh{K}_{\wh{\sigma}}$ corresponds to $K_\sigma$ via $\epsilon$. Recall by that \cite[Proposition 6.10]{Spa17} it is no loss of generality to assume that all chains considered in this proof are radical. Then, every (radical) $p$-chain $\wh{\sigma}$ of $\wh{K}$ starting with $\wh{V}$ can be written as $\wh{\sigma}=(\wh{V}=\wh{D}_0<\wh{D}_1<\dots<\wh{D}_n)$ with $\wh{D}_j=\wh{D}_{j,1}\times \dots\times \wh{D}_{j,r}$. Then $\wh{\sigma}_i=(\wh{V}_i=\wh{D}_{0,i}<\wh{D}_{1,i}<\dots<\wh{D}_{n,i})$ is a $p$-chain of $\wh{K}_i$ starting with $\wh{V}_i$ and we have $\wh{K}_{\wh{\sigma}}=\bigtimes_i \wh{K}_{i,\wh{\sigma}_i}$. Finally, observe that by \cite[Theorem 5.8.8 and Theorem 5.8.11]{Nag-Tsu89} there is a unique block $\wh{C}$ of $\wh{K}$ corresponding to the block $C$ of $K$ via $\epsilon$. We can write $\wh{C}=\wh{C}_1\otimes \dots\otimes \wh{C}_r$ for unique blocks $\wh{C}_i$ of $\wh{K}_i$. Since $V$ is strictly contained in the defect groups of $C$, we conclude that the defect groups of $\wh{C}$ and of $\wh{C}_i$ strictly contain $\wh{V}$ and $\wh{V}_i$ respectively.

Now, by Lemma \ref{lem:Perfect}, we know that Conjecture \ref{conj:CTC} holds for every covering group of $S_i^{n_i}$ and hence, as $\wh{K}_i$ is the universal covering group of $S_i^{n_i}$, Lemma \ref{lem:CTC vs CTC+} shows that Conjecture \ref{conj:CTC+} holds for $\wh{K}$. Thus, we obtain an $\n_{\wh{A}_i}(\wh{V}_i)_{\wh{C}_i}$-equivariant bijection
\[\wh{\Omega}_i^{d_i}:\C^{d_i}\left(\wh{C}_i,\wh{V}_i\right)_+/\wh{K}_i\to\C^{d_i}\left(\wh{C}_i,\wh{V}_i\right)_-/\wh{K}_i\]
satisfying the properties required by Conjecture \ref{conj:CTC+}. By combining the bijections $\wh{\Omega}_i^{d_i}$ with $d_i$ running over all non-negative integers, we obtain a map
\[\wh{\Omega}_i:\C\left(\wh{C}_i,\wh{V}_i\right)_+/\wh{K}_i\to\C\left(\wh{C}_i,\wh{V}_i\right)_-/\wh{K}_i.\]
As before, $\wh{\Omega}_i$ is an $\n_{\wh{A}_i}(\wh{V}_i)_{\wh{C}_i}$-equivariant bijection satisfying the conclusions of Conjecture \ref{conj:CTC+} and such that $d(\wh{\vartheta}_i)=d(\wh{\chi}_i)$ for every $(\wh{\sigma}_i,\wh{\vartheta}_i)\in\C(\wh{C}_i,\wh{V}_i)_+$ and $(\wh{\rho}_i,\wh{\chi}_i)\in\wh{\Omega}_i(\overline{(\wh{\sigma}_i,\wh{\vartheta}_i)})$. Since $\n_{\wh{A}}(\wh{V})_{\wh{C}}=\bigtimes_i\n_{\wh{A}_i}(\wh{V}_i)_{\wh{C}_i}$ and recalling the discussion on the $p$-chains of $\wh{K}$ from the previous paragraph, we obtain an $\n_{\wh{A}}(\wh{V})_{\wh{C}}$-equivariant bijection
\[\wh{\Omega}:\C\left(\wh{C},\wh{V}\right)_+/\wh{K}\to\C\left(\wh{C},\wh{V}\right)_-/\wh{K}.\]
By \cite[Theorem 5.1]{Spa17}, the map $\wh{\Omega}$ satisfies the condition on block isomorphisms of character triples from Conjecture \ref{conj:CTC+}. Moreover $d(\wh{\vartheta})=d(\wh{\chi})$ for every $(\wh{\sigma},\wh{\vartheta})\in\C(\wh{C},\wh{V})_+$ and $(\wh{\rho},\wh{\chi})\in\wh{\Omega}(\overline{(\wh{\sigma},\wh{\vartheta})})$. In particular, if we consider $a\geq 0$ such that $p^a=|W|_p$ and we define $\wh{d}:=d+a$, then $\wh{\Omega}$ restricts to an $\n_{\wh{A}}(\wh{V},W)_{\wh{C}}$-equivariant bijection
\[\wh{\Omega}_0:\C^{\wh{d}}\left(\wh{C},\wh{V}, 1_W\right)_+/\wh{K}\to\C^{\wh{d}}\left(\wh{C},\wh{V}, 1_W\right)_-/\wh{K}\]
where $\C^{\wh{d}}(\wh{C},\wh{V}, 1_W)_\pm$ is the set of pairs $(\wh{\sigma},\wh{\vartheta})\in\C^{\wh{d}}(\wh{C},\wh{V})_\pm$ such that $\wh{\vartheta}\in\irr(\wh{K}_{\wh{\sigma}})$ lies above $1_W$. By \eqref{eq:Perfect 1} it follows that the set $\C^{\wh{d}}(\wh{C},\wh{V}, 1_W)_\pm/\wh{K}$ is in bijection with $\C^d(C,V)_\pm/K$ and therefore $\wh{\Omega}_0$ induces a bijection
\[\Omega:\C^d(C,V)_+/K\to\C^d(C,V)_-/K.\]
As $\n_{\wh{A}}(W)$ induces all automorphisms of $K$ (see \cite[Corollary B.8]{Nav18}) we deduce that $\Omega$ is $\n_A(V)_C$-equivariant and it remains to prove that
\begin{equation}
\label{eq:Perfect 2}
\left(A_{\sigma,\vartheta},K_\sigma,\vartheta\right)\iso{K}\left(A_{\rho,\chi},K_\rho,\chi\right)
\end{equation}
for every $(\sigma,\vartheta)\in\C^d(C,V)_+$ and $(\rho,\chi)\in\Omega(\overline{(\sigma,\vartheta)})$. For this, let $(\wh{\sigma},\wh{\vartheta})\in\C^{\wh{d}}(\wh{C},\wh{V},1_W)_+$ and $(\wh{\rho},\wh{\chi})\in\C^{\wh{d}}(\wh{C},\wh{V},1_W)_-$ correspond to $(\sigma,\vartheta)$ and $(\rho,\chi)$ respectively and notice that $\vartheta$ and $\chi$ can be identified via $\epsilon$ with $\overline{\wh{\vartheta}}$ and $\overline{\wh{\chi}}$ respectively and where $\overline{\wh{\vartheta}}$ and $\overline{\wh{\chi}}$ are the characters of $\wh{K}_{\wh{\sigma}}/W$ and $\wh{K}_{\wh{\rho}}/W$ corresponding to $\wh{\vartheta}$ and $\wh{\chi}$ via inflation of characters. Since $\wh{\Omega}_0$ satisfies the conclusions of Conjecture \ref{conj:CTC+} and using \cite[Lemma 3.8 (b) and Corollary 4.4]{Spa17}, we deduce that
\begin{equation}
\label{eq:Perfect 3}
\left(\n_{\wh{A}}(W)_{\wh{\sigma},\wh{\vartheta}}/W,\wh{K}_{\wh{\sigma}}/W,\overline{\wh{\vartheta}}\right)\iso{\wh{K}/W}\left(\n_{\wh{A}}(W)_{\wh{\rho},\wh{\chi}},\wh{K}_{\wh{\rho}}/W,\overline{\wh{\chi}}\right).
\end{equation}
We then deduce \eqref{eq:Perfect 2} from \eqref{eq:Perfect 3} by applying \cite[Theorem 5.3]{Spa17}.
\end{proof}

Next, we consider the case where a covering group $K$ of a product of non-abelian simple groups is normally embedded in a larger group $G$. Our next result can be seen as a version of \cite[Corollary 8.2]{Spa17} compatible with $G$-block isomorphisms of character triples.

\begin{cor}
\label{cor:Perfect, lifting}
Let $G\unlhd A$ be finite groups and consider $V\leq K\unlhd A$ with $V\leq \z(G)$ a $p$-subgroup and $K$ a non-central perfect subgroup of $G$ such that $K/(\z(G)\cap K)$ is isomorphic to a direct product of non-abelian simple groups all of whose covering groups satisfy Conjecture \ref{conj:CTC} at the prime $p$. If $B$ is a block of $G$ that covers blocks of $K$ with defect groups strictly containing $V$, then there exists an $\n_A(V)_B$-equivariant bijection
\[\Omega^B_{K,V}:\C^d_K(B,V)_+/G\to\C^d_K(B,V)_-/G\]
such that
\[\left(A_{\sigma,\vartheta},G_\sigma,\vartheta\right)\iso{G}\left(A_{\rho,\chi},G_\rho,\chi\right)\]
for every $(\sigma,\vartheta)\in\C^d_K(B,V)_+$ and any $(\rho,\chi)\in\Omega_{K,V}^B(\overline{(\sigma,\vartheta)})$.
\end{cor}

\begin{proof}
We may assume without loss of generality that $B$ is $A$-invariant and hence obtain $\n_A(V)_B=G\n_A(V)_C$ for any block $C$ of $K$ covered by $B$ (see \cite[Corollary 9.3]{Nav98}). Since $K$ is perfect, Proposition \ref{prop:Perfect} implies that Conjecture \ref{conj:CTC} holds for $K$. In particular, recalling that $C$ has defect groups strictly containing $V$ by assumption, for every non-negative integer $d'$ there exists an $\n_A(V)_C$-equivariant bijection
\[\Omega^{d'}:\C^{d'}(C,V)_+/K\to\C^{d'}(C,V)_-/K\]
such that
\begin{equation}
\label{eq:Perfect, lifting 1}
\left(A_{\sigma,\varphi},K_\sigma,\varphi\right)\iso{K}\left(A_{\rho,\psi},K_\rho,\psi\right)
\end{equation}
for every $(\sigma,\varphi)\in\C^{d'}(C,V)_+$ and $(\rho,\psi)\in\Omega^{d'}(\overline{(\sigma,\varphi)})$ where $\overline{(\sigma,\varphi)}$ denotes the $K$-orbit of $(\sigma,\varphi)$. Combining the bijections $\Omega^{d'}$ for every non-negative $d'$ we obtain a bijection $\Omega:\C(C,V)_+/K\to\C(C,V)_-/K$ with similar properties and such that $d(\varphi)=d(\psi)$ for every $(\sigma,\varphi)\in\C(C,V)_+$ and $(\rho,\psi)\in\Omega(\overline{(\sigma,\varphi)})$.

Next, fix an $\n_A(V)_C$-transversal $\mathcal{T}_{1,+}$ in $\mathfrak{N}(K,V)_+$ and an $\n_A(V)_{C,\sigma}$-transversal $\mathcal{T}_{2,+}^{\sigma}$ in $\irr(C_\sigma)$ for each $\sigma\in\mathcal{T}_{1,+}$. We recall here once again that we assume, as we may, that all $p$-chains considered in our study are normal. The set
\[\mathcal{T}_+:=\left\lbrace\overline{(\sigma,\varphi)}\enspace\middle|\enspace\sigma\in\mathcal{T}_{1,+}, \varphi\in\mathcal{T}_{2,+}^{\sigma}\right\rbrace\]
is an $\n_A(V)_C$-transversal in $\C(C,V)_+/K$. Moreover, since $\Omega$ is $\n_A(V)_C$-equivariant, we can find an $\n_A(V)_C$-transversal $\mathcal{T}_{1,-}$ in $\mathfrak{N}(K,V)_-$ and, for each $\rho\in\mathcal{T}_{1,-}$, an $\n_A(V)_{C,\rho}$-transversal $\mathcal{T}_{2,-}^{\rho}$ in $\irr(C_\rho)$ such that
\[\mathcal{T}_-:=\left\lbrace\overline{(\rho,\psi)}\enspace\middle|\enspace \rho\in\mathcal{T}_{1,-}, \psi\in\mathcal{T}_{2,-}^{\rho}\right\rbrace\]
is an $\n_A(V)_C$-transversal in $\C(C,V)_-/K$ and there is a bijection
\begin{align}
\label{eq:Perfect, lifting 2}
\mathcal{T}_+&\to\mathcal{T}_-
\\
\overline{(\sigma,\varphi)}&\mapsto\Omega\left(\overline{(\sigma,\varphi)}\right).\nonumber
\end{align}

Our aim is now to find $\n_A(V)_B$-transversals $\mathcal{S}_+$ in $\C^d_K(B,V)_+/G$ and $\mathcal{S}_-$ in $\C^d_K(B,V)_-/G$ in bijection with each other and inducing $G$-block isomorphisms of character triples. Let $(\sigma,\varphi)\in\mathcal{T}_+$ and $(\rho,\psi)\in\mathcal{T}_-$ whose $K$-orbits correspond under \eqref{eq:Perfect, lifting 2}. By Clifford theory (see also \cite[Theorem B]{Kos-Spa15}) we deduce that
\begin{equation}
\label{eq:Perfect, lifting 3}
\irr^d(B_\sigma)=\coprod\limits_{\varphi\in\mathcal{T}_{2,+}^\sigma}\irr^d(B_\sigma\mid \varphi)
\end{equation}
and
\begin{equation}
\label{eq:Perfect, lifting 4}
\irr^d(B_\rho)=\coprod\limits_{\psi\in\mathcal{T}_{2,-}^\rho}\irr^d(B_\rho\mid \psi).
\end{equation}
We choose an $\n_A(V)_{\sigma,\varphi}$-transversal $\mathcal{S}^\sigma_{2,+}(\varphi)$ in $\irr^d(B_\sigma\mid \varphi)$. If $\irr^d(B_{\sigma,\varphi}\mid \varphi)$ denotes the set of Clifford correspondents over $\varphi$ of elements in $\irr^d(B_\sigma\mid \varphi)$, then the set $\mathcal{S}_{2,+}^\sigma(\varphi)_0$ of characters $\varphi_0\in\irr^{d}(B_{\sigma,\varphi}\mid\varphi)$ such that $\varphi_0^{G_{\sigma}}\in\mathcal{S}_{2,+}^\sigma(\varphi)$ is an $\n_A(V)_{\sigma,\varphi}$-transversal in $\irr^d(B_{\sigma,\varphi}\mid \varphi)$. Since \eqref{eq:Perfect, lifting 1} is satisfied for our choice of $(\sigma,\varphi)$ and $(\rho,\psi)$, using \cite[Lemma 2.9 (ii)]{Ros22} we obtain a bijection
\[f:\irr(G_{\sigma,\varphi}\mid \varphi)\to\irr(G_{\rho,\psi}\mid \psi)\]
satisfying
\[\left(A_{\sigma,\varphi,\vartheta_0},G_{\sigma,\varphi},\vartheta_0\right)\iso{KG_{\sigma,\varphi}}\left(A_{\rho,\psi,\chi_0},G_{\rho,\psi},\chi_0\right)\]
for every $\vartheta_0\in\irr(G_{\sigma,\varphi}\mid \varphi)$ and $\chi_0:=f(\vartheta_0)$. Then, \cite[Lemma 2.11]{Ros22} shows that
\begin{equation}
\label{eq:Perfect, lifting 5}
\left(A_{\sigma,\varphi,\vartheta_0},G_{\sigma,\varphi},\vartheta_0\right)\iso{G}\left(A_{\rho,\psi,\chi_0},G_{\rho,\psi},\chi_0\right).
\end{equation}
Notice that $f$ maps $\irr(B_{\sigma,\varphi}\mid \varphi)$ to $\irr(B_{\rho,\psi}\mid \psi)$ while, since $d(\varphi)=d(\psi)$ by the construction of $\Omega$, \cite[Lemma 2.9 (iii)]{Ros22} implies that $f$ preserves the defect of characters. Recalling that $f$ is a strong isomorphism of character triples (see \cite[Theorem 3.3]{Spa17}), we deduce that the image $\mathcal{S}_{2,-}^\rho(\psi)_0$ of $\mathcal{S}_{2,+}^\sigma(\varphi)_0$ under $f$ is an $\n_A(V)_{\rho,\psi}$-transversal in $\irr^d(B_{\rho,\psi}\mid \psi)$. Then, by the Clifford correspondence, the set $\mathcal{S}_{2,-}^\rho(\psi)$ of induced characters $\chi_0^{G_\rho}$ for $\chi_0\in\mathcal{S}_{2,-}^\rho(\psi)_0$ is an $\n_A(V)_{\rho,\psi}$-transversal in $\irr^d(B_{\rho}\mid \psi)$ and there exists a bijection
\begin{align*}
\mathcal{S}_{2,+}^\sigma(\varphi)&\to\mathcal{S}_{2,-}^\rho(\psi)
\\
\vartheta&\mapsto \chi
\end{align*}
such that
\begin{equation}
\label{eq:Perfect, lifting 6}
\left(A_{\sigma,\vartheta},G_{\sigma},\vartheta\right)\iso{G}\left(A_{\rho,\chi},G_{\rho},\chi\right)
\end{equation}
where \eqref{eq:Perfect, lifting 6} follows from \eqref{eq:Perfect, lifting 5} via an application of \cite[Proposition 2.8]{Ros22}. To conclude, we observe that
\[\mathcal{S}_{2,+}^\sigma:=\coprod\limits_{\varphi\in\mathcal{T}_{2,+}^\sigma}\mathcal{S}^\sigma_{2,+}(\varphi)\]
is an $\n_A(V)_{C,\sigma}$-transversal in $\irr^d(B_\sigma)$ by \eqref{eq:Perfect, lifting 3} and similarly that
\[\mathcal{S}_{2,-}^\rho:=\coprod\limits_{\psi\in\mathcal{T}_{2,-}^\rho}\mathcal{S}^\rho_{2,-}(\psi)\]
is an $\n_A(V)_{C,\rho}$-transversal in $\irr^d(B_\rho)$ by \eqref{eq:Perfect, lifting 4}. As a consequence, and recalling that $\n_A(V)_B=G\n_A(V)_C$, the sets
\[\mathcal{S}_+:=\left\lbrace\overline{(\sigma,\vartheta)}\enspace\middle|\enspace\sigma\in\mathcal{T}_{1,+}, \vartheta\in\mathcal{S}_{2,+}^\sigma\right\rbrace\]
and
\[\mathcal{S}_-:=\left\lbrace\overline{(\rho,\chi)}\enspace\middle|\enspace\rho\in\mathcal{T}_{1,-}, \chi\in\mathcal{S}_{2,-}^\rho\right\rbrace\]
are $\n_A(V)_B$-transversals in $\C^d_K(B,V)_+/G$ and $\C^d_K(B,V)_-/G$ respectively and there exists a bijection
\begin{align*}
\Omega^B_{K,V}:\mathcal{S}_+&\to\mathcal{S}_-
\\
\overline{(\sigma,\vartheta)}&\mapsto \overline{(\rho,\chi)}
\end{align*}
with $(\sigma,\vartheta)$ and $(\rho,\chi)$ satisfying \eqref{eq:Perfect, lifting 6}. We can now conclude the proof by extending $\Omega^B_{K,V}$ to a bijection between $\C^d_K(B,V)_+/G$ and $\C^d_K(B,V)_-/G$.
\end{proof}

Before proceeding further, we point out that in the proof of \cite[Theorem 8.4]{Spa17}, the conclusion of \cite[Corollary 8.2]{Spa17} is applied with respect to the group $K\z(G)$ where $K$ is the layer of $G$ and hence $K\z(G)$ is the generalised Fitting subgroup. However, notice that the latter is not necessarily perfect and hence the hypothesis of \cite[Corollary 8.2]{Spa17} is not satisfied. Nevertheless, this apparent issue can be easily circumvented by extending the result of the corollary to groups that are not necessarily perfect. While this observation is implicit in the proof of \cite[Theorem 8.4]{Spa17}, for the reader's convenience we now state it precisely in the case of our Corollary \ref{cor:Perfect, lifting}. For this, we need a couple of group theoretic considerations.

\begin{lem}
\label{lem:Central product and chains}
Let $K\unlhd N$ be finite groups with $N=K\z(N)$ and $\z(K)=K\cap \z(N)$. Assume that $\O_p(N)\leq \z(N)$ and notice that $\O_p(K)=\O_p(N)\cap K\leq \z(K)$. Denote by $\mathfrak{P}_N$ the set of $p$-subgroups of $N$ containing $\O_p(N)$ and by $\mathfrak{P}_K$ the set of $p$-subgroups of $K$ containing $\O_p(K)$. Then the map
\begin{align*}
\mathfrak{P}_N&\to\mathfrak{P}_K
\\
P&\mapsto P\cap K
\end{align*}
is a bijection with inverse $Q\mapsto \O_p(N)Q$.
\end{lem}

\begin{proof}
Since $\O_p(N)\cap K=\O_p(K)$ we deduce that the map $\alpha:\mathfrak{P}_N\to\mathfrak{P}_K$, $P\mapsto P\cap K$ is well defined. We also consider the map $\beta:\mathfrak{P}_K\to\mathfrak{P}_N$, $Q\mapsto Q\O_p(N)$. Let $P\in\mathfrak{P}_N$ and $Q\in\mathfrak{P}_K$. By elementary group theory we have $\alpha(\beta(Q))=Q\O_p(N)\cap K=Q(\O_p(N)\cap K)=Q\O_p(K)=Q$. On the other hand, noticing that $N/K\O_p(N)$ is a $p'$-group, we deduce that $P\leq K\O_p(N)$ and therefore $\beta(\alpha(P))=(P\cap K)\O_p(N)=P\cap K\O_p(N)=P$. This shows that $\alpha$ and $\beta$ are inverses of each other. 
\end{proof}

\begin{cor}
\label{cor:Central product chains}
Let $K$ and $N$ be as in Lemma \ref{lem:Central product and chains} and suppose furthermore that $K,N\unlhd G$. Then there is a bijection
\[\mathfrak{P}(N,\O_p(N))/G\to\mathfrak{P}(K,\O_p(K))/G\] 
such that $G_\sigma=G_{\rho}$ and $|\sigma|=|\rho|$ for every $\sigma\in\mathfrak{P}(N,\O_p(N))$ whose $G$-orbit corresponds to the one of $\rho\in\mathfrak{P}(K,\O_p(K))$.
\end{cor}

\begin{proof}
Let $\alpha:\mathfrak{P}_N\to\mathfrak{P}_K$ be the bijection given by Lemma \ref{lem:Central product and chains}. If $\sigma=\{D_i\}_i$ is a $p$-chain in $\mathfrak{P}(N,\O_p(N))$, then $\rho:=\{\alpha(D_i)\}_i$ is a $p$-chain in $\mathfrak{P}(K,\O_p(K))$. Using the fact that $\alpha$ is a bijection, it follows that the map $\sigma\mapsto \rho$ induces a length preserving bijection between $\mathfrak{P}(N,\O_p(N))/G$ and $\mathfrak{P}(K,\O_p(K))/G$. Moreover, since $K, N\unlhd G$, we deduce that $G_\sigma=G_\rho$.
\end{proof}

We can now extend Corollary \ref{cor:Perfect, lifting} to the case where the group $K$ (denoted below by $N$) is not perfect. This finally shows how to construct the bijection from \eqref{eq:Cancellation subset}.

\begin{theo}
\label{thm:CTC over simple groups}
Let $G\unlhd A$ be finite groups and consider $Z\leq N\unlhd A$ with $Z\leq \z(G)$ a $p$-subgroup and $N$ a non-central subgroup of $G$ such that $N/(\z(G)\cap N)$ is isomorphic to a direct product of non-abelian simple groups all of whose covering groups satisfy Conjecture \ref{conj:CTC} at the prime $p$. If $B$ is a block of $G$ that covers blocks of $N$ with defect groups strictly containing $Z$, then there exists an $\n_A(Z)_B$-equivariant bijection
\[\Omega^B_{N,Z}:\C^d_N(B,Z)_+/G\to\C^d_N(B,Z)_-/G\]
such that
\[\left(A_{\sigma,\vartheta},G_\sigma,\vartheta\right)\iso{G}\left(A_{\rho,\chi},G_\rho,\chi\right)\]
for every $(\sigma,\vartheta)\in\C^d_N(B,Z)_+$ and any $(\rho,\chi)\in\Omega_{N,Z}^B(\overline{(\sigma,\vartheta)})$.
\end{theo}

\begin{proof}
To start notice that if $Z<\O_p(N)$, then the argument used in \cite[Lemma 2.3]{Ros22} shows how to construct the required bijection $\Omega^B_{N,Z}$. Thus, we may assume without loss of generality that $Z=\O_p(N)$. Let $K:=[N,N]$, $V:=K\cap Z=\O_p(K)$ and observe that $K$ is a perfect group such that $K\z(N)=N$ (see \cite[(33.3)]{Asc86}). Moreover $\z(N)=N\cap \z(G)$, $\z(K)=K\cap \z(N)=K\cap \z(G)$ and we have $K/\z(K)\simeq N/\z(N)$. Next, assume that $B$ covers a block $b$ of $N$ which covers a block $c$ of $K$ and let $D$ be a defect group of $B$ such that $D\cap N$ and $D\cap K$ are defect groups of $b$ and $c$ respectively. By assumption $\O_p(N)=Z<D\cap N$ and hence $\O_p(K)=V<D\cap K$ by Lemma \ref{lem:Central product and chains}. This shows that the hypothesis of Corollary \ref{cor:Perfect, lifting} is satisfied with respect to the groups $V$, $K$, $G$, $A$, and the block $B$. Then, there exist a bijection
\[\Omega^B_{K,V}:\C^d_K(B,V)_+/G\to\C^d_K(B,V)_-/G\]
with the required properties. Using $\Omega^B_{K,V}$ and applying Corollary \ref{cor:Central product chains} we then obtain the required bijection
\[\Omega^B_{N,Z}:\C^d_N(B,Z)_+/G\to\C^d_N(B,Z)_-/G.\]
This concludes the proof.
\end{proof}

\subsection{Proof of Theorem \ref{thm:Main reduction}}

We can prove the main result of this paper which we now restate for the reader's convenience. Notice that Theorem \ref{thm:Main reduction} follows immediately by the following result. Recall once again that a simple group $S$ is said to be involved in $G$ if there exist subgroups $N\unlhd H\leq G$ such that $S$ is isomorphic to $H/N$.

\begin{theo}
\label{thm:Reduction for CTC}
Let $G$ be a finite group and $p$ a prime number. If Conjecture \ref{conj:CTC} holds at the prime $p$ for every covering group of every non-abelian simple group involved in $G$ and whose order is divisible by $p$, then it holds for $G$ at the prime $p$.
\end{theo}

\begin{proof}
For any finite group $H$, denote by $\mathcal{S}_p(H)$ the set of non-abelian simple groups involved in $H$ whose order is divisible by $p$. Suppose that the above assertion is false and assume that $G\unlhd A$ is a counterexample. We may assume that $G$ and $A$ have been minimized with respect to $|A:\z(G)|$. In particular, Conjecture \ref{conj:CTC} holds for every finite groups $G_1\unlhd A_1$ such that:
\begin{enumerate}
\item Conjecture \ref{conj:CTC} holds at $p$ for every covering group of every simple group in $\mathcal{S}_p(G_1)$; and
\item $|A_1:\z(G_1)|<|A:\z(G)|$.
\end{enumerate}
and hence Hypothesis \ref{hyp:Minimal counterexample} is satisfied for $G\unlhd A$. Fix a $p$-subgroup $Z\leq \z(G)$, a $p$-block $B$ of $G$ with defect groups strictly larger than $Z$ and a non-negative integer $d$ for which Conjecture \ref{conj:CTC} fails to hold. Set $N:=\F^\star(G)$ and notice that $Z\leq N\unlhd A$ is a non-central subgroup of $G$ (notice that $G$ cannot be abelian). Applying Theorem \ref{thm:Minimal counterexample}, we obtain an $\n_A(Z)_B$-equivariant bijection
\[\Lambda_{N,Z}^B:\mathcal{D}^d_N(B,Z)_+/G\to\D^d_N(B,Z)_-/G\]
inducing $G$-block isomorphisms of character triples. Next, we point out that the proof of \cite[Theorem 5.2]{Ros22} shows that $\O_p(G)\O_{p'}(G)\leq \z(G)$ whenever Hypothesis \ref{hyp:Minimal counterexample} is satisfied. In particular, under our assumptions we deduce that $\F(G)=\z(G)$ and therefore $N/(\z(G)\cap N)$ is isomorphic to a direct product of non-abelian simple groups all of which lie in $\mathcal{S}_p(G)$. We can now apply Theorem \ref{thm:CTC over simple groups} and obtain an $\n_A(Z)_B$-equivariant bijection
\[\Omega^B_{N,Z}:\C^d_N(B,Z)_+/G\to\C^d_N(B,Z)_-/G\]
inducing $G$-block isomorphisms of character triples. Finally, we define
\[\Omega\left(\overline{\left(\sigma,\vartheta\right)}\right):=\begin{cases}
\Lambda_{N,Z}^B\left(\overline{\left(\sigma,\vartheta\right)}\right), \text{ if } (\sigma,\vartheta)\in\mathcal{D}^d_N(B,Z)_+
\\
\Omega_{N,Z}^B\left(\overline{\left(\sigma,\vartheta\right)}\right), \text{ if } (\sigma,\vartheta)\in\mathcal{C}^d_N(B,Z)_+
\end{cases}\]
for every $(\sigma,\vartheta)\in\C^d(B,Z)_+$ which is the union of $\C^d_N(B,Z)_+$ and $\D^d_N(B,Z)_+$. This concludes the proof.
\end{proof}

Notice that to prove the above result we still have to go beyond Conjecture \ref{conj:CTC} and use its generalisation stated in Conjecture \ref{conj:CTC non-central Z}. However, a fully self contained reduction can be deduced from Theorem \ref{thm:Reduction for CTC} together with Theorem \ref{thm:CTC equivalent to CTC non-central}.

\begin{theo}
\label{thm:Reduction for CTC non-central}
Let $G$ be a finite group and $p$ a prime number. If Conjecture \ref{conj:CTC non-central Z} holds at the prime $p$ for every covering group of every non-abelian simple group involved in $G$ and whose order is divisible by $p$, then it holds for $G$ at the prime $p$.
\end{theo}

\begin{proof}
Observe that our hypothesis is equivalent to that of Theorem \ref{thm:Reduction for CTC} as explained in Remark \ref{rmk:CTC non-central vs CTC for quasi-simple}. In particular, we deduce that Conjecture \ref{conj:CTC} holds at the prime $p$ for every finite group $G_1$ satisfying $\mathcal{S}_p(G_1)\subseteq \mathcal{S}_p(G)$. We can now apply Theorem \ref{thm:CTC equivalent to CTC non-central} to conclude that Conjecture \ref{conj:CTC non-central Z} holds for $G$ at the prime $p$ as desired.
\end{proof}

\bibliographystyle{alpha}
\bibliography{References}

\vspace{1cm}

{\sc{Department of Mathematical Science, Loughborough University, LE$11$ $3$TU, UK}}

\textit{Email address:} \href{mailto:damiano.rossi.math@gmail.com}{damiano.rossi.math@gmail.com}

\end{document}